\documentclass[12pt]{amsart}
\usepackage[margin=2cm]{geometry}
\usepackage{amsmath}
\usepackage{amsbsy}
\usepackage[latin1]{inputenc}
\usepackage{amssymb}
\usepackage{dsfont}
\usepackage{pdfpages}
\usepackage[colorlinks=true,breaklinks=true,linkcolor=blue]{hyperref}

\def\a {{\boldsymbol a}}
\def\n {{\boldsymbol n}}

\newtheorem{theorem}{Theorem}[section]
\newtheorem{lemma}{Lemma}[section]
\newtheorem{proposition}{Proposition}[section]
\newtheorem{corollary}{Corollary}[section]

\newtheorem{remark}{Remark}[section]

\allowdisplaybreaks

\title[Approximation of a Hele-Shaw-like system]{ From a cell model with active motion to a Hele-Shaw-like system. A numerical approach}

\author[F. Guillén-González]{Francisco Guill\'en-Gonz\'alez$\dag$}
\address[$\dag$]{Dpto.~E.D.A.N. and IMUS, Universidad 
de Sevilla, Aptdo.~1160, 41080 Sevilla, Spain.  E-mail: {\tt guillen@us.es}}

\author[J. V. Gutiérrez-Santacreu]{Juan Vicente Gutiérrez-Santacreu$\ddag$}
\address[$\ddag$]{ Dpto. de Matemática Aplicada I, E. T. S. I. Informática, Universidad de Sevilla. Avda. Reina Mercedes, s/n. E-41012 Sevilla, Spain.  E-mail: {\tt juanvi@us.es}}
\thanks{This work was partially supported by Ministerio de Economía y Competitividad 
  under Spanish grant MTM2015-69875-P with the participation of FEDER}

\begin{document}
\date{\today}
\begin{abstract}
In this paper we deal with the  numerical solution of a Hele--Shaw-like system via a cell model with active motion. Convergence of approximations is established for well-posed initial data. These data are chosen in such a way the time derivate is positive at the initial time.  

The numerical method is constructed by means of a finite element procedure together with the use of a closed-nodal integration. This gives rise to an algorithm which preserves positivity whenever a right-angled triangulation is considered. As a result, uniform-in-time a priori estimates 
 are proven which allows us to pass to limit towards a solution to the Hele--Shaw problem.
\end{abstract}

\maketitle


\tableofcontents

\section{Introduction}

\subsection{The models} Tumour cells are active mechanical systems that are able to produce forces which cause random migration \cite{Betteridge_Owen_Byre_Alarcon_Maini_2006, Drasdo_Hoehme_2012,Saut_Lagaert_Colin_2014}. This movement is due to rather complicate mechanisms which occur inside cells and give rise to changes in cell shape. Another important mechanism under which cells move is pressure \cite{Bru_Albertos_Subiza_Asenjo_Broe_2003, Drasdo_Hoehme_2012, Ranft_Basan_Egeti_Joanny_Prost_Julicher_2010}  as a consequence of space competition generated by cell proliferation itself. In the setting up we take into consideration a very simplified model which incorporates the two spatial effects for describing tumour growth.      

Let $\Omega$ be a connected, open, bounded set of $\mathds{R}^d$, with $d=2$ or $3$,  and $[0,T]$ a time interval. Consider the cell model with active motion \cite{Perthame_Quiros_Tang_Vauchelet_2014} which consists in finding a tumour cell population density  $n: \overline\Omega\times [0,T]\to \mathds{R}^+$ satisfying 
\begin{equation}\label{Cell-Model}
\partial_t n - \nabla\cdot ( n \nabla p(n) )-\nu \Delta n  = n \, G(p(n))  \quad \mbox{ in $\Omega\times (0,T)$},
\end{equation}
subject to the (natural) boundary condition
\begin{equation}\label{Boundary-Condition}
\nabla n\cdot \n =0 \quad \mbox{ on $\partial\Omega\times (0,T)$},
\end{equation}
with $\n$ being the outwards unit normal vector on the boundary $\partial\Omega$, and the initial condition 
\begin{equation}\label{Initial-Condition}
n|_{t=0}=n^0\quad \mbox{ in $\Omega$}.
\end{equation}
Here $p:[0,+\infty) \to [0,+\infty)$ is defined by 
\begin{equation} \label{def_p(n)}
p=p(n):=\frac{k}{k-1} n^{k-1}\quad \forall\, n\ge 0,
\quad (k\in \mathds{N},\ k\ge 2),
\end{equation}
and $G=G(p)$ is a truncated decreasing function such that there exists  $P_{\rm max}>0$ (the homeostatic pressure) with 
\begin{equation}\label{prop:G}
G(0)>0,\quad G(p)=0 \quad \forall\, p\ge P_{\rm max}>0,\quad \hbox{and}\quad G'(p)<0\quad \forall\, p\in (0, P_{\rm max}).
\end{equation} 

In the above, $G$ stands for the decrease in the tumuor cell growth rate when space is limited; the lack of space is governed by the local pressure $p$, the parameter $P_{\rm max}$ is the maximum pressure threshold that tumour cells can exceed before entering a quiescent state, and the parameter $\nu>0$ represents the effect of including the active (random) motion of cells. 

It should be noted that the relationship of $p(n)$ given in \eqref{def_p(n)} is invertible for $n\ge 0$:  
\begin{equation} \label{def_n(p)}
n(p):=\left(\frac{k-1}{k} p \right)^{1/(k-1)} \quad \forall\, p\ge 0.
\end{equation}

In this work we assume that $\{n^0_k\}_{k\in\mathds{N}}$ is   a sequence of initial data \eqref{Initial-Condition} for \eqref{Cell-Model} such that 
\begin{equation}\label{Min-Max-Initial-Data-p}
 0\le p(n^0_k) \le P_{\rm max}\quad\mbox{ in  $\Omega$},
\end{equation}
and that there exists a limit function $n_\infty^0$ such that 
\begin{equation}\label{ID-limit}
 n^0_k \to n^0_\infty \quad\mbox{in $L^p(\Omega)$-strongly for any $p<\infty$ as $k\to \infty$.} 
\end{equation}
Consequently, defining $N_{max}(k):=n(P_{max})$  with $n(\cdot)$ being given in \eqref{def_n(p)}, we have 
\begin{equation}\label{Min-Max-Initial-Data-n}
0 \le n^0_k \le N_{\rm max}(k)  \quad\mbox{ in  $\Omega$} .
\end{equation}
from which we infer that there must exist $  N_0>0$ such that $ N_{\rm max}(k) \le N_0$. Under the above assumptions, equation \eqref{Cell-Model} generates a sequence of solutions $\{ n_{k} \}_{k\in \mathds{N}}$  which lead to a solution describing  the dynamics of  tumour growth as a free-boundary problem. To be more precise, the convergence of the solutions $\{n_k\}_{k\in\mathds{N}}$ of the active motion cell model problem \eqref{Cell-Model}-\eqref{Initial-Condition} towards a weak  solution to a Hele--Shaw-like system, as the parameter $k$ goes to infinity, was proven in \cite{Perthame_Quiros_Tang_Vauchelet_2014}. This limit system reads as follows. Find $n_\infty: \overline \Omega\times [0,T]\to \mathds{R}^+$ and $p_\infty: \overline \Omega\times [0,T]\to \mathds{R}^+$ such that  
\begin{equation}\label{Hele-Shaw}
\partial_t n_{\infty}-\Delta p_\infty-\nu \Delta n_\infty=n_\infty G(p_\infty) \quad\mbox{ in $\Omega\times(0,T)$},
\end{equation}
subject to 
\begin{equation}\label{IC-limit}
 n_{\infty}|_{t=0} =n_{\infty}^0  \quad \mbox{ in $\Omega$},
\end{equation}
\begin{equation}\label{BC-limit}
\nabla n_{\infty}\cdot \n =0\quad \mbox{ and }\quad \nabla p_\infty\cdot \n =0  \quad \mbox{ on $\partial\Omega\times (0,T)$},
\end{equation}
jointly to the complementary relation
\begin{equation}\label{Complemetary-Hele-Shaw}
p_\infty (\Delta p_\infty+G(p_\infty))=0 \quad\mbox{ in $\Omega\times(0,T)$}.
\end{equation}
The key point in establishing convergence is imposing that $\partial_t n_k(0)\ge 0$. Moreover, equation \eqref{Hele-Shaw} is equivalent to solving 
\begin{equation}\label{Hele-Shaw_II}
\partial_t n_{\infty}-\nabla\cdot (n_\infty\nabla p_\infty)-\nu \Delta n_\infty=n_\infty G(p_\infty) \quad\mbox{ in $\Omega\times(0,T)$}.
\end{equation}
This equivalence will be accomplished due to the equality $\nabla p_\infty = n_\infty \nabla p_\infty$, which comes from the equalities  $p_\infty\nabla n_\infty=0$ and $p_\infty n_\infty = p_\infty$.

In this paper, we shall be concerned with the convergence of a finite element scheme,  the time variable being continuous, for the active motion cell model problem \eqref{Cell-Model}-\eqref{Initial-Condition} towards the Hele-Shaw system \eqref{Hele-Shaw}-\eqref{Complemetary-Hele-Shaw} as the space discrete parameter $h$ goes to zero and $k$ goes to infinity. 

\subsection{Notation} We will assume the following notation throughout this paper. Let
$\mathcal{O}\subset \mathds{R}^M$, with $M\ge1$, be a Lebesgue-measurable set and let $1\le p\le\infty$. We denote by $L^p(\mathcal{O})$ the space of all Lesbegue-measurable real-valued functions, $f:\mathcal{O}\to \mathds{R} $, being $p$th-summable in $\mathcal{O}$ for $p<\infty$  or essentially bounded for $p=\infty$, and by $\|f\|_{L^p(\mathcal{O})}$ its norm. When $p=2$, the
$L^2(\mathcal{O})$ space is a Hilbert space whose inner product  is denoted by $(\cdot,\cdot)$. To shorten the notation, the norm $\|\cdot\|_{L^2(\Omega)}$  is abbreviated by $\|\cdot\|$.

Let $\alpha = (\alpha_1, \alpha_2, . . . , \alpha_M)\in \mathds{N}^M$ be a
multi-index with $|\alpha|=\alpha_1+\alpha_2+...+\alpha_M$, and let
$\partial^\alpha$ be the differential operator such that
$$\partial^\alpha=
\Big(\frac{\partial}{\partial{x_1}}\Big)^{\alpha_1}...\Big(\frac{\partial}{\partial{x_d}}\Big)^{\alpha_M}.$$

For $m\ge 0$ and $1 \le p\le \infty $, we define $W^{m,p}(\mathcal{O})$ to be the Sobolev space of all functions whose $m$ derivatives are in $L^p(\mathcal{O})$, with the norm
\begin{align*}
\|f\|_{W^{m,p}(\mathcal{O})}&=\left(\sum_{|\alpha|\le m} \|\partial^\alpha f\|^p_{L^p(\mathcal{O})}\right)^{1/p} \quad  &&\hbox{for} \ 1 \leq p < \infty, \\
\|f\|_{W^{m,p}(\mathcal{O})}&=\max_{|\alpha|\le m} \|\partial^\alpha f\|_{L^\infty(\Omega)}, \quad  && \hbox{for} \  p = \infty,
\end{align*}
where $\partial^\alpha$ is understood in the distributional sense. For p = 2, $W^{m,2}(\mathcal{O})$  will be denoted by $H^m(\mathcal{O})$. We also consider $C^\infty(\mathcal{O})$ to be the space of functions continuously differentiable any number of times, and $C^\infty_c(\mathcal{O})$ to be the subspace of $C^\infty(\mathcal{O})$ with compact support in $\mathcal{O}$. 

Spaces of Bochner-measurable functions from a time interval $[0,T]$ to a Banach space $X$ will be denoted as $L^p(0,T; X)$ with $\|f\|_{L^2(0,T; X)}=\int_0^T\|f(s)\|^p_{X} {\rm d} s$ if $1\le p<\infty$ or  $\|f\|_{L^\infty(0,T, X)}={\rm ess}\sup_{s\in(0,T)}\|f(s)\|_X<\infty$ if $p=\infty$.

\subsection{Outline} Next we sketch the remaining content of this work. In section 2 we present our finite-element spaces and some preliminary result mainly concerning interpolation operators. Furthermore, we set out our finite element numerical method, where the time variable remains continuous, and the main result of this paper.  Next is section 3 which is devoted to demonstrating the main result. Firstly, a discrete maximum principle for finite-element approximations is achieved by assuming a partition of the computational domain being made up of right-angled simplexes, and a priori estimates are also established independent of $(h,k)$ with $h$ being the space parameter associated to our finite-element space. As a result, we are able to prove positivity for the time derivative of finite-element approximations. Then better a priori energy estimates lead to obtaining compactness for passing to the limit as $(h,k)\to (0,+\infty)$. In section 4, we propose a variant of our numerical algorithm for nonobtuse triangulations which keeps with a discrete maximum principle and positive for the discrete time but whose convergence is not clear. Finally, in section 4, some numerical experiments are presented for studying the behavior of several parameters.       

\section{Spatial discretization}
\subsection{Finite-element approximation}
Herein we introduce the hypotheses that will be required along this work.
\begin{enumerate}
\item [(H1)] Let $\Omega$ be a bounded domain of $\mathds{R}^d$ ($d=2$ or $3$) with a polygonal or polyhedral Lipschitz-continuous boundary. 
\
\item[(H2)] Let $\{{\mathcal T}_{h}\}_{h>0}$ be a family of shape-regular, quasi-uniform triangulations of  $\overline{\Omega}$ made up of right-angled simplexes being triangles in two dimensions 
and tetrahedra in three dimensions, so that  $\overline \Omega=\cup_{K\in {\mathcal T}_h}K$, where $h=\max_{K\in \mathcal{T}_h} h_K$, with $h_K$ being the diameter of $K$. Further,  let 
${\mathcal N}_h = \{\a_i\}_{i\in I}$ denote the set of all the nodes of ${\mathcal T}_h$.

\item [(H3)] Conforming piecewise linear, finite element spaces associated to ${\mathcal T}_h$ are assumed for approximating $H^1(\Omega)$. Let  $\mathcal{P}_1(K)$ be the set of linear polynomials on  $K$; the space of continuous, piecewise $\mathcal{P}_1(K)$ polynomial 
functions on ${\mathcal T}_h$  is then denoted as
$$
N_h = \left\{ n_h \in {C}^0(\overline\Omega) \;:\; 
n_h|_K \in \mathcal{P}_1(K) \  \forall K \in \mathcal{T}_h \right\},
$$
whose Lagrange basis is denoted by $\{\varphi_\a\}_{\a\in{\mathcal{N}_h}}$. 
\end{enumerate}
\
We now give some auxiliary results for later use. We begin by an inverse inequality whose proof can be found in \cite[Lem. 4.5.3]{Brenner_Scott_2008} or \cite[Lem. 1.138]{Ern_Guermond_2004}. 
\begin{proposition} Under hypotheses $\rm(H1)$--$\rm(H3)$, it follows that, 
\begin{equation}\label{inv_H1toL2_nk}
\|\nabla n_h \|_{L^2(K)}\le C_{\rm inv}\, h_K^{-1} \|n_h\|_{L^2(K)}\quad\forall \, K\in\mathcal{T}_h, \quad\forall \, n_h\in N_h,
\end{equation}
where $C_{\rm inv}>0$ is a constant independent of $h$. 
\end{proposition}
Let $\mathcal{I}_h$ be the nodal interpolation operator from $ C^0(\overline\Omega)$ to $N_h$ and consider the discrete inner product 
$$
(n_h,\overline n_h)_h=\int_\Omega \mathcal{I}_{h}( n_h\,\overline n_h)=\sum_{\a\in\mathcal{N}_h}  n_h(\a)\, \overline n_h(\a) \int_\Omega \varphi_\a fo\quad\forall\, n_h,\overline n_h\in  N_h, 
$$
which induces the norm $\| n_h\|_h = \sqrt{( n_h, n_h)_h}$ defined on $N_h$. We recall the following local error estimate. See \cite[Thm. 4.4.4]{Brenner_Scott_2008} or \cite[Thm. 1.103]{Ern_Guermond_2004} for a proof.

\begin{proposition} Under hypotheses $\rm(H1)$--$\rm(H3)$, it follows that, 
\begin{equation}\label{interp_error_nodal_Linf_and_W1inf_for_K}
\|\varphi-\mathcal{I}_h\varphi\|_{L^\infty(K)}\le C_{\rm app} h_K^2 \|\nabla^2\varphi\|_{{L}^\infty(K)} \quad\forall\, K\in\mathcal{T}_h,
\quad\forall\, \varphi\in W^{2,\infty}(K),
\end{equation}
where $C_{\rm app}>0$ is independent of $h$. 
\end{proposition}
We next state the equivalence between the norms $\|\cdot\|_h$ and $\|\cdot\|$ in $N_h$ and a discrete commuter approximation property for $\mathcal{I}_h$.  
\begin{proposition} Under hypotheses $\rm(H1)$--$\rm(H3)$, it follows that, for all $n_h,\overline n_h\in N_h$, 
\begin{equation}\label{Equivalence-L2}
\| n_h\| \le \| n_h\|_h \le 5^{1/2} \| n_h\|
\end{equation}
and
\begin{equation}\label{error-L1}
\|n_h \overline n_h-\mathcal{I}_h(n_h \overline n_h)\|_{L^1(\Omega)}\le C_{\rm app} h\, \| n_h\| \, \|\nabla \overline n_h\|,
\end{equation}
where $C_{\rm app}>0$ is independent of $h$. %Moreover, for all $n_h\in N_h$ and for all $k\in\mathds{N}$, one has 
%\begin{equation}\label{error-L2}
%\|n_h^k-\mathcal{I}_{h}(n_h^k)\|\le C_{\rm app}(k)\, h \, \|\nabla (n^k_h)\|,
%\end{equation}
%\begin{equation}\label{error-H1}
%\|\nabla \mathcal{I}_{h}( n_h^k)\|\le C_{\rm stab}(k)  \|\nabla (n_h^k)\|  
%\end{equation}
%where $C_{\rm app}(k)>0$ and $C_{\rm stab}(k)>0 $ are constants independent of $h$, but possibly dependent on $k$. 
\end{proposition}

\begin{proof}  We have
$$
\|n_h\|^2=\displaystyle\sum_{\a\in\mathcal{N}_h} n_h^2(\a) \int_{\Omega}\varphi_\a^2+\sum_{\a\not=\widetilde\a\in\mathcal{N}_h} n_h(\a) n_h(\widetilde\a) \int_{\Omega}\varphi_\a \varphi_{\widetilde\a}
$$
and 
$$
\|n_h\|_h^2=\sum_{\a\in\mathcal{N}_h} n_h^2(\a)\int_\Omega\varphi_\a.
$$
Since $1=\sum_{\widetilde\a\in\mathcal{N}_h}\varphi_{\widetilde\a}$, we write
$$
\|n_h\|_h^2=\sum_{\a,\widetilde\a\in\mathcal{N}_h} n_h^2(\a)\int_\Omega\varphi_\a\varphi_{\widetilde\a}=\sum_{\a\in\mathcal{N}_h} n_h^2(\a)\int_\Omega\varphi_\a^2+ \sum_{\a\not=\widetilde\a\in\mathcal{N}_h}n_h^2(\a)\int_\Omega\varphi_\a\varphi_{\widetilde\a}.
$$
Then 
$$
\begin{array}{rcl}
\|n_h\|_h^2-\|n_h\|^2&=&\displaystyle\sum_{\a>\widetilde\a\in\mathcal{N}_h}(n_h^2(\a)+n_h^2(\widetilde\a)-2n_h(\a) n_h(\widetilde\a))\int_\Omega\varphi_\a\varphi_{\widetilde\a}
\\
&=&\displaystyle\sum_{\a>\widetilde\a\in\mathcal{N}_h}(n_h(\a)-n_h(\widetilde\a))^2\int_\Omega\varphi_\a\varphi_{\widetilde\a}\ge 0.
\end{array}
$$

From the above equality and Young's inequality, we have
$$
\begin{array}{rcl}
\|n_h\|_h^2&=&\displaystyle \|n_h\|^2+ \sum_{\a>\widetilde\a\in\mathcal{N}_h}(n_h(\a)-n_h(\widetilde\a))^2\int_\Omega\varphi_\a\varphi_{\widetilde\a}
\\
&\le&\displaystyle \|n_h\|^2 
+2
\sum_{\a>\widetilde\a\in\mathcal{N}_h}(n_h^2(\a)+n_h^2(\widetilde\a))\int_\Omega\varphi_\a\varphi_{\widetilde\a}
\\
&=&\displaystyle \|n_h\|^2 
+
2
\sum_{\a\in\mathcal{N}_h}n_h^2(\a)\int_\Omega\varphi_\a \sum_{\widetilde\a<a}\varphi_{\widetilde\a}
+ 2
\sum_{\widetilde\a\in\mathcal{N}_h}n_h^2(\widetilde\a)\int_\Omega\varphi_{\widetilde\a} \sum_{a>\widetilde a}\varphi_{\a}
\\
&\le&\displaystyle \|n_h\|^2 +4\|n_h\|^2\le 5\|n_h\|^2.
\end{array}
$$
We now prove \eqref{error-L1}. By using \eqref{interp_error_nodal_Linf_and_W1inf_for_K}, we obtain
$$
\begin{array}{rcl}
\|\mathcal{I}_h(n_h \overline n_h)-n_h \overline n_h\|_{L^1(\Omega)}
&=&\displaystyle \sum_{K\in\mathcal{T}_h} \|\mathcal{I}_h(n_h \overline n_h)-n_h \overline n_h \|_{L^\infty(K)} \int_K 1
\\
&\le&\displaystyle
C_{\rm app}\sum_{K\in\mathcal{T}_h} h_K^2 \|\nabla^2( n_h \overline n_h)\|_{{L}^\infty(K)}
\int_K 1.
\end{array}
$$
Since $n_h, \overline n_h\in \mathds{P}_1(K)$ on $K\in\mathcal{T}_h$,  we write
$$
\nabla^2( n_h \overline n_h)=2\sum_{i,j=1}^d \partial_i n_h \partial_j \overline n_h.
$$
Then, from \eqref{inv_H1toL2_nk} and on noting that $\nabla n_h,\nabla \overline n_h$ are piecewise constant on each $K\in\mathcal{T}_h$, we deduce that 
$$
\begin{array}{rcl}
\|\mathcal{I}_h(n_h \overline n_h)-n_h \overline n_h\|_{L^1(\Omega)}&\le&\displaystyle C_{\rm app}  \sum_{K\in\mathcal{T}_h} h_K^2 \|\nabla n_h\|_{{L}^\infty(K)} \|\nabla\overline n_h \|_{{L}^\infty(K)} \int_K 1
\\
&\le&\displaystyle 
C_{\rm app} \sum_{K\in\mathcal{T}_h} h_K^2 \int_K |\nabla n_h|\, |\nabla\overline n_h |
\\
&\le&\displaystyle
C_{\rm app} C_{\rm inv}\sum_{K\in\mathcal{T}_h} h_K \| n_h\|_{{L}^2(K)} \|\nabla\overline n_h \|_{{L}^2(K)}
\\
&\le& C_{\rm app} C_{\rm inv} \, h \, \|n_h\| \, \|\nabla\overline n_h\|,
\end{array}
$$
from which we conclude that \eqref{error-L1} holds.
\end{proof}
We will need to use an (average) interpolation operator into $N_h$ with the following properties. In particular we use an extension of the Scott-Zhang interpolation operator to $L^1(\Omega)$ function.  We refer to \cite{Scott_Zhang_1990,Girault_Lions_2001} and \cite{Bertoluzza_1999}. 
\begin{proposition} Under hypotheses $\rm(H1)$--$\rm(H3)$, there exists an (average)  interpolation operator  $\mathcal{Q}_h$ from $L^1(\Omega)$ to $N_h$ such that 
\begin{equation}\label{Q-sta}
\|\mathcal{Q}_h \psi\|_{W^{s,p}(\Omega)}\le C_{\rm sta} \| \psi \|_{W^{s,p}(\Omega)}\quad \mbox{for }  s=0,1\mbox{ and } 1\le p\le\infty, 
\end{equation}
\begin{equation}\label{Q-app}
\|\mathcal{Q}_h(\psi)- \psi \|_{W^{s,p} (\Omega)}\le C_{\rm app}  h^{1+m-s} \|\psi \|_{W^{m+1, p}(\Omega)} \quad \mbox{for } 0\le s\le m\le 1,
\end{equation}
and, for all $\psi\in C^\infty(\overline\Omega)$ and $\overline n_h\in N_h$,
\begin{equation}\label{Q-com}
\|\mathcal{Q}_h(\overline n_h \psi)- \overline n_h \psi \|_{W^{s,p} (\Omega)}\le C_{\rm app}  h^{1+m-s} \|\overline n_h \|_{W^{m, p}(\Omega)} \|\psi\|_{W^{m+1,\infty}}\quad \mbox{for } 0\le s\le m\le 1. 
\end{equation}
\end{proposition}

The key point in proving a discrete maximum principle is the following property which is accomplished for right-angled simplexes assumed in $\rm (H2)$. 
\begin{proposition} Under hypotheses $\rm(H1)$--$\rm(H3)$, it follows that, for any diagonal nonnegative matrix $D={\rm diag}(d_i)_{i=1}^d$ (with $d_i\ge 0$),
\begin{equation}\label{off-diagonal} 
D\nabla\varphi_{\a}\cdot\nabla\varphi_{\widetilde \a}\le0
\quad \hbox{a.e.~in $\Omega$}
\end{equation} if $\a\not = \widetilde \a$ with $\a,\widetilde\a\in {\mathcal N}_h$.
\end{proposition}
\begin{proof}  For every right-angled $d$-simplex $K\in\mathcal{T}_h$ of vertices $\{ \boldsymbol{a}_i\}_{i=0,\dots,d}$ with $\boldsymbol{a}_0$ being the vertex supporting the right angle, we denote by $F_{\boldsymbol{a}_i}$ the opposite face  to $\boldsymbol{a}_i$ and by $\boldsymbol{n}_{\boldsymbol{a}_i}$ the exterior (to the $d$-simplex $K$) unit normal vector to the face $F_{\boldsymbol{a}_i}$.  Let $\widehat K$ be the reference unit $d$-simplex with vertices $\widehat{ \boldsymbol{a}}_0 = \boldsymbol{0}$ and $\widehat{\boldsymbol{a}}_i= \boldsymbol{e}_i$, $i = 1,\cdots,d$, where $\{\boldsymbol{e}_i\}_{i = 1,\cdots,d}$ is the canonical basis of $\mathds{R}^d$. Let $F_K$ be the invertible affine mapping that maps $\widehat K$ onto $K$ defined by $F_K \widehat{\boldsymbol{x}}=\boldsymbol{a}_0+B_K \widehat{\boldsymbol{x}}$, where  $B_K\in \mathds{R}^{d\times d}$ is orthogonal.

Let $\widehat{\varphi}_{\widehat{\boldsymbol{a}}_i}(\widehat{\boldsymbol{x}})=\varphi_{\boldsymbol{a}_i}(F_K\widehat{\boldsymbol{x}})$. Then we have 
$$
\widehat \nabla\widehat\varphi_{\widehat{\boldsymbol{a}}_i}=-\frac{1}{d}\frac{|\widehat F_{\widehat{\boldsymbol{a}}_i}|}{|\widehat K|}\boldsymbol{n}_{\widehat{\boldsymbol{a}}_i}.
$$
In particular,  $\boldsymbol{n}_{\boldsymbol{\widehat a}_i}=-\boldsymbol{e}_{i}$ if $i\not=0$ and $\boldsymbol{n}_{\boldsymbol{\widehat a}_0}=[1,\cdots,1]^T$. Thus, we obtain
$$
\widehat \nabla\widehat\varphi_{\widehat{\boldsymbol{a}}_i}\cdot \widehat\nabla\widehat\varphi_{\widehat{\boldsymbol{a}}_j}=\frac{1}{d^2}\frac{|\widehat F_{\widehat{\boldsymbol{a}}_i}||\widehat F_{\widehat{\boldsymbol{a}}_j}|}{|\widehat K|^2}\boldsymbol{n}_{\boldsymbol{\widehat a}_i}\cdot\boldsymbol{n}_{\boldsymbol{\widehat a}_j}\le 0
\quad \hbox{if $i\not=j$}.
$$
Therefore, by means of the change of variable $\boldsymbol{x}=\boldsymbol{a}_0+B_K \widehat{\boldsymbol{x}}$, it follows that $\nabla\varphi_{\a_i}=B_K\widehat \nabla\widehat\varphi_{\widehat{\boldsymbol{a}}_i}$ and hence
$$
D\nabla\varphi_{\a_i}\cdot\nabla\varphi_{ \a_j}=D B_K\widehat \nabla\widehat\varphi_{\widehat{\boldsymbol{a}}_i}\cdot B_K\widehat \nabla\widehat\varphi_{\widehat{\boldsymbol{a}}_j}=\frac{1}{d^2}\frac{|\widehat F_{\widehat{\boldsymbol{a}}_i}||\widehat F_{\widehat{\boldsymbol{a}}_j}|}{|\widehat K|^2}  \boldsymbol{n}_{\boldsymbol{\widehat a}_i}^T B_K^T D B_K \boldsymbol{n}_{\boldsymbol{\widehat a}_j}\le 0
 \quad \hbox{if $i\not=j$}
$$
because, since $B_K$ is a orthogonal matrix,  the inner products defined by $D$ and $B_K^T D B_K$ preserves angles. 
\end{proof}
\begin{remark} When $D=I_d$ with $I_d$ being the $d\times d$ identity matrix, property \eqref{off-diagonal} can be proved for nonobtuse triangulations \cite{Ciarlet_Raviart_1973}. Then property \eqref{off-diagonal} can be somewhat seen a  generalization restricted for right-angled triangulations. 
\end{remark}
\
Let us now introduce the discrete Laplacian associated to the mass-lumping scalar product $(\cdot, \cdot)_h$. For any  $\Sigma_h\in N_h$, let $-\widetilde\Delta_h\Sigma_h\in N_h$ solve
\begin{equation}\label{Discrete-Laplacian-P1}
-(\widetilde\Delta_h \Sigma_h, \overline n_h)_h= (\nabla \Sigma_h, \nabla \overline n_h)  \quad  \forall\, \overline n_h\in N_h. 
\end{equation}

We end up with a compactness result \cite[Lm. 2.4]{Becker_Feng_Prohl_2008} needed in proving  the equivalence between problems \eqref{Hele-Shaw} and \eqref{Hele-Shaw_II}.
\begin{theorem}\label{th:L2H1-strong} 
Assume that $\rm(H1)$-$\rm (H3)$ holds. Let $\frac{2 d}{d +2} < \ell < \infty$. Suppose that $\{\rho_{h,k}\}_{h,k\ge0} \subset L^2(0,T; L^2(\Omega))$ is such that $\rho_{h,k}(t,\cdot)\in N_h$ for all $t\in [0,T]$ and satisfies
$$
\|\rho_{h,k}\|_{H^1(0,T; L^\ell(\Omega))}+\|\rho_{h,k}\|_{L^\infty(0,T; L^2(\Omega))\cap L^2(0,T; H^1(\Omega))}+\|\widetilde \Delta_h\rho_{h,k}\|_{L^2(0,T;L^2(\Omega))}\le C_{\rm dat}.
$$ 
Then there exist a subsequence $\{\rho_{h,k}\}_{h,k>0}$ (not relabeled) and a limit function $\rho$, such that
$$
\rho_{h,k}\to\rho\quad \mbox{in } L^2(0,T, H^1(\Omega))\mbox{-strongly as }\quad (h,k)\to (0,+\infty).
$$
\end{theorem}

Hereafter $C$ will denote a generic constant whose value may change at each occurrence. This constant may depend on the data problem and the constants $C_{\rm inv}$, $C_{\rm app}$,  $C_{\rm com }$ and $C_{\rm dat}$.

\subsection{The numerical scheme}
In order to avoid dense technical calculations, we assume for simplicity that each element $K\in \mathcal{T}_h$ has its edges lined up with the axes.    

The numerical scheme relies on a finite-element method combined with a closed-nodal integration applied to the time-derivative and pressure-migration terms. Thus our numerical  method which consists in finding $n_{h,k}\in C^1([0,T]; N_h)$ such that   
\begin{equation}\label{FEM}
\left\{
\begin{array}{l}
(\partial_t n_{h,k}, \overline n_h)_h + ( \nabla \mathcal{I}_h((n_{h,k})^k), \nabla \overline n_h )+\nu (\nabla n_{h,k},  \nabla \overline n_h)
  =(G(p (n_{h,k})) n_{h,k}, \overline n_h)_h
\quad  \forall\, \overline n\in N_h
\\
n_{h,k}(0)=n_{h,k}^0,
\end{array}
\right.
\end{equation}
with $p(n_{h,k})=\dfrac{k}{k-1} (n_{h,k})^{k-1}$. 

Equivalently, we may write $\eqref{FEM}_1$ as
\begin{equation} \label{FEM-b}
(\partial_t n_{h,k}, \overline n_h)_h + (\mathcal{D}(n_{h,k}) \nabla n_{h,k}, \nabla \overline n_h )+\nu (\nabla n_{h,k},  \nabla \overline n_h)  = (G(p(n_{h,k})) n_{h,k}, \overline n_h)_h,
\end{equation}
where $\mathcal{D}(n_{h,k})$ is a piecewise constant, $d\times d$ diagonal matrix function with respect to $\mathcal{T}_h$ defined as follows. Let $K\in\mathcal{T}_h$ with vertices $\{ \boldsymbol{a}_i\}_{i=0,\cdots,d}$ where $\boldsymbol{a}_0$ corresponds to the right angle. Then 
\begin{equation}\label{Descrite_Derivative}
[\mathcal{D}(n_{h,k})|_K]_{ii}=
\left\{
\begin{array}{cl}
  \dfrac{(n_{h,k})^k(\boldsymbol{a}_i)-(n_{h,k})^k(\boldsymbol{a}_0)}{n_{h,k}(\boldsymbol{a}_i)-n_{h,k}(\boldsymbol{a}_0)}
  & \hbox{if }  n_{h,k}(\boldsymbol{a}_i)-n_{h,k}(\boldsymbol{a}_0)\not=0,
\\
0 & \hbox{if }  n_{h,k}(\boldsymbol{a}_i)-n_{h,k}(\boldsymbol{a}_0)=0.
\end{array}
\right.
\end{equation}
By the mean value theorem, one can write
\begin{equation} \label{ident-D}
[\mathcal{D}(n_{h,k})|_K]_{ii}=k\, (n_{h,k})^{k-1}(\boldsymbol{\xi}_{i}),
\end{equation}
 where $\boldsymbol{\xi}_{i}=\alpha\boldsymbol{a}_i+(1-\alpha)\boldsymbol{a}_0$ for a certain  $\alpha\in (0,1)$.

The above choice for the sequence of $\{n_{h,k}^0\}_{h,k>0}$ is as follows. Let $\{n^0_k\}_{k\in\mathds{N}}\subset H^1(\Omega)\cap L^\infty(\Omega)$ satisfy \eqref{Min-Max-Initial-Data-p} and \eqref{Min-Max-Initial-Data-n}. Then we select $n^0_{h,k}=\mathcal{Q}_h(n^0_k)$ so that 
\begin{equation} \label{bound-n0}
0\le n^0_{h,k}(\a) \le N_{\rm max}(k)\quad \forall \hbox{$\a\in\mathcal{N}_h$},
\qquad \|\nabla n^0_{h,k}\|\le C_{stab} \|\nabla n^0_k\|,
\end{equation}
\begin{equation} \label{ID-limit-h}
n^0_{h,k}\to n^0_{k}\quad \hbox{ in $H^1(\Omega)$-strongly as $h\to 0$.}
\end{equation}
There is an additional technicality regarding the sequence of initial data that we must consider: 
\begin{enumerate}
\item[{\rm (H4)}] Assume $\{n_{h,k}\}_{h,k>0}$ to be such that  
\begin{equation}\label{Positivity-Derivative}
-(\nabla \mathcal{I}_h (n^0_{h,k})^k, \nabla \overline n_h)-\nu (\nabla n^0_{h,k} , \nabla \overline n_h)+(G(p(n^0_{h,k})) n^0_{h,k}, \overline n_h)_h\ge0\quad \forall \overline n_h\in N_h\mbox{ with } \overline n_h\ge 0.
\end{equation}
 
 \end{enumerate}
\begin{remark}
This last condition is related to imposing $\partial_t n_{h,k}(0)\ge0$ which is crucial to prove the $k\to +\infty$ limit. 
\end{remark}

The existence and uniqueness of a solution to scheme \eqref{FEM} may be readily justified by Picard's theorem. To be more precise, one may prove that there exists a time interval $[0,T_h)$ for which problem \eqref{FEM} is uniquely solvable. As a consequence of a priori energy estimates, which we shall prove in the next section, one deduces that $T_h=T$ for all $h>0$.

\subsection{Main result} We now are ready to state our main result of this paper. We shall prove that  scheme \eqref{FEM} produces a sequence of discrete solutions which satifies a priori energy bounds uniform with respect to $(h,k)$ allowing us to pass to the limit as $(h,k)\to (0,+\infty)$ towards weak solutions of the Hele--Shaw-like system \eqref{Hele-Shaw}-\eqref{Complemetary-Hele-Shaw}. 

\begin{theorem}\label{Th:Main} Assume that {\rm (H1)-(H3)} hold. Then the discrete solution $\{(n_{h,k}, p_{h,k})\}_{h,k}$ of \eqref{FEM} satisfies the following estimates, for all $\a\in {\mathcal N}_h$  and $t\in[0,T]$:
$$
0\le n_{h,k}(\a,t )\le N_{\rm max}(k),
$$
$$
0\le p(n_{h,k}(\a,t ))\le P_{\rm max},
$$
$$
\partial_t n_{h,k}(\a,t)\ge 0, \quad \partial_t p(n_{h,k}(\a, t))\ge 0.
$$
Furthermore, $\{n_{h,k},\mathcal{I}_h((n_{h,k})^k)\}_{h,k}$ converges towards weak solutions  $(n_\infty, p_\infty)$ of problem \eqref{Hele-Shaw}-\eqref{Complemetary-Hele-Shaw} in the sense that 
$$
n_{h,k} \to n_\infty \quad \mbox{ in $L^\infty(0,T; H^1(\Omega))$-weakly-$\star$ and 
 in $L^p((0,T)\times \Omega)$-strongly},
$$
and
$$
\mathcal{I}_h( (n_{h,k})^k) \to  p_\infty \quad \mbox{ in $L^\infty(0,T; H^1(\Omega))$-weakly-$\star$ and in $L^p((0,T)\times\Omega)$-strongly},
$$
for any $1<p<\infty$ provided that 
\begin{equation}\tag{\rm H5}
k\,h \to 0\quad \mbox{ as }\quad (h,k)\to (0,+\infty).
\end{equation}
\end{theorem}
\section{Proof of Theorem \ref{Th:Main}}
\subsection{A priori energy estimates}
Our goal is to prove a priori energy estimates for the discrete solution $n_{h,k}$ of \eqref{FEM} independent of $(h,k)$. 

This first lemma will be focused on proving a discrete maximum principle for $n_{h,k}$ based on the hypothesis of right-angled triangulations. Moreover, some a priori energy estimates are obtained. 
\begin{lemma} Assume that {\rm (H1)-(H3)} hold. Then the  solution $n_{h,k}$ of  scheme \eqref{FEM} satisfies
\begin{equation}\label{lm1:DMP-n}
0\le n_{h,k}(\a ,t)\le N_{\rm max}(k)\quad \forall\, \a\in {\mathcal N}_h\quad\mbox{ and }\quad \forall\, t\ge 0,
\end{equation}
and
\begin{equation}\label{lm1:Energy-FEM}
\|n_{h,k}\|_{L^\infty(0,T; L^2(\Omega))}+\|n_{h,k}\|_{L^2(0,T; H^1(\Omega))}\le C, 
%(n_{h,k})\quad \hbox{is bounded (independent of $(h,k)$) in $L^\infty(0,T; L^2(\Omega))\cap L^2(0,T; H^1(\Omega))$.}
\end{equation}
where $C>0$ is independent of $(h,k)$. 
\end{lemma}
\begin{proof} We first proceed to verify \eqref{lm1:DMP-n}. In doing so, we introduce a modification to scheme \eqref{FEM-b} which truncates the nonlinear diffusion term as follows:
\begin{equation}\label{lm3.1-FEM}
(\partial_t n_{h,k}, \overline n_h)_h + (\mathcal{D}([n_{h,k}]_T) \nabla n_{h,k}, \nabla \overline n_h )+\nu (\nabla n_{h,k},  \nabla \overline n_h)  
= (G(p([n_{h,k}]_T) ) n_{h,k}, \overline n_h)_h,
\end{equation}
where 
%$[\mathcal{D}([n_{h,k}]_T)|_K]_{ii}= \max \{0, k \,[n_{h,k}]_T^{k-1}(\boldsymbol{\xi}_{\boldsymbol{a}_i})\}$ and 
$[n_{h,k}]_T$ is the usual truncation of $n_{h,k}$ from below by $0$ and from above by $N_{max}(k)$. Again, by means of Picard's theorem, one has the existence and uniqueness of a solution $n_{h,k}$ to \eqref{lm3.1-FEM}.
 
Let $n_{h,k}^{\rm min}=\mathcal{I}_h(n_{h,k}^{-})\in N_h$ be defined as 
$$
n_{h,k}^{\rm min}=\sum_{\a\in\mathcal{N}_h} n_{h,k}^{-}(\a)\varphi_\a,
$$
where $n_{h,k}^{-}(\a)=\min\{0, n_{h,k}(\a)\}$. Analogously, one defines $n^{\rm max}_{h,k}
=\mathcal{I}_h(n_{h,k}^{+})\in N_h$ as
$$
n_{h,k}^{\rm max}=\sum_{\a\in\mathcal{N}_h} n_{h,k}^{+}(\a)\varphi_\a,$$
where $n_{h,k}^{+}(\a)=\max\{0, n_{h,k}(\a)\}$. Notice that $n_{h,k}=n_{h,k}^{\rm min}+ n_{h,k}^{\rm max}$.

On choosing $\overline n_h=n_{h,k}^{\rm min}$ in \eqref{lm3.1-FEM}, it follows that  
\begin{equation}\label{lm3.1-lab1}
\begin{array}{rcl}
\displaystyle
\frac12\frac{d}{dt}\|n_{h,k}^{\rm min}\|^2_h+(\mathcal{D}([n_{h,k}]_T)\nabla n_{h,k}, \nabla n_{h,k}^{\rm min})+\nu(\nabla n_{h,k}, \nabla n_{h,k}^{\rm min})
&=&\|G(p([n_{h,k}]_T))^{1/2} n_{h,k}^{\rm min}\|^2_h
\\
&\le& G(0) \| n_{h,k}^{\rm min}\|^2_h. 
\end{array}
\end{equation}
Next observe that 
$$
\begin{array}{rcl}
(\mathcal{D}([n_{h,k}]_T)\nabla n_{h,k}, \nabla n_{h,k}^{\rm min})&=&
(\mathcal{D}([n_{h,k}]_T)\nabla n_{h,k}^{\rm min}, \nabla n_{h,k}^{\rm min})
+(\mathcal{D}([n_{h,k}]_T)\nabla n_{h,k}^{\rm max}, \nabla n_{h,k}^{\rm min})
\\
\\
&=&
\displaystyle
\| \mathcal{D}([n_{h,k}]_T)^{1/2} \nabla n_{h,k}^{\rm min}\|^2
+ \sum_{\a\not=\tilde \a\in\mathcal{N}_h} n_{h,k}^{-}(\a) n_{h,k}^{+}(\widetilde\a) (\mathcal{D}([n_{h,k}]_T)\nabla\varphi_\a, \nabla\varphi_{\widetilde {\boldsymbol{a}}}).
\end{array}
$$
Then, using the fact that $n_{h,k}^{-}(\a) n_{h,k}^{+}(\widetilde\a)\le 0$ if $\a\not=\tilde \a$ and that $\mathcal{D}([n_{h,k}]_T)$ is a nonnegative diagonal matrix function, one deduces, from \eqref{off-diagonal}, that 
$$
\mathcal{D}([n_{h,k}]_T) \nabla \varphi_\a \cdot\nabla\varphi_{\widetilde \a} \le 0\quad \forall\, \a\not=\widetilde\a\in\mathcal{N}_{h}
$$
and thereby 
\begin{equation}\label{lm3.1-lab2}
(\mathcal{D}([n_{h,k}]_T)\nabla n_{h,k}, \nabla n_{h,k}^{\rm min})\ge
\|\mathcal{D}([n_{h,k}]_T)^{1/2} \nabla n_{h,k}^{\rm min}\|^2.
\end{equation}
Analogously, one obtains  
\begin{equation}\label{lm3.1-lab3}
\nu (\nabla n_{h,k}, \nabla n_{h,k}^{\rm min})\ge \nu \|\nabla n_{h,k}^{\rm min}\|^2,
\end{equation}
where we have used again \eqref{off-diagonal} but now for $\mathcal{D}=I_d$, with $I_d$ being the $d\times d$ unit matrix.  Inserting \eqref{lm3.1-lab2} and \eqref{lm3.1-lab3} into \eqref{lm3.1-lab1} yields 
$$
\frac12\frac{d}{dt}\| n_{h,k}^{\rm min}\|^2_h+\| \mathcal{D}([n_{h,k}]_T)^{1/2} \nabla n_{h,k}^{\rm min}\|^2+\nu \|\nabla n_{h,k}^{\rm min}\|^2
\le G(0) \| n_{h,k}^{\rm min}\|^2_h. 
$$
By Grönwall's lemma, we have $n_{h,k}^{\rm min}(t)\equiv 0$ in $\Omega$, for any $t\ge 0$, since $n_{h,k}^{\rm min}(0) \equiv 0$ in $\Omega$; thereby this implies $0\le n_{h,k}$ in \eqref{lm1:DMP-n}. For the other inequality $n_{h,k}\le N_{\rm max}(k)$ in \eqref{lm1:DMP-n}, we proceed in a similar fashion. In this case, one chooses $\overline n_h=(n_{h,k}-N_{\rm max}(k))^{\rm max}$  in \eqref{lm3.1-FEM} and takes into account  that $G(p([n_{h,k}]_T)) n_{h,k} (n_{h,k}-N_{\rm max}(k))^{\rm max}\equiv 0$ due to  $p([n_{h,k}]_T)= P_{\rm max}$ if $n_{h,k}\ge N_{max}(k)$. 

It should be noted that any solution $n_{h,k}$ of the modified scheme  \eqref{lm3.1-FEM} satisfies the discrete maximum principle \eqref{lm1:DMP-n}, and consequently $[n_{h,k}]_T\equiv n_{h,k}$; hence $n_{h,k}$ satisfies the non-truncated scheme \eqref{FEM} as well. Finally, by uniqueness of solutions for scheme \eqref{FEM}, the solution of  \eqref{FEM} takes values between $0$ and $N_{max}(k)$; that is \eqref{lm1:DMP-n}.

Now selecting $\overline n_h= n_{h,k}$ in \eqref{FEM-b} and invoking Grönwall's lemma, the following energy estimate holds,  for all $t\in [0,T]$:
\begin{equation}\label{lm1-lab2}
\frac12\|n_{h,k}(t)\|^2_h+  \int_0^T\| \mathcal{D}(n_{h,k})^{1/2} \nabla n_{h,k} \|^2+\nu\int_0^T\|\nabla n_{h,k}\|^2\le \exp(2G(0)T) \frac12 \| n_{h,k}^0\|_h^2.
\end{equation}
Then  the weak estimates \eqref{lm1:Energy-FEM} are deduced from \eqref{lm1-lab2} and \eqref{Equivalence-L2}. 
\end{proof}

A discrete maximum principle for $(n_{h,k})^{k-1}$ and $(n_{h,k})^k$ follows as a direct consequence of \eqref{lm1:DMP-n}.
\begin{corollary} There holds  
\begin{equation}\label{lm1:DMP-nk-1}
0\le (n_{h,k})^{k-1}(\a ,t)\le P_{\rm max} \quad \forall\, \a\in {\mathcal N}_h\quad\mbox{ and }\quad \forall\, t\ge 0.
\end{equation}
and
\begin{equation}\label{lm1:DMP-nk}
0\le (n_{h,k})^k(\a ,t)\le P_{\rm max} N_{\rm max}(k)\quad \forall\, \a\in {\mathcal N}_h\quad\mbox{ and }\quad \forall\, t\ge 0.
\end{equation}
\end{corollary}
\begin{proof} Assertions \eqref{lm1:DMP-nk-1} and \eqref{lm1:DMP-nk} are satisfied in view of \eqref{lm1:DMP-n} and  the bounds 
$$n^{k-1}_{h,k}(\a ,t)\le N_{max}(k)^{k-1} = \dfrac{k-1}{k} P_{max} \le P_{max} $$
and
$$n^k_{h,k}(\a ,t)\le N_{max}(k)^k = N_{max}(k)^{k-1} N_{max}(k) \le P_{max} N_{max}(k).$$
\end{proof}
\

The following lemma provides the positivity and some a priori estimates for the time derivative of $n_{h,k}$ and $(n_{h,k})^k$.

\begin{lemma} Suppose that {\rm (H1)-(H4)} hold. Then it follows that 
\begin{equation}\label{lm3.2:Positivity-derivative}
\partial_t n_{h,k}(\a,t)\ge 0, \quad \partial_t (n_{h,k}(\a,t))^k\ge 0
%\quad \hbox{in $(0,T)\times \Omega$.}
\quad\forall\,\a\in\mathcal{N}_h\quad\mbox{ and }\quad \forall\,t\in[0,T],
\end{equation}
and the a priori estimates
\begin{equation}\label{lm3.2:bound-derivative-n}
\|\partial_t n_{h,k}\|_{L^\infty(0,T; L^1(\Omega))}\le C,
\end{equation}
\begin{equation}\label{lm3.2:bound-derivative-nk}
\|\partial_t (n_{h,k})^k\|_{L^1(0,T; L^1(\Omega))}\le C,
\end{equation}
where $C>0$ is a constant independent of $(h,k)$.
%
%{\bf K: Also $\|\partial_t \mathcal{I}_h(n^k_{h,k})\|_{L^1(0,T; L^1(\Omega))}\le C$, but it is not used}
\end{lemma}
\begin{proof} Let us define $\Sigma(n_{h,k})\in N_h$ such that 
$$\Sigma(n_{h,k})=\mathcal{I}_h((n_{h,k})^k)+\nu\,n_{h,k} = \mathcal{I}_h((n_{h,k})^k+\nu\,n_{h,k}).
$$
Moreover, let $\Sigma'(n_{h,k})\in N_h$ and $\Sigma''(n_{h,k})\in N_h$ be defined as 
$$
\Sigma'(n_{h,k})=k\,\mathcal{I}_h((n_{h,k})^{k-1} ) +\nu \quad\hbox{and}\quad
\Sigma''(n_{h,k})=k(k-1)\ \mathcal{I}_h((n_{h,k})^{k-2}).
$$
%$$
%\Sigma'(n_{h,k})(\boldsymbol{a},t)=k (n_{h,k}(\boldsymbol{a},t))^{k-1}+\nu\quad \hbox{and} \quad \Sigma''(n_{h,k})=k(k-1) (n_{h,k}(\boldsymbol{a},t))^{k-2},
%$$
%for all $\boldsymbol{a}\in \mathcal{N}_h$ and $t\in[0,T]$. Note that $\Sigma'(n_{h,k})=\mathcal{I}_h(k\,n_{h,k}^{k-1} +\nu)  $ and $\Sigma''(n_{h,k})=k(k-1)\ \mathcal{I}_h(n_{h,k}^{k-2})$
Then scheme \eqref{FEM} can be rewritten as
$$
(\partial_t n_{h,k}, \overline n_h)_h+(\nabla \Sigma(n_{h,k}), \nabla\overline n_h)=(G(p(n_{h,k})) n_{h,k}, \overline n_h)_h,
$$
and equivalently, from \eqref{Discrete-Laplacian-P1}, as 
\begin{equation}\label{lm2-lab1}
(\partial_t n_{h,k}, \overline n_h)_h-(\widetilde\Delta_h \Sigma(n_{h,k}),\overline n_h)_h=(G(p(n_{h,k})) n_{h,k}, \overline n_h)_h.
\end{equation}
Now take $\overline n_h=\mathcal{I}_h(\Sigma'(n_{h,k}) \overline w_h)$, for any  $\overline w_h\in N_h$ to get
$$
(\partial_t \Sigma(n_{h,k}), \overline w_h)_h-(\Sigma'(n_{h,k})\widetilde\Delta_h \Sigma(n_{h,k}),\overline w_h)_h=(\Sigma'(n_{h,k}) G(p(n_{h,k})) n_{h,k}, \overline w_h)_h.
$$
Differentiating with respect to time and defining $w_{h,k}\in N_h$ such that, for each $\boldsymbol{a}\in\mathcal{N}_h$ and $t\in[0,T]$, 
$$w_{h,k}(\boldsymbol{a},t):=\partial_t\Sigma(n_{h,k})(\boldsymbol{a},t) 
=\Sigma'(n_{h,k})(\boldsymbol{a},t) \partial_t n_{h,k}(\boldsymbol{a},t), 
$$ 
one arrives at 
\begin{eqnarray*}
&&(\partial_t w_{h,k}, \overline w_h)_h-(\Sigma'(n_{h,k})  \widetilde\Delta_h w_{h,k}, \overline w_h)_h=(\Sigma''(n_{h,k}) \partial_t n_{h,k} \widetilde\Delta_h \Sigma(n_{h,k}), \overline w_h )_h
\\
&&+(\Sigma''(n_{h,k})\partial_t n_{h,k} G(p(n_{h,k}))n_{h,k}, \overline w_h)_h
+ k\, (\Sigma'(n_{h,k}) G'(p(n_{h,k})) (n_{h,k})^{k-1} \partial_t n_{h,k} , \overline w_h)_h
\\
&&+(\Sigma'(n_{h,k})  G(p(n_{h,k})) \partial_t n_{h,k}, \overline w_h)_h,
\end{eqnarray*}
for any $\overline w_h \in N_h$. 
Since $w_{h,k}(\boldsymbol{a},t)=\Sigma'(n_{h,k})(\boldsymbol{a},t) \partial_t n_{h,k}(\boldsymbol{a},t) $ and $\Sigma'(n_{h,k})(\boldsymbol{a},t)\ge \nu>0$, we have
$$
\partial_t n_{h,k}(\boldsymbol{a},t) = \frac{w_{h,k}(\boldsymbol{a},t) }{ \Sigma'(n_{h,k})(\boldsymbol{a},t)} \quad\forall\,\boldsymbol{a}\in\mathcal{N}_h\ \forall\,t\in[0,T].
$$
Both previous equalities  yield 
$$
(\partial_t w_{h,k}, \overline w_h)_h-(\Sigma'(n_{h,k}) \widetilde\Delta_h w_{h,k}, \overline w_h)_h=(F(n_{h,k}) w_{h,k}, \overline  w_h)_h,
$$
for any $\overline w_h \in N_h$, 
where 
$$
F(n_{h,k}):=\frac{\Sigma''(n_{h,k})}{\Sigma'(n_{h,k})}\Big\{\widetilde\Delta_h\Sigma(n_{h,k})+n_{h,k} G(p(n_{h,k})) \Big\}+k (n_{h,k})^{k-1} G'(p(n_{h,k})) +G(p(n_{h,k})) .
$$
Taking $\overline w_{h}=w_{h,k}^{\rm min}=\mathcal{I}_h(w_{h,k}^-)$ in the above variational formulation, we get  
\begin{equation} \label{ineq-a}
\frac12\frac{d}{dt}\|w_{h,k}^{\rm min}\|^2_h-(\Sigma'(n_{h,k})\widetilde\Delta_h w_{h,k}, w_{h,k}^{\rm min})_h 
%=\|F(n_{h,k})^{1/2} w_{h,k}^{\rm min}\|^2_h
\le \|F(n_{h,k})\|_{L^\infty}  \|w_{h,k}^{\rm min}\|^2_h.
\end{equation}
Since $n_{h,k}\in C^0([0,T]; N_h)$ and $N_h$ is a finite dimensional space, we have that $\|F(n_{h,k})(t)\|_{L^\infty(\Omega)}\le C_{h,k}$ for all $t\in[0,T]$, where $C_{h,k}>0$  may depend on $h$ and $k$.    
It should also be noted that $-(\Sigma'(n_{h,k})\widetilde\Delta_h w_{h,k}, w_{h,k}^{\rm min})_h\ge 0$. Indeed, choose $\overline n_h=\varphi_\a$ in \eqref{Discrete-Laplacian-P1} to obtain
$$
-(\widetilde\Delta_h w_{h,k})(\a) \int_\Omega\varphi_\a=(\nabla w_{h,k}, \nabla\varphi_\a).
$$
Then
$$
\begin{array}{rcl}
\displaystyle
-(\Sigma'(n_{h,k})\widetilde\Delta_h w_{h,k}, w_{h,k}^{\rm min})_h&=&
\displaystyle-\sum_{\a\in\mathcal{N}_h}\Sigma'(n_{h,k}(\a))(\widetilde\Delta_h w_{h,k})(\a)w^{\rm min}_{h,k}(\a)
\int_\Omega\varphi_\a
\\
&=&\displaystyle\sum_{\a\in\mathcal{N}_h}\Sigma'(n_{h,k}(\a)) (\nabla w_{h,k}, \nabla\varphi_\a) w^{\rm min}_{h,k}(\a) 
\\
&=&\displaystyle\sum_{\a\in\mathcal{N}_h}\Sigma'(n_{h,k}(\a)) (\nabla w_{h,k}^{\rm max}, \nabla\varphi_\a) w^{\rm min}_{h,k}(\a) 
\\
&&+
\displaystyle\sum_{\a\in\mathcal{N}_h}\Sigma'(n_{h,k}(\a)) (\nabla w_{h,k}^{\rm min}, \nabla\varphi_\a) w^{\rm min}_{h,k}(\a) .
\end{array}
$$
Therefore, using the fact that  $\Sigma'(n_{h,k})\ge \nu>0$, we obtain
$$
\sum_{\a\in\mathcal{N}_h}\Sigma'(n_{h,k}(\a)) (\nabla w_{h,k}^{\rm max}, \nabla\varphi_\a) w^{\rm min}_{h,k}(\a) 
=\sum_{\a\not= \tilde\a\in\mathcal{N}_h}\Sigma'(n_{h,k}(\a)) w_{h,k}^{\rm max}(\tilde\a) w^{\rm min}_{h,k}(\a) (\nabla\varphi_{\tilde\a}, \nabla\varphi_\a)\ge0
$$
and
$$
\sum_{\a\in\mathcal{N}_h}\Sigma'(n_{h,k}(\a)) (\nabla w_{h,k}^{\rm min}, \nabla\varphi_\a) w^{\rm min}_{h,k}(\a) \ge\nu\|\nabla w^{\rm min}_{h,k}\|^2\ge0.
$$
Thus, \eqref{ineq-a} leads to
$$
\frac12\frac{d}{dt}\|w_{h,k}^{\rm min}\|^2_h
\le \|F(n_{h,k})\|_{L^\infty}  \|w_{h,k}^{\rm min}\|^2_h
$$
and hence, by Grönwall's lemma, 
$$
\|w_{h,k}^{\rm min}(t)\|^2_h\le \exp(2T \|F(n_{h,k})\|_{L^\infty }) \|w_{h,k}^{\rm min}(0)\|^2 
\quad \forall\, t\in [0,T].
$$
From \eqref{Positivity-Derivative} in $\rm (H4)$, we deduce that $w_{h,k}(\boldsymbol{a}, 0)= \Sigma'(n_{h,k})(\boldsymbol{a},0) \partial_t n_{h,k}(\boldsymbol{a},0)\ge 0$ holds; therefore $w_{h,k}^{\rm min}(t)\equiv0$ since $w_{h,k}^{\rm min}(0)\equiv0$. As a result, we have that $\partial_t n_{h,k}\ge 0$ and in particular  $\partial_t (n_{h,k})^k = k (n_{h,k})^{k-1} \partial_t n_{h,k} \ge 0$. Thus, \eqref{lm3.2:Positivity-derivative} is true.   

\

Now we are going to obtain bounds  \eqref{lm3.2:bound-derivative-n} and \eqref{lm3.2:bound-derivative-nk}. For this, we take $\overline n_h=1$ in \eqref{FEM-b} and use \eqref{lm3.2:Positivity-derivative} to have
$$
\|\partial_t n_{h,k}\|_{L^1(\Omega)}= (\partial_t n_{h,k}, 1) = (\partial_t n_{h,k}, 1)_h\le G(0) \|n_{h,k}\|_{L^1(\Omega)}\le G(0) |\Omega| N_{\rm max}(k);
$$  
hence estimate \eqref{lm3.2:bound-derivative-n} holds. Furthermore, we have, by \eqref{lm1:DMP-nk} and \eqref{lm3.2:Positivity-derivative}, that
$$
\|\partial_{t} (n_{h,k})^k\|_{L^1(0,T; L^1(\Omega))}=\int_0^T\frac{d}{dt} ((n_{h,k})^k, 1)\,dt = ((n_{h,k})^k(T)-(n_{h,k})^k (0), 1) \le 2 |\Omega| N_{\rm max}(k)  P_{\rm max};
$$  
hence estimate \eqref{lm3.2:bound-derivative-nk} holds. 
\end{proof}

We are now concerned with an a priori estimate for the gradient of $n_{h,k}$ and $\mathcal{I}_h ((n_{h,k})^k)$. These estimates will play an important role in obtaining compactness results which allow us to pass to the limit  as $(h,k)\to (0,+\infty)$ from scheme \eqref{FEM}  towards weak solutions  $(n_\infty, p_\infty)$ of problem (\ref{Hele-Shaw})-\eqref{Complemetary-Hele-Shaw}. 

\begin{lemma} Suppose that $\rm(H1)$-$\rm (H4)$ are satisfied. Then there exists a constant $C>0$, independent of $h$ and $k$, such that 
\begin{equation}\label{lm3.3:unifor-grad-n}
\|\mathcal{D}(n_{h,k})^{1/2}\nabla n_{h,k} \|_{L^{\infty}(0,T; L^2(\Omega))}
+\|\nabla n_{h,k}\|_{L^{\infty}(0,T; L^2(\Omega))}\le C
\end{equation}
and
\begin{equation}\label{lm3.3:unifor-grad-nk}
\|\nabla\mathcal{I}_h ((n_{h,k})^k)\|_{L^{\infty}(0,T; L^2(\Omega))}\le C.
\end{equation}
\end{lemma}
\begin{proof}  Select $\overline n_h=n_{h,k}\in N_h$ in \eqref{FEM} to obtain 
$$
(\partial_t n_{h,k}, n_{h,k})_h +(\mathcal{D}(n_{h,k})\nabla n_{h,k}, \nabla n_{h,k})+\nu \|\nabla n_{h,k}\|^2 = \|G(p(n_{h,k}))^{1/2} n_{h,k}\|_h^2 \le G(0)\|n_{h,k}\|_h^2.
$$
From \eqref{lm1:DMP-n} and \eqref{lm3.2:Positivity-derivative}, we deduce that $(\partial_t n_{h,k}, n_{h,k})_h \ge 0$. Therefore, 
$$
\|\mathcal{D}(n_{h,k})^{1/2}\nabla n_{h,k}\|^2+\nu \|\nabla n_{h,k}\|^2 \le G(0)\|n_{h,k}\|_h^2.
$$
This last expression combined with \eqref{lm1:Energy-FEM} gives \eqref{lm3.3:unifor-grad-n}.  

Take $\overline n_h=\mathcal{I}_h((n_{h,k})^k)$ in \eqref{FEM} to have
$$
\begin{array}{l}
(\partial_t n_{h,k}, \mathcal{I}_h((n_{h,k})^k))_h + \|\nabla \mathcal{I}_h((n_{h,k})^k)\|^2
+\nu (\mathcal{D}((n_{h,k})^k)\nabla n_{h,k},  \nabla  n_{h,k})
\\
= (G(p(n_{h,k})) n_{h,k}, \mathcal{I}_h((n_{h,k})^k))_h
\le G(0) \|(n_{h,k})^{k-1}\|_{L^\infty(\Omega)} \|n_{h,k}\|_h^2
\le  G(0) P_{\rm max} \|n_{h,k}\|_h^2.
\end{array}
$$
From this, it follows that \eqref{lm3.3:unifor-grad-nk} holds from \eqref{lm1:Energy-FEM},  \eqref{lm3.2:Positivity-derivative}  and from noting that $(\mathcal{D}((n_{h,k})^k)\nabla n_{h,k},  \nabla  n_{h,k})\ge 0$ on recalling  \eqref{Descrite_Derivative}.

\end{proof}

\subsection{Passing to the limit}  From estimates \eqref{lm1:DMP-n} and \eqref{lm3.3:unifor-grad-n} jointly with \eqref{lm1:DMP-nk} and \eqref{lm3.3:unifor-grad-nk}, we have that there exist two  limit functions $(n_\infty, p_\infty )\in L^{\infty}(0,T; H^1(\Omega))^2$ and a subsequence of $\{(n_{h,k}, \mathcal{I}_h( (n_{h,k})^k)\}_{h,k}$, which we still denote in the same way, such that  the following convergences hold, as $(h,k)\to (0,\infty)$:
\begin{equation}\label{limit-w-n_inf}
n_{h,k} \to n_\infty \quad \mbox{ in $L^{\infty}(0,T; H^1(\Omega)\cap L^\infty(\Omega))$-weakly-$\star$},
\end{equation}
and
\begin{equation}\label{limit-w-p_inf}
\mathcal{I}_h( (n_{h,k})^k) \to  p_\infty \quad \mbox{ in $L^\infty(0,T; H^1(\Omega)\cap L^\infty(\Omega))$-weakly-$\star$}.
\end{equation}

Before proceeding to  pass to the limit, we need to obtain some strong convergences via an Aubin-Lions campactness lemma \cite{Simon_1987}. From (\ref{lm1:DMP-n}), (\ref{lm3.2:bound-derivative-n}) and (\ref{lm3.3:unifor-grad-n}), we have that there exists a subsequence (not relabeled)  such that, as $(h,k)\to (0,\infty)$,
\begin{equation}\label{strong-conver-n}
n_{h,k} \to n_\infty\quad \mbox{in } L^p(\Omega\times(0,T))\mbox{-strongly, }  \forall\, p<\infty,
\end{equation}
and 
\begin{equation}\label{strong-conver-cont-n}
n_{h,k} \to n_\infty\quad \mbox{in } C^0([0,T];L^q(\Omega))\mbox{-strongly, } \forall\, q<2^*,
\end{equation}
where $2^*$ stands for  the conjugate exponent of $2$ defined by $1/2^*=1/2-1/d$. Analogously,  from (\ref{lm1:DMP-nk}), (\ref{lm3.2:bound-derivative-nk}), and (\ref{lm3.3:unifor-grad-nk}), we have 
\begin{equation}\label{strong-conver-nk}
\mathcal{I}_h((n_{h,k})^k)\to p_\infty\quad \mbox{in } L^p(\Omega\times (0,T))\mbox{-strongly, } \forall\, p<\infty.  
\end{equation}

As a result, we also have the  strong convergence of $p(n_{h,k})$ towards $p_\infty$, but under  hypothesis $(\rm H5)$ in Theorem \ref{Th:Main}.
\begin{lemma} Assuming hypotheses $\rm (H1)$-$\rm (H5)$, 
 it follows that, as $(h,k)\to (0,\infty)$,
\begin{equation}\label{strong-conver-p}
p(n_{h,k}) \to p_\infty \quad \mbox{in } L^p((0,T)\times \Omega)\mbox{-strongly for any } p<\infty.  
\end{equation}
Moreover, 
\begin{equation} \label{eq-p-inf-n-inf}
p_\infty n_\infty \equiv p_\infty \quad \hbox{a.e.~in $(0,T)\times \Omega$.}
\end{equation}
\end{lemma}
\begin{proof} For each element $K\in \mathcal{T}_h$ with vertices $\{\boldsymbol{a}_0,\cdots\boldsymbol{a}_d\}$,  we associate once and for all a vertex $a_K$ of K. Thus we define a piecewise constant function $\mathcal{P}_h(n^k_{h,k})(\boldsymbol{x})= n^k_{h,k}(\boldsymbol{a}_K)$ for all $\boldsymbol{x}\in K$, which satisfies  
$$
\mathcal{P}_h(n^k_{h,k})(\boldsymbol{x})-n^k_{h,k}(\boldsymbol{x})= \nabla (n^k_{h,k}(\boldsymbol{\xi}_{\boldsymbol{a}_K}) )\cdot(\boldsymbol{a}_K-\boldsymbol{x})
=k\, n^{k-1}_{h,k}(\boldsymbol{\xi}_{\boldsymbol{a}_K})
\nabla n_{h,k}|_K \cdot(\boldsymbol{a}_K-\boldsymbol{x})
$$
where $\boldsymbol{\xi}_{\boldsymbol{a}_K}=\lambda \boldsymbol{a}_K+(1-\lambda)\boldsymbol{x}$ with $\lambda\in (0,1)$.  Then we have, by \eqref{lm1:DMP-nk-1} and \eqref{lm3.3:unifor-grad-n}, that 
$$
\|\mathcal{P}_h(n^k_{h,k})-n^k_{h,k}\|_{L^\infty(0,T;L^2(\Omega))}\le C\, k\, h \, \|n_{h,k}^{k-1}\|_{L^\infty(0,T;L^\infty(\Omega))} \|\nabla n_{h,k}\|_{L^\infty(0,T;L^2(\Omega))}\le C \, k\, h\, P_{\rm max}. 
$$
The above argument also shows by replacing $n^k_{h,k}$ by $\mathcal{I}_h(n^k_{h,k})$ and using  \eqref{lm3.3:unifor-grad-nk} that   
$$
\|\mathcal{I}_h(n^k_{h,k})-\mathcal{P}_h(n^k_{h,k})\|_{L^\infty(0,T;L^2(\Omega))}\le C\, h \, \|\nabla \mathcal{I}_h(n^k_{h,k})\|_{L^\infty(0,T;L^2(\Omega))} \le C\, h.
$$
Thus, by (\ref{strong-conver-nk}) and $\rm(H5)$, we deduce, the following convergence, as $(h,k)\to (0,\infty)$:
\begin{equation}\label{lm3.4-lab1}
n_{h,k}^k \to p_\infty\quad \mbox{in } L^p((0,T)\times \Omega)\mbox{-strongly } \forall p<\infty.
\end{equation} 

In view of \eqref{strong-conver-n} and  \eqref{lm3.4-lab1}, there is a subsequence (not relabeled) of $\{(n_{h,k}, n_{h,k}^k)\}_{h,k}$ such that, as $(h,k)\to (0,\infty)$:
$$
(n_{h,k}(\boldsymbol{x},t), n_{h,k}^k(\boldsymbol{x},t))
\to (n_{\infty}(\boldsymbol{x},t), p_{\infty}(\boldsymbol{x},t))
\quad \hbox{a.e.~$(\boldsymbol{x},t)\in \Omega\times(0,T)$.}
$$
Thus, defining
$$
\widetilde p_\infty(\boldsymbol{x},t)=
\left\{
\begin{array}{ll}
\dfrac{p_\infty(\boldsymbol{x},t)}{n_\infty(\boldsymbol{x},t)}  &   \hbox{if $n_\infty(\boldsymbol{x},t)\not=0$,}   \\
 0 &    \hbox{otherwise,}
\end{array}
\right.
$$ 
it follows that, as $(h,k)\to (0,\infty)$,
$$
p(n_{h,k}(\boldsymbol{x},t))=\frac{k}{k-1}\frac{n_{h,k}^k(\boldsymbol{x},t)}{n_{h,k}(\boldsymbol{x},t)}\to \widetilde p_\infty(\boldsymbol{x},t) 
\quad \hbox{a.e.~$(\boldsymbol{x},t)\in \Omega\times(0,T)$;}
$$
furthermore,
$$
p_\infty(\boldsymbol{x},t)\leftarrow\frac{k}{k-1} n^{k}_{h,k}(\boldsymbol{x},t)=\Big(1-\frac{1}{k} \Big)^\frac{1}{k-1} p(n_{h,k}(\boldsymbol{x},t))^{\frac{k}{k-1}}\to\widetilde p_\infty(\boldsymbol{x},t). 
$$
Thus, $p_\infty\equiv\widetilde p_\infty$ a.e.~$(\boldsymbol{x},t)\in \Omega\times(0,T)$ and, in particular, one has equality \eqref{eq-p-inf-n-inf} and the pointwise convergence
$$
p(n_{h,k}(\boldsymbol{x},t))\to p_{\infty}(\boldsymbol{x},t)\quad \hbox{a.e.~$(\boldsymbol{x},t)\in \Omega\times(0,T)$.}
$$
Finally, \eqref{strong-conver-p} is deduced from the dominated convergence theorem since  $p(n_{h,k})$ is bounded in $L^\infty(\Omega\times (0,T))$.
\end{proof}
\subsubsection{Convergence towards \eqref{Hele-Shaw}}
We are now ready to pass to the limit in scheme \eqref{FEM} as $(h,k)\to (0,\infty)$. Let $\overline n\in C_c^\infty (\Omega)$ and $\phi\in C_c^\infty(0,T)$. Consider $\overline n_h=\mathcal{Q}_h(\overline n)$ in \eqref{FEM}, multiply by $\phi$ and integrate on (0,T) to get
$$
\begin{array}{l}
\displaystyle
-\int_0^T (n_{h,k}, \mathcal{Q}_h(\overline n))_h \phi'(t ) {\rm dt} +  \int_0 ^T( \nabla \mathcal{I}_h(n^k_{h,k}), \nabla \mathcal{Q}_h(\overline n)) \phi(t ) {\rm dt}
\\
\displaystyle
+\nu \int_0^T(\nabla n_{h,k},  \nabla \mathcal{Q}_h(\overline n)) \phi(t ) {\rm dt}
 = \int_0^T(G(p (n_{h,k})) n_{h,k}, \mathcal{Q}_h(\overline n))_h \phi(t ) {\rm dt}. 
\end{array}
$$
We briefly outline the main steps of the passage to the limit since the arguments are quite classical. We write
$$
\int_0^T (n_{h,k}, \mathcal{Q}_h(\overline n))_h \phi'(t ) {\rm dt}= \int_0^T (n_{h,k}, \mathcal{Q}_h(\overline n)) \phi'(t ) {\rm dt}
 +\int_0^T \big[ (n_{h,k}, \mathcal{Q}_h(\overline n))_h-(n_{h,k}, \mathcal{Q}_h(\overline n))\big] \phi'(t ) {\rm dt}.
$$
It is an easy matter to show, from \eqref{Q-app} and \eqref{strong-conver-n}, that
$$
\int_0^T (n_{h,k}, \mathcal{Q}_h(\overline n)) \phi'(t ) {\rm dt} \to \int_0^T (n_\infty, \overline n) \phi'(t ) {\rm dt},
$$
and, from \eqref{error-L1} and \eqref{Q-sta}, that
$$
\int_0^T \big[ (n_{h,k}, \mathcal{Q}_h(\overline n))_h-(n_{h,k}, \mathcal{Q}_h(\overline n))\big] \phi'(t ) {\rm dt}\to 0.
$$
Therefore,
$$
\int_0^T (n_{h,k}, \mathcal{Q}_h(\overline n))_h \phi'(t ) {\rm dt}
 \to \int_0^T (n_\infty, \overline n) \phi'(t ) {\rm dt}.
$$
 Analogously, we obtain   
$$
\int_0^T(G(p (n_{h,k})) n_{h,k}, \mathcal{Q}_h(\overline n))_h \phi(t ) {\rm dt} \to  \int_0^T(G(p_{\infty}) n_\infty, \overline n)\phi(t ) {\rm dt}
$$
from \eqref{Q-app},   \eqref{strong-conver-n} and \eqref{strong-conver-p}. The diffusion terms are treated as follows. In view of \eqref{Q-app}, \eqref{limit-w-n_inf} and  \eqref{limit-w-p_inf}, it is easy to check that
$$
\int_0 ^T( \nabla \mathcal{I}_h(n^k_{h,k}), \nabla \mathcal{Q}_h(\overline n)) \phi(t ) {\rm dt}\to \int_0 ^T( \nabla p_\infty, \nabla \overline n) \phi(t ) {\rm dt}
$$
and
$$
\nu \int_0^T(\nabla n_{h,k},  \nabla \mathcal{Q}_h(\overline n)) \phi(t ) {\rm dt}\to \nu \int_0^T(\nabla n_\infty,  \nabla \overline n) \phi(t ) {\rm dt}. 
$$
 We have thus proved that \eqref{Hele-Shaw} holds in the distributional sense.

\subsubsection{Initial condition \eqref{IC-limit}} The initial condition \eqref{IC-limit} can be recovered from \eqref{strong-conver-cont-n}, which gives $n_{h,k}|_{t=0}\to n_\infty |_{t=0}$ in $L^q(\Omega)$, for $1\le q< 2^*$, and from \eqref{ID-limit} and \eqref{ID-limit-h}, which give $n^0_{k,h}\to n_\infty^0$ in $L^p(\Omega)$, for $1\le p<\infty$ as $(h,k)\to (0,+\infty)$.

\subsubsection{Equivalence between \eqref{Hele-Shaw} and \eqref{Hele-Shaw_II}}
In order to see the equivalence between \eqref{Hele-Shaw} and \eqref{Hele-Shaw_II} we must prove that $\nabla p_\infty \equiv n_\infty\nabla p_\infty$ which will be  obtained by proving $p_\infty \nabla n_\infty \equiv 0$ and using the equality in \eqref{eq-p-inf-n-inf}.  Indeed, for each $\boldsymbol{x}\in K$, we decompose $p(n_{h,k}(\boldsymbol{x}) ) \partial_{\boldsymbol{x}_i}n_{h,k}(\boldsymbol{x})$ by using the intermediate vector $\boldsymbol{\xi}_i$ given in \eqref{ident-D} into  
$$
\begin{array}{l}
\displaystyle
p(n_{h,k}(\boldsymbol{x}) )\partial_{\boldsymbol{x}_i}n_{h,k}(\boldsymbol{x})=
\displaystyle
\frac{k}{k-1} n^{k-1}_{h,k}(\boldsymbol{\xi}_i) \partial_{\boldsymbol{x}_i}n_{h,k}(\boldsymbol{x})
+\frac{k}{k-1} (n^{k-1}_{h,k}(\boldsymbol{x})-n^{k-1}_{h,k}(\boldsymbol{\xi}_i)) \partial_{\boldsymbol{x}_i}n_{h,k}(\boldsymbol{x})
\\
\qquad =
\displaystyle
\frac{\sqrt{k}}{k-1} n^{\frac{k-1}{2}}_{h,k}(\boldsymbol{\xi}_i) \sqrt{k}\, n^{\frac{k-1}{2}}_{h,k}(\boldsymbol{\xi}_i) \partial_{\boldsymbol{x}_i} n_{h,k}(\boldsymbol{x})
+ k (\boldsymbol{x}-\boldsymbol{\xi}_i) n^{k-2}_{h,k}(\boldsymbol{\eta}_i) (\partial_{\boldsymbol{x}_i}n_{h,k}(\boldsymbol{x}))^2,
\end{array}
$$  
where we have utilized the mean value theorem in the last term for $\boldsymbol{\eta_i}=\alpha \boldsymbol{\xi}_i+(1-\alpha) \boldsymbol{x}$ with $\alpha\in (0,1)$ and that $ \partial_{\boldsymbol{x}_i} n_{h,k}(\boldsymbol{x})$ is constant on $K$. Thus, by virtue of 
\eqref{ident-D}, we find 
$$
\begin{array}{rcl}
\|p(n_{h,k})\partial_{\boldsymbol{x}_i}n_{h,k}\|_{L^1(K)}&\le& \dfrac{\sqrt{k}}{k-1} \| n^{\frac{k-1}{2}}_{h,k}(\boldsymbol{\xi}_i) \sqrt{k}\, n^{\frac{k-1}{2}}_{h,k}(\boldsymbol{\xi}_i) \partial_{\boldsymbol{x}_i} n_{h,k}\|_{L^1(K)}
\\
&&+ k \, h\, \| n^{k-2}_{h,k}(\boldsymbol{\eta}_i) (\partial_{\boldsymbol{x}_i}n_{h,k}(\boldsymbol{x}))^2\|_{L^1(K)}
\\
&\le&\dfrac{\sqrt{k}}{k-1} \sqrt{P_{\rm max}}  \| \mathcal{D}(n^k_{h,k})^{1/2}\nabla n_{h,k}\|_{L^2(K)}
\\
&&+ C k \, h\,  P_{\rm max}\|\nabla n_{h,k} \|_{L^2(K)}^2,
\end{array}
$$
where we have used
 $n^{k-2}_{h,k}(\boldsymbol{\eta}_i)\le N_{\rm max}(k)^{k-2}=(\frac{k}{k-1} P_{\rm max})^{\frac{k-2}{k-1}}\to P_{\rm max}$ as $k\to +\infty$ in the last line.

Summing over $K\in\mathcal{T}_h$, noting \eqref{lm3.3:unifor-grad-n} and recalling the constraint $h\,k\to 0$ given in (H5), we conclude that 
$$
p(n_{h,k})\nabla n_{h,k}\to \boldsymbol{0}\mbox{ in } L^\infty(0,T; L^1(\Omega))\mbox{-strongly as } (h,k)\to (0,\infty).
$$  
We further know, by \eqref{limit-w-n_inf} and  \eqref{strong-conver-p}, that  
$$
p(n_{h,k}) \nabla n_{h,k} \to p_\infty\nabla n_\infty  \mbox{ as } (h,k)\to (0,\infty),
$$
and hence $p_\infty \nabla n_\infty \equiv 0$ a.e.~in $\Omega\times(0,T)$.

\subsubsection{Convergence towards the complementary relation \eqref{Complemetary-Hele-Shaw}}
To finish the proof of Theorem~\ref{Th:Main}, it remains to prove that \eqref{Complemetary-Hele-Shaw} holds in the distributional sense. In doing so, we will start by proving that
\begin{equation}\label{eq:compl-positive}
0\le \int_0^T(G(p_\infty) n_\infty, p_\infty\psi)-(\nabla(p_\infty+\nu n_\infty),\nabla (p_\infty\psi)) {\rm d}s
\end{equation}
and 
\begin{equation}\label{eq:compl-negative}
0\ge \int_0^T(G(p_\infty) n_\infty, p_\infty\psi)-(\nabla(p_\infty+\nu n_\infty),\nabla (p_\infty\psi)) {\rm d}s
\end{equation}
hold for all $\psi\in C_c^\infty(\overline\Omega\times [0,T])$ with $\psi\ge 0$. 

$\bullet$ To begin with, we prove that \eqref{eq:compl-positive} is true. We use \eqref{lm2-lab1} to write
$$
\partial_t n_{h,k}-\widetilde\Delta_h \Sigma(n_{h,k})=\mathcal{I}_h(G(p(n_{h,k})) n_{h,k}).
$$
Let $\rho_\varepsilon=\rho_\varepsilon(t)$ be a time regularizing kernel with compact support of length $\varepsilon>0$. Then,  extending $n_{h,k}$ by zero  outside $[0,T]$, we have 
\begin{equation}\label{eq:regularized}
\partial_t n_{h,k}*\rho_\varepsilon-\widetilde\Delta_h ( \Sigma(n_{h,k})*\rho_\varepsilon)=\mathcal{I}_h((G(p(n_{h,k})) n_{h,k})*\rho_\varepsilon),
\end{equation}
where we have used the equalities 
$\widetilde\Delta_h ( \Sigma(n_{h,k})*\rho_\varepsilon) = \widetilde\Delta_h ( \Sigma(n_{h,k}))*\rho_\varepsilon$ and $\mathcal{I}_h((G(p(n_{h,k})) n_{h,k})*\rho_\varepsilon)=\mathcal{I}_h(G(p(n_{h,k})) n_{h,k})*\rho_\varepsilon$ owing to the separation between spatial and temporal variables.

Since $\partial_t n_{h,k}*\rho_\varepsilon$ and  $(G(p(n_{h,k})) n_{h,k})*\rho_\varepsilon$ are uniformly bounded in $L^p(\Omega\times (0,T))$ for  $1\le p \le\infty$ with respect to $(h,k)$ for each fixed $\varepsilon$, we also have that 
$$
 -\widetilde\Delta_h ( \Sigma(n_{h,k})*\rho_\varepsilon)\quad \hbox{is  bounded in $L^p(\Omega\times (0, T))$.}
$$ as well. In virtue of Theorem \ref{th:L2H1-strong} and the above bounds combined with \eqref{strong-conver-n} and \eqref{strong-conver-nk}, we infer the following convergence, as $(h,k)\to (0,\infty)$:
\begin{equation}\label{strong-conver-grad-rho}
\nabla (\Sigma(n_{h,k})*\rho_\varepsilon) \to \nabla ((p_\infty+\nu n_\infty)*\rho_\varepsilon)
\quad\mbox{in $L^2(\Omega\times (0, T))$-strongly.} 
\end{equation}
 
On testing  \eqref{eq:regularized} against $\mathcal{Q}_h(\mathcal{I}_h(n_{h,k}^k)\psi)$ with $\psi\in C_c^\infty(\overline\Omega\times [0,T])$ such that $\psi\ge0$, it follows that 
\begin{equation}\label{eq:regularized-test}
\begin{array}{rcl}
\displaystyle
\int_0^T (\partial_t n_{h,k}*\rho_\varepsilon,\mathcal{Q}_h(\mathcal{I}_h(n_{h,k}^k)\psi) )_h &=& \displaystyle \int_0^T  ((G(p(n_{h,k})) n_{h,k})*\rho_\varepsilon, \mathcal{Q}_h(\mathcal{I}_h(n_{h,k}^k)\psi))_h
\\
&-& \displaystyle \int_0^T  (\nabla(\Sigma(n_{h,k})*\rho_\varepsilon),\nabla \mathcal{Q}_h(\mathcal{I}_h(n_{h,k}^k)\psi)).
\end{array}
\end{equation}
Since  $(\partial_t n_{h,k}*\rho_\varepsilon,\mathcal{Q}_h(\mathcal{I}_h(n_{h,k}^k)\psi) )_h\ge0$, we obtain
\begin{equation} \label{seq-ineq}
0\le \int_0^T((G(p(n_{h,k})) n_{h,k})*\rho_\varepsilon, \mathcal{Q}_h(\mathcal{I}_h(n_{h,k}^k)\psi))_h
-\int_0^T (\nabla(\Sigma(n_{h,k})*\rho_\varepsilon),\nabla \mathcal{Q}_h(\mathcal{I}_h(n_{h,k}^k)\psi)).
\end{equation}
Taking the limit as $(h,k)\to (0,\infty)$ yields
\begin{equation}\label{compl-lab1}
\int_0^T ((G(p(n_{h,k})) n_{h,k})*\rho_\varepsilon, \mathcal{Q}_h(\mathcal{I}_h(n_{h,k}^k)\psi))_h {\rm d}t
\to \int_0^T ((G(p_\infty) n_\infty)*\rho_\varepsilon, p_\infty\psi){\rm d}t
\end{equation}
and 
\begin{equation}\label{comp-lab2}
\int_0^T (\nabla(\Sigma(n_{h,k})*\rho_\varepsilon),\nabla \mathcal{Q}_h(\mathcal{I}_h(n_{h,k}^k)\psi)){\rm d}t
\to \int_0^T(\nabla((p_\infty+\nu n_\infty)*\rho_\varepsilon),\nabla (p_\infty\psi)) {\rm d}t.
\end{equation}

In order to prove \eqref{compl-lab1}, we use the decomposition $(u_h,v_h)_h=(u_h,v_h)+(\mathcal{I}_h (u_hv_h) - u_h v_h, 1)$ for  $u_h=(G(p(n_{h,k})) n_{h,k})*\rho_\varepsilon$ and $v_h=\mathcal{Q}_h(\mathcal{I}_h(n_{h,k}^k)\psi)$ to  write  
\begin{align*}
%\nonumber
&\int_0^T ((G(p(n_{h,k})) n_{h,k})*\rho_\varepsilon, \mathcal{Q}_h(\mathcal{I}_h(n_{h,k}^k)\psi))_h= 
\int_0^T ((G(p(n_{h,k})) n_{h,k})*\rho_\varepsilon, \mathcal{Q}_h(\mathcal{I}_h(n_{h,k}^k)\psi))
\\
%\nonumber
&\qquad
+\int_0^T (\mathcal{I}_h((G(p(n_{h,k})))n_{h,k})*\rho_\varepsilon \,
\mathcal{Q}_h(\mathcal{I}_h(n_{h,k}^k)\psi)-(G(p(n_{h,k})) n_{h,k})*\rho_\varepsilon  \, \mathcal{Q}_h(\mathcal{I}_h(n_{h,k}^k)\psi), 1).
\end{align*}
Then, it follows from \eqref{Q-com}, \eqref{strong-conver-n} and \eqref{strong-conver-p} that the first term converges to $\int_0^T ((G(p_\infty) n_\infty)*\rho_\varepsilon, p_\infty\psi){\rm d}t,$ and, on noting that $$\|\nabla \mathcal{Q}_h(\mathcal{I}_h(n_{h,k}^k)\psi)\|\le  C \|\nabla \mathcal{I}_h(n_{h,k}^k)\| \|\psi\|_{L^\infty}+ C \|\mathcal{I}_h(n_{h,k}^k)\|_{L^\infty}\|\nabla\psi\|$$ 
from \eqref{Q-sta}, and on recalling  \eqref{error-L1} and \eqref{lm3.3:unifor-grad-nk}, the second term converges to zero;  thereby \eqref{compl-lab1} holds.

In order to prove \eqref{comp-lab2}, we write    
\begin{align*}
\int_0^T (\nabla(\Sigma(n_{h,k})*\rho_\varepsilon),\nabla \mathcal{Q}_h(\mathcal{I}_h(n_{h,k}^k)\psi))&=
\int_0^T(\nabla(\Sigma(n_{h,k})*\rho_\varepsilon),\nabla (\mathcal{I}_h(n_{h,k}^k)\psi))
\\
&- \int_0^T(\nabla(\Sigma(n_{h,k})*\rho_\varepsilon),\nabla (\mathcal{I}_h(n_{h,k}^k)\psi-\mathcal{Q}_h(\mathcal{I}_h(n_{h,k}^k)\psi))).
\end{align*}
Then, it follows from \eqref{strong-conver-grad-rho}, \eqref{limit-w-p_inf} and \eqref{strong-conver-nk} that the first term converges to $\int_0^T (\nabla((p_\infty+\nu n_\infty)*\rho_\varepsilon),\nabla (p_\infty\psi ) ) {\rm d}t$, and on noting that $$
\|\nabla (\mathcal{I}_h(n_{h,k}^k)\psi-\mathcal{Q}_h(\mathcal{I}_h(n_{h,k}^k)\psi))\|\le h \|\nabla \mathcal{I}_h(n_{h,k}^k) \| \|\psi\|_{W^{2,\infty}(\Omega)}$$  from \eqref{Q-com}  and on recalling \eqref{lm3.3:unifor-grad-nk}, the second term converges to zero; thereby \eqref{comp-lab2} holds.

Thus, by applying the previous convergences \eqref{compl-lab1} and \label{comp-lab2} to \eqref{seq-ineq}, we arrive at  
$$
0\le \int_0^T  (G(p_\infty) n_\infty*\rho_\varepsilon, p_\infty\psi)-(\nabla((p_\infty+\nu n_\infty)*\rho_\varepsilon),\nabla (p_\infty\psi)) {\rm d}t,
$$
and finally \eqref{eq:compl-positive} holds by taking the limit as $\varepsilon\to 0$.

$\bullet$ We proceed to prove \eqref{eq:compl-negative}.  Write the first term on the right-hand side of \eqref{eq:regularized-test} as 
\begin{equation}\label{compl-lab3}
\begin{array}{rcl}
\displaystyle\int_0^T
(\partial_t n_{h,k}*\rho_\varepsilon,\mathcal{Q}_h(\mathcal{I}_h(n_{h,k}^k)\psi ))_h&=& \displaystyle\int_0^T(\partial_t n_{h,k}*\rho_\varepsilon,\mathcal{I}_h(n_{h,k}^k)\psi )_h
\\
&+& \displaystyle\int_0^T (\partial_t n_{h,k}*\rho_\varepsilon,\mathcal{Q}_h(\mathcal{I}_h(n_{h,k}^k)\psi)-\mathcal{I}_h(n_{h,k}^k)\psi)_h.
\end{array}
\end{equation}
These two terms are handled as follows. For the second term of \eqref{compl-lab3}, we have, by   \eqref{Equivalence-L2},  \eqref{Q-com} and \eqref{lm3.2:bound-derivative-n}, that
$$
\int_0^T (\partial_t n_{h,k}*\rho_\varepsilon,\mathcal{Q}_h(\mathcal{I}_h(n_{h,k}^k)\psi)-\mathcal{I}_h(n_{h,k}^k)\psi)_h {\rm d}s\to 0\quad \mbox{as } (h,k)\to (0,\infty).
$$
For the first term of \eqref{compl-lab3}, we have that, for each $\boldsymbol{a}\in\mathcal{N}_h$,
$$
\begin{array}{rcl}
(\partial_t n_{h,k}(\boldsymbol{a},t)*\rho_\varepsilon)\, n^k_{h,k}(\boldsymbol{a},t)&=&
\displaystyle
n^k_{h,k}(\boldsymbol{a},t)\int_\mathds{R} \partial_t n_{h,k}(\boldsymbol{a},s)\rho_\varepsilon(t-s) {\rm d}s
\\
&=&
\displaystyle
\int_\mathds{R} n^k_{h,k}(\boldsymbol{a},s) \partial_t n_{h,k}(\boldsymbol{a},s)\rho_\varepsilon(t-s) {\rm d}s
\\
&+&
\displaystyle
\int_\mathds{R}  (n^k_{h,k}(\boldsymbol{a},t)-n^k_{h,k}(\boldsymbol{a},s)) \partial_tn_{h,k}(\boldsymbol{a},s)\rho_\varepsilon(t-s) {\rm d}s.
\end{array}
$$
On integrating by parts in time and using \eqref{lm1:DMP-n} and  \eqref{lm1:DMP-nk}, we obtain 
$$
\int_\mathds{R} n^k_{h,k}(\boldsymbol{a},s) \partial_t n_{h,k}(\boldsymbol{a},s)\rho_\varepsilon(t-s) {\rm d}s
=\frac{1}{k+1}\int_\mathds{R} n^{k+1}_{h,k}(\boldsymbol{a},s)  \partial_t\rho_\varepsilon(t-s) {\rm d}s\to 0
$$
as $(h,k)\to (0,\infty)$. Furthermore, for $s > t$, we have that
$n^k_{h,k}(\boldsymbol{a},t) - n^k_{h,k}(\boldsymbol{a},s)\le 0$ owing to \eqref{lm3.2:Positivity-derivative}. Then, if we choose $\mbox{supp} (\rho_\varepsilon)\subset (-\varepsilon,0)$, then 
$$
\int_\mathds{R}  (n^k_{h,k}(\boldsymbol{a},t)-n^k_{h,k}(\boldsymbol{a},s)) \partial_tn_{h,k}(\boldsymbol{a},s)\rho_\varepsilon(t-s) {\rm d}s\le 0. 
$$
Letting first $(h,k)\to (0,\infty)$ in \eqref{compl-lab3} and then $\varepsilon\to 0$, we obtain \eqref{eq:compl-negative} by repeating the arguments that led to \eqref{eq:compl-positive}.

As a result of \eqref{eq:compl-positive} and \eqref{eq:compl-negative}, we note that 
\begin{equation} \label{variat-eq}
\int_0^T(G(p_\infty) n_\infty, p_\infty\psi)-(\nabla(p_\infty+\nu n_\infty),\nabla (p_\infty\psi)) {\rm d}s=0
\end{equation} 
 is satisfied for all $\psi\in C_c^\infty(\overline\Omega\times [0,T])$ with $\psi\ge0$, and therefore it also holds for all $\psi\in C_c^\infty(\overline\Omega\times [0,T])$. 

From the fact that $p_\infty\nabla n_\infty=0$ and $p_\infty\ge 0$ a.e.~in $\Omega\times(0,T)$, we also deduce that $\nabla p_\infty\cdot\nabla n_\infty=0$ a.e.~in $\Omega\times(0,T)$. As a consequence, the above variational equation \eqref{variat-eq} is equivalent to
$$
\int_0^T(G(p_\infty) n_\infty, p_\infty\psi)-(\nabla p_\infty,\nabla (p_\infty\psi)) {\rm d}s=0
$$
which, taking into account \eqref{eq-p-inf-n-inf}, implies  \eqref{Complemetary-Hele-Shaw} in the distributional sense. 

\section{An algorithm on unstructured meshes} 
In order to avoid using structured meshes, we propose the following scheme. Find $n_{h,k}\in C^1([0,T]; N_h)$ such that   
\begin{equation}\label{FEM-II}
\left\{
\begin{array}{l}
(\partial_t n_{h,k}, \overline n_h)_h +k ( (n_{h,k})^{k-1} \nabla n_{h,k}, \nabla \overline n_h )+\nu (\nabla n_{h,k},  \nabla \overline n_h)  =(G(p(n_{h,k})) n_{h,k}, \overline n_h)_h
\  \forall\, \overline n_h\in N_h,
\\
n_{h,k}(0)=n_{h,k}^0.
\end{array}
\right.
\end{equation}
Equivalently, we may write $\eqref{FEM-II}_1$ as
$$
(\partial_t n_{h,k}, \overline n_h)_h + ( n_{h,k} \nabla p(n_{h,k}), \nabla \overline n_h )+\nu (\nabla n_{h,k},  \nabla \overline n_h)  = (G(p(n_{h,k})) n_{h,k}, \overline n_h)_h.
$$

Here the finite-element space $N_h$ is constructed over a family of triangulations $\{{\mathcal T}_{h}\}_{h>0}$ of $\overline \Omega$ being shape-regular, quasi-uniform and with acute angles. This acuteness property implies \eqref{off-diagonal} for the particular case where  $D$ is the $d\times d$ identity matrix \cite{Ciarlet_Raviart_1973}. We summarize the properties of scheme \eqref{FEM-II} in the following theorem. 
\begin{theorem} Suppose that $(\rm H1)$-$(\rm H4)$ are satisfied. Then scheme \eqref{FEM-II} satisfies the following properties. For all $\a\in {\mathcal N}_h$ and $t\ge 0$, we have:
$$
0\le n_{h,k}(\a ,t)\le N_{\rm max}(k)
$$
$$
0\le n^k_{h,k}(\a ,t)\le P_{\rm max} N_{\rm max}(k),
$$
$$
\partial_t n_{h,k}(\a, t)\ge 0, \quad \partial_t n^k_{h,k}(\a,t)\ge 0,
$$
and the a priori estimates:
$$
\|n_{h,k}\|_{L^\infty(0,T; L^2(\Omega))\cap L^2(0,T; H^1(\Omega))}\le C,
$$
$$
\|\partial_t n_{h,k}\|_{L^\infty(0,T; L^1(\Omega))}+\|\partial_t n^k_{h,k}\|_{L^1(0,T; L^1(\Omega))}\le C,
$$
with $C>0$ being a constant independent of $(h,k)$.
\end{theorem}
\begin{proof} Full details of the proof are left to the interested reader since it follows {\it mutatis mutandis} the same arguments as for scheme \eqref{FEM}. 
\end{proof}
\begin{corollary} Under hypotheses $\rm(H1)$-$\rm(H4)$, it follows that

\begin{equation}\label{est:Grad_II}
\sum_{K\in\mathcal{T}_h}\left(\int_{K_{>}} |\partial_{\boldsymbol{x}_i} n^{k}_{h,k}(\boldsymbol{x})|^2+ \int_{K_{<}} |\partial_{{\boldsymbol x}_i} \mathcal{I}_h n^{k}_{h,k}(\boldsymbol{x})|^2 \right) {\rm d}\boldsymbol{x}\le C, 
\end{equation}
where
$$
K_{>}=\left\{ \boldsymbol{x}\in K : \frac{n^{k-1}_{k,h}(\boldsymbol{\xi}_i)}{n^{k-1}_{h,k}(\boldsymbol{x})}>1 \right\}
$$
and 
$$
K_{<}=\left\{ \boldsymbol{x}\in K : \frac{n^{k-1}_{k,h}(\boldsymbol{\xi}_i)}{n^{k-1}_{h,k},(\boldsymbol{x})}<1 \right\},
$$
with $C>0$ being a constant independent of $(h,k)$.
\end{corollary}
\begin{proof} Choose $\bar n_h=\mathcal{I}_h(n^k_{h,k})$ to get 
\begin{equation}\label{co4.1-lab1}
(\partial_t n_{h,k}, \mathcal{I}_h(n^k_{h,k}) )_h +((n_{h,k})^{k-1} \nabla n_{h,k}, \nabla \mathcal{I}_h(n^k_{h,k}) )+\nu (\nabla n_{h,k},  \nabla \mathcal{I}_h(n^k_{h,k}))  =(G(p(n_{h,k})) n_{h,k}, \mathcal{I}_h(n^k_{h,k}))_h.
\end{equation}
It follows immediately from  \eqref{lm1:DMP-nk} and \eqref{lm3.2:bound-derivative-n} that  
\begin{equation}\label{co4.1-lab2}
(\partial_t n_{h,k}, \mathcal{I}_h(n^k_{h,k}) )_h\ge  0,
\end{equation}
and from  \eqref{Descrite_Derivative} that
\begin{equation}\label{co4.1-lab3}
\nu (\nabla n_{h,k},  \nabla \mathcal{I}_h(n^k_{h,k}))=\nu (\mathcal{D}(n_{h,k})\nabla n_{h,k},\nabla n_{h,k})\ge 0.
\end{equation}
Combining \eqref{co4.1-lab1}-\eqref{co4.1-lab3} yields on noting \eqref{lm1:DMP-n} and \eqref{lm1:DMP-nk} that
$$
((n_{h,k})^{k-1} \nabla n_{h,k}, \nabla \mathcal{I}_h(n^k_{h,k}) )\le G(0) |\Omega| N_{\rm max}(k)^2 P_{\rm max}. 
$$ 
Finally, we invoke again \eqref{Descrite_Derivative} and recall \eqref{ident-D} to set 
$$
\begin{array}{rcl}
k((n_{h,k})^{k-1} \nabla n_{h,k}, \nabla \mathcal{I}_h(n^k_{h,k}) )&=&(\nabla n_{h,k}^k, \nabla \mathcal{I}_h(n^k_{h,k}) )
\\
&=&\displaystyle
\sum_{K\in\mathcal{T}_h}\int_{K} \left(\frac{n^{k-1}_{k,h}(\boldsymbol{\xi}_i)}{n^{k-1}_{h,k}(\boldsymbol{x})} |\partial_{\boldsymbol{x}_i} n^{k}_{h,k}(\boldsymbol{x})|^2+\frac{n^{k-1}_{h,k}(\boldsymbol{x})}{n^{k-1}_{k,h}(\boldsymbol{\xi}_i)} |\partial_{{\boldsymbol x}_i} \mathcal{I}_h n^{k}_{h,k}(\boldsymbol{x})|^2 \right) {\rm d}\boldsymbol{x}.
\end{array} 
$$
This completes the proof via the definitions of $K_<$ and $K_>$.
\end{proof}
\begin{remark} Unfortunately, convergence for  scheme \eqref{FEM-II} is not clear because estimate \eqref{est:Grad_II} does not provide enough control over the gradient of $\{n_{h,k}^k\}_{h,k}$ or $\{\mathcal{I}_h n_{h,k}^k\}_{h,k}$ in order  to obtain compactness and therefore to pass to the limit as  $( k, h )\to (0,+\infty)$.
\end{remark}
\section{Numerical simulation}
\subsection{Temporal integration}
It is assumed here for simplicity that we have a uniform partition of $[0,T]$ into $M$ pieces, with time step size $\tau= T/M$ and the time values $(t_m = m \tau)_{m=0}^M$. To simplify the notation let us denote $\delta_t n^{m+1}=\dfrac{n^{m+1}- n^m}{\tau}$.

First we present a first-order time integration for scheme (\ref{FEM-II}).
\begin{center}
\noindent\fbox{
\begin{minipage}{0.8\textwidth}
\textbf{Algorithm 1}: Linear semi-implicit  time-stepping scheme
\end{minipage}
}
\noindent\fbox{
\begin{minipage}{0.8\textwidth}
\textbf {Step $(m+1)$}: Given $n^{m}_{h,k}\in N_{h}$, find $n^{m+1}_{h,k}\in N_{h}$ solving the algebraic linear system
\begin{equation}\label{FEM-Euler}
\left\{
\begin{array}{rcl}
\displaystyle
(\delta_t n_{h,k}^{m+1}, \overline n_h)_h+  k( (n^{m}_{h,k})^{k-1}\nabla n^{m+1}_{h,k}, \nabla\overline n_h) &&
\\
\displaystyle
+\nu (\nabla n^{m+1}_{h,k}, \nabla\overline n_h) &=&(G(p(n^{m}_{h,k})) n_{h,k}^m, \overline n)_h, 
\end{array}
\right.
\end{equation}
for all $\overline n_h \in N_h $. 
\end{minipage}
}
\end{center}

\subsection{Computational experiments} In this section, we present several numerical experiments to test the algorithm presented herein. To do this, we consider the evolution of problem \eqref{Cell-Model}-\eqref{Initial-Condition} with 
$$
n_0(x,y)=\alpha \, e^{-(x^2+y^2)}
$$
on the computational domain $\Omega=(-10,10)\times(-10, 10)$ with $\alpha>0$. 

In the numerical setting, we construct a structured triangulation partitioning the edges of square into $100$ subintervals, corresponding  with the mesh size $h=0.12582843$ and the time step size is $\tau=10^{-5}$. The choice of the time step $\tau$ is such that it helps to mitigate the possibly numerical deviation of the $d$-simplexes $K\in\mathcal{T}_h$ from the right-angled structure. The resulting matrix is strictly diagonally dominant.     

Our intention is to illustrate the behavior exhibited by the solution to problem \eqref{Cell-Model}-\eqref{Initial-Condition} when the diffusion coefficient $\nu$, the parameter $k$ and the homeostatic pressure $P_{\rm max}$ vary. 

We will set $\alpha$ and $P_{\rm max}$ to be $1$ and $k$ to be $100$ if not stated otherwise. Moreover, we consider 
$$
G(p)=\frac{200}{\pi }\arctan (4(P_{\rm max} - p)_{+}),
$$

\subsubsection{Analysis of the effect of $\alpha$ (contraction/dilation coefficient of the initial datum)} In this test we choose $\alpha=0.5$ and $1$.  We are interested in comparing the evolution of the density $n$ and the pressure $p(n)$ when the maximum of the initial density takes different values. In particular, we have for $\alpha=0.5$ that the maximum value of $n_0$ occurs only at the point $(0,0)$ and is $0.5$, hence    $N_{\rm max}(k)$ remains below of $1$. We thus observe that the maximum increases without modifying essentially the exponential shape of the initial datum $n_0$ until reaches $N_{\rm max}(k)=1$. Once the density takes the value $1$ at $t=0.01583$ the measure of points at which the density reaches the maximum grows radially around $(0,0)$ due to the fact that the pressure starts increasing and pushes forward the tumor cells. Then the exponential structure of the initial datum $n_0$ becomes a traveling wave shape which moves outwards as $t$ increases.  This behavior causes that the evolution of the interface is delayed concerning the case $\alpha=1$ as shown in Figures \ref{fig1:alpha} and \ref{fig2:alpha} since the maximum value $1$ is reached from the beginning.

Figure \ref{fig3:alpha} represents the difference between the density and the pressure at times $t=0.1$, $0.2$, $0.3$ and $0.4$, and indicates that the pressure is responsible for the advance of the tumor cells which is deduced from the annulus shape of the difference. 
%%%%%% 
 %\end{document} 
 %%%%%%%%
\begin{figure}[h!]
\centering
\includegraphics[scale=0.12]{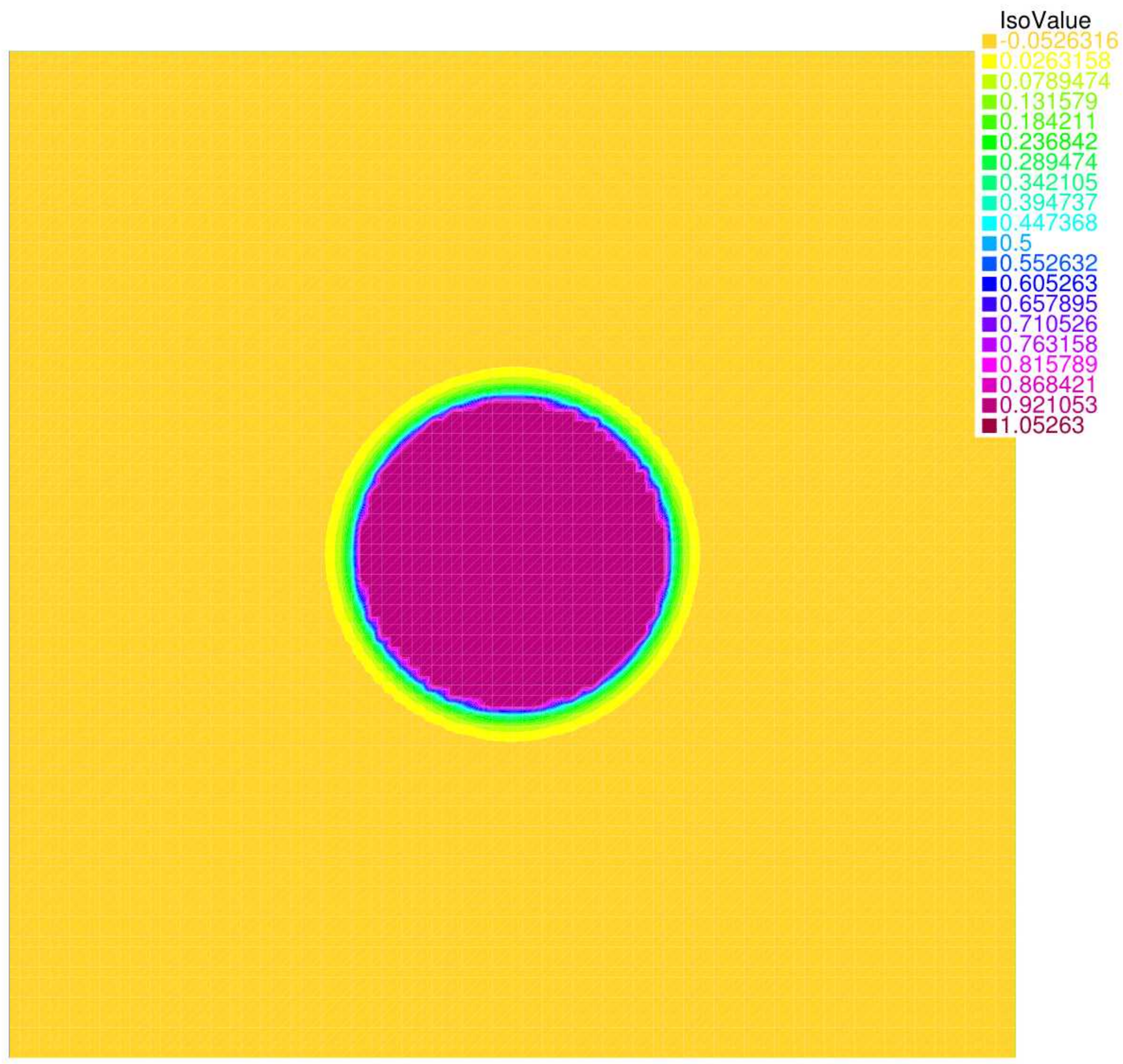}
\includegraphics[scale=0.12]{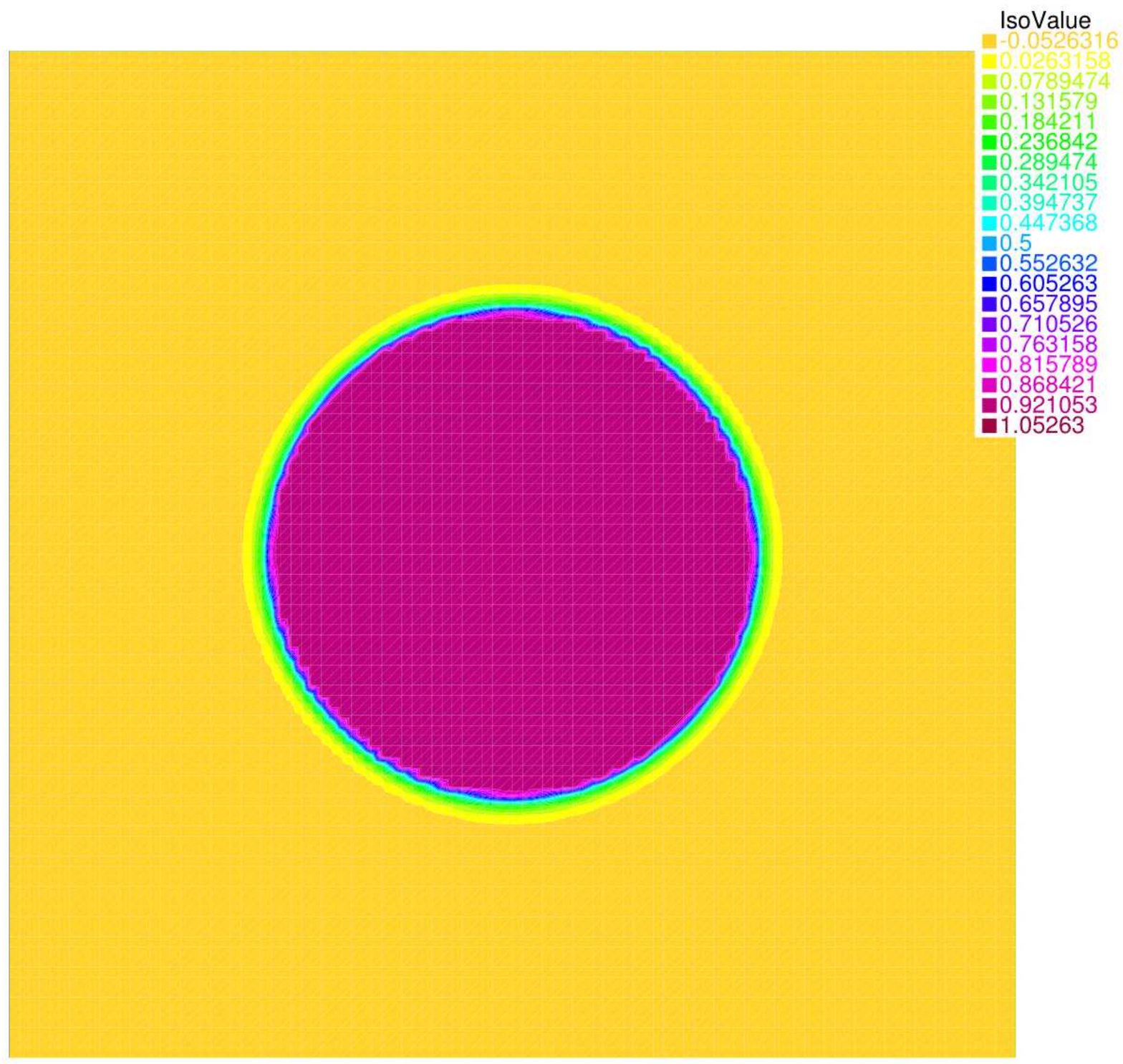}
\includegraphics[scale=0.12]{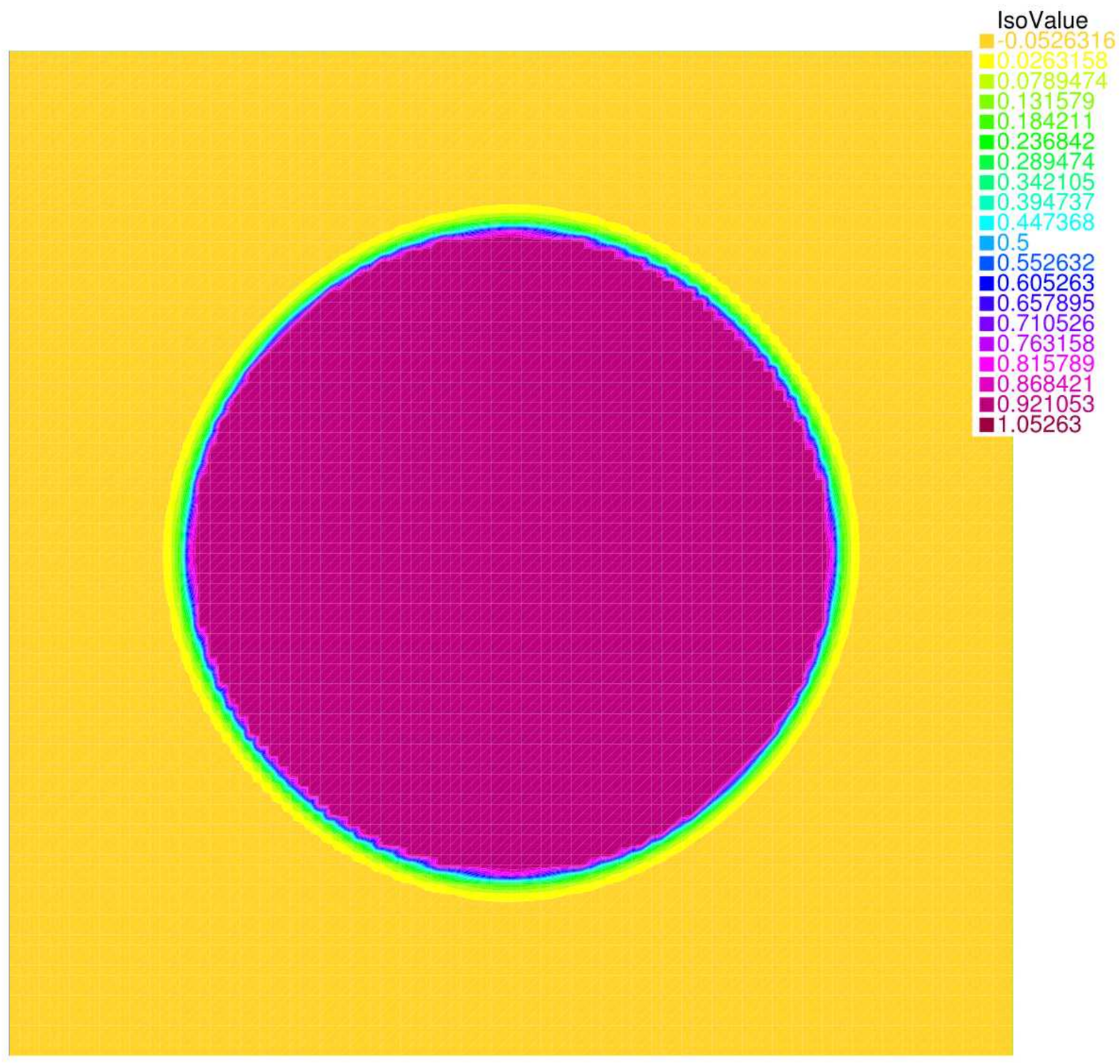}
\includegraphics[scale=0.12]{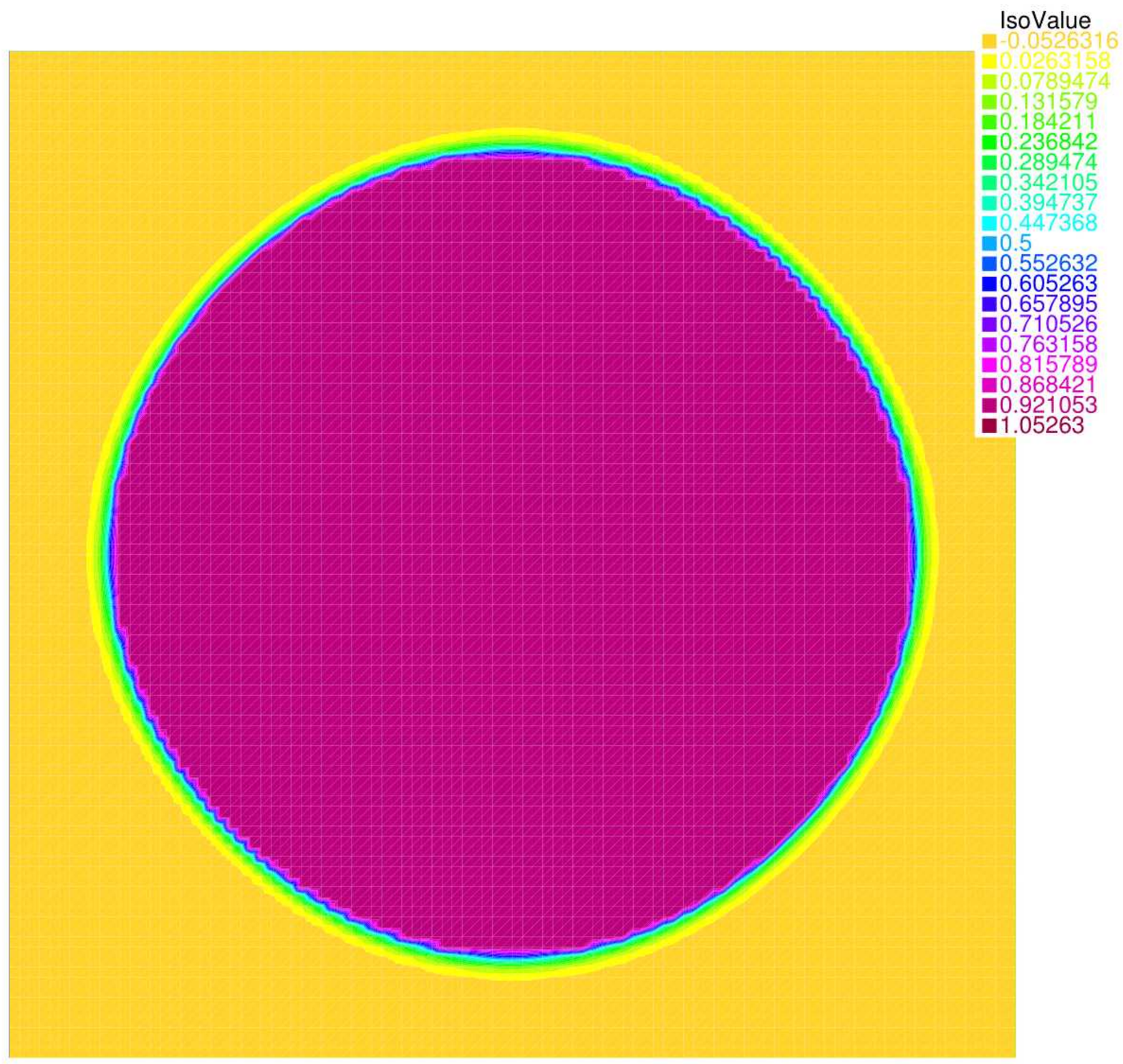}
\\
\includegraphics[scale=0.12]{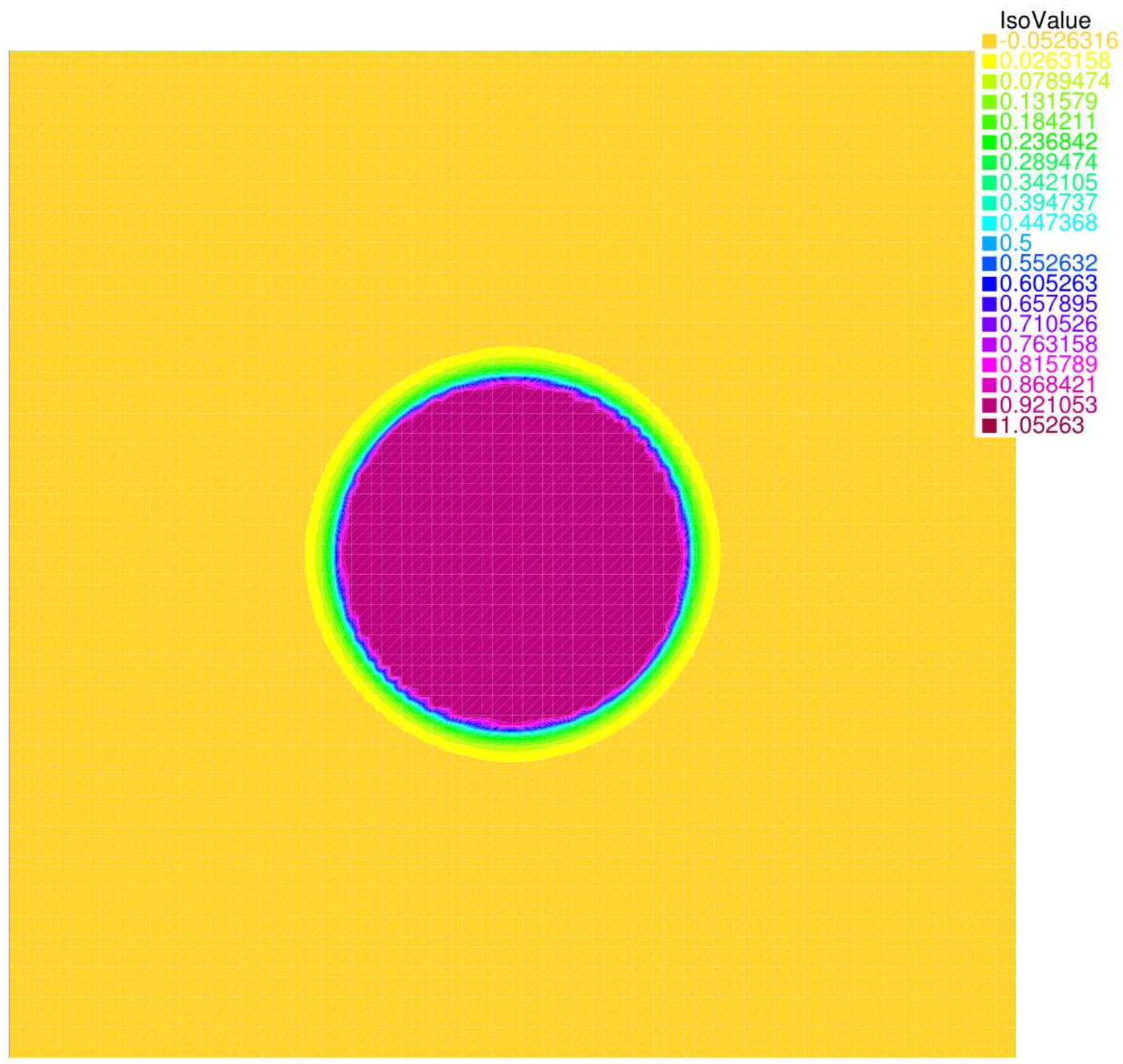}
\includegraphics[scale=0.12]{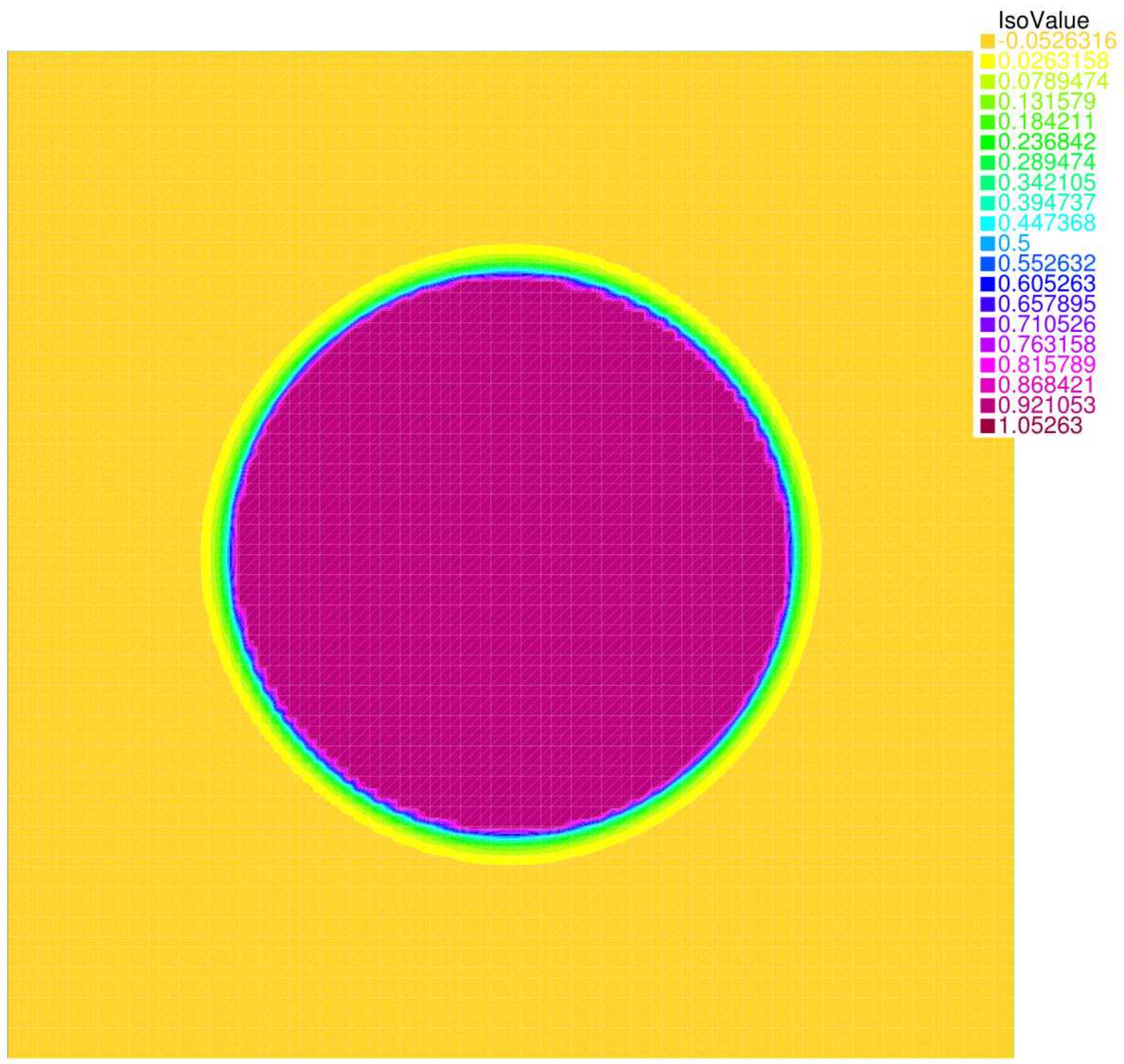}
\includegraphics[scale=0.12]{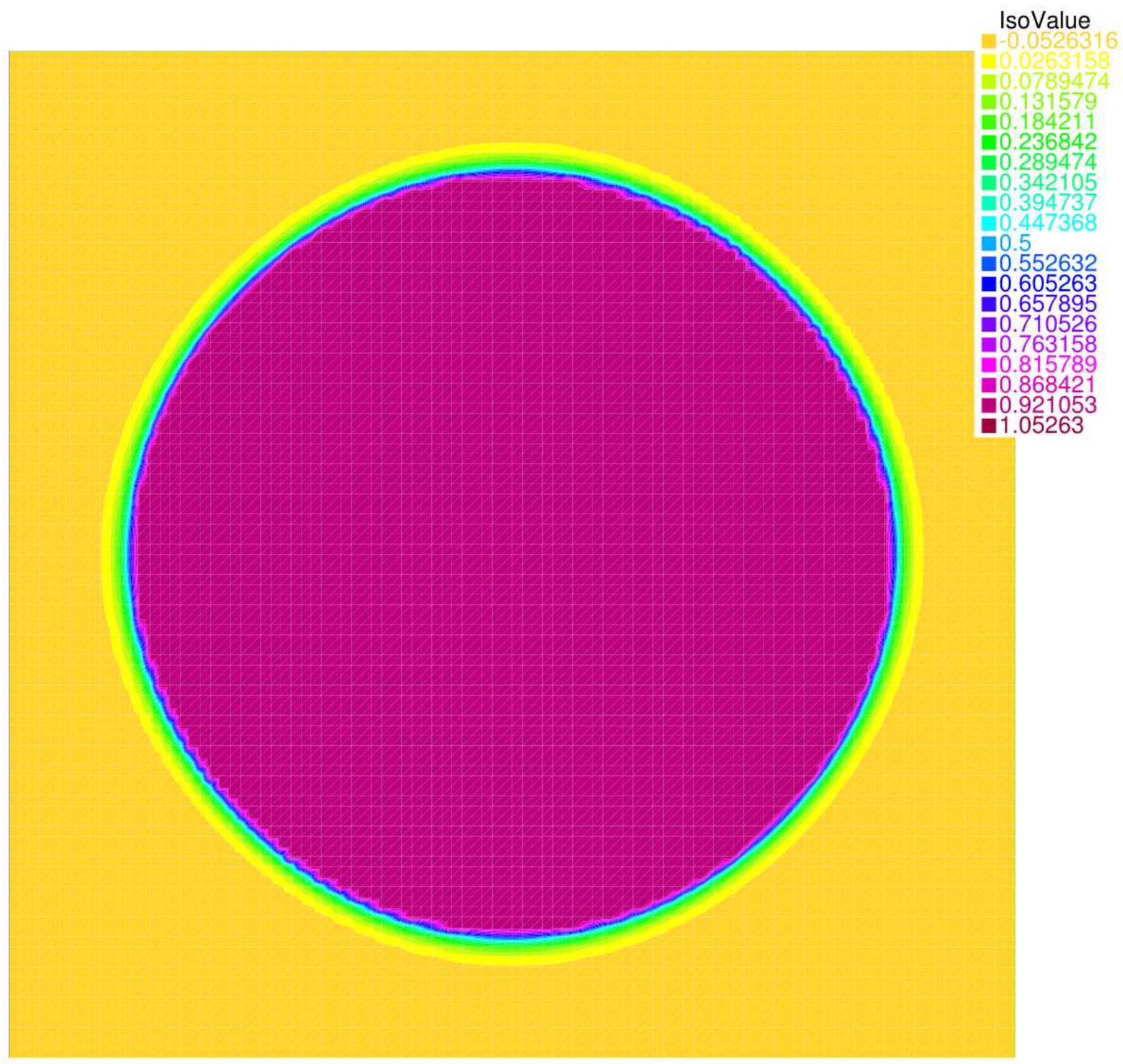}
\includegraphics[scale=0.12]{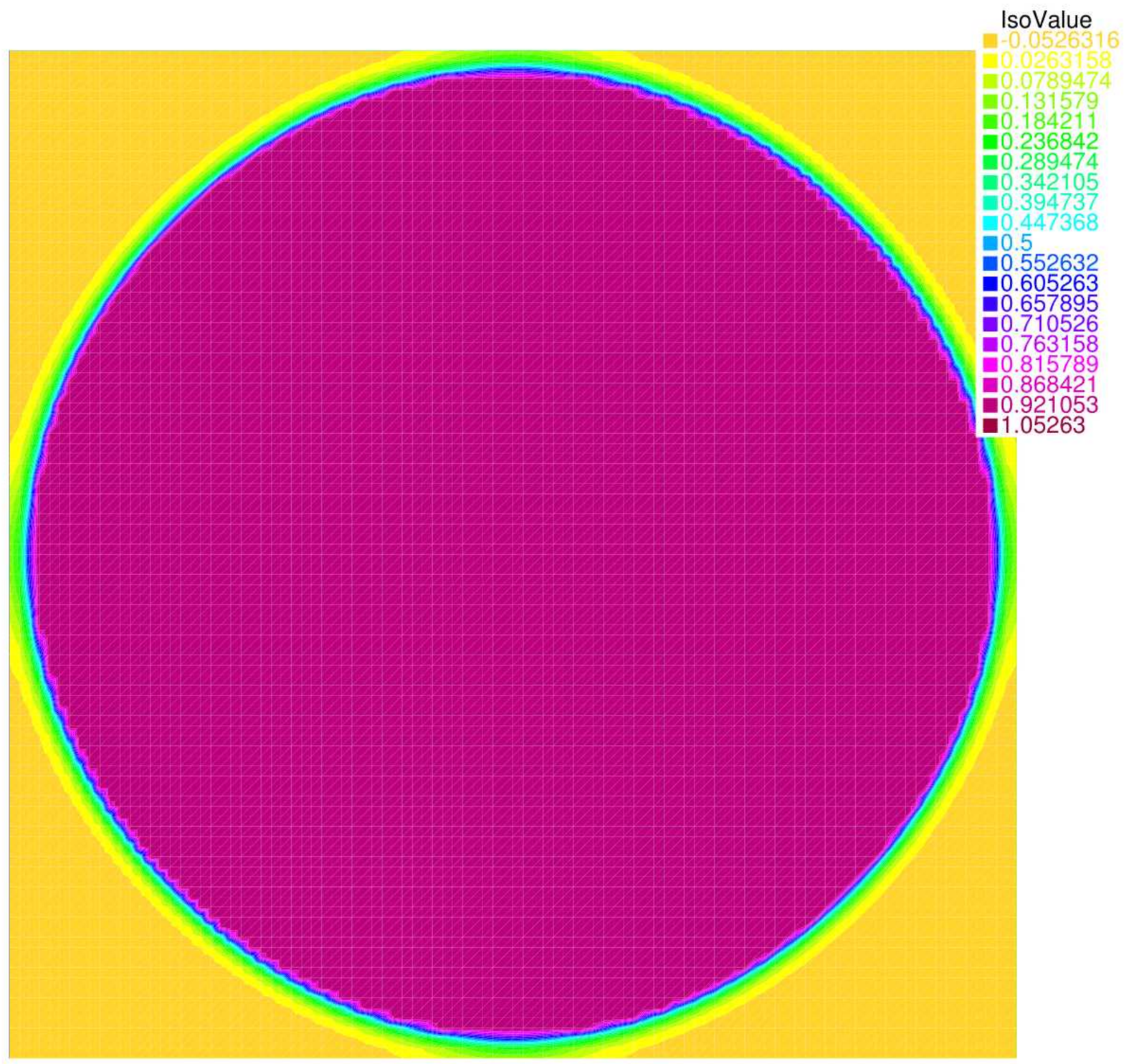}
\caption{Evolution of the density at times $t = 0.1, 0.2, 0.3, 0.4$ for $\alpha=0.5$ (top) and $1$ (bottom).}
\label{fig1:alpha}
\end{figure}

\begin{figure}[h!]
\centering
\includegraphics[scale=0.12]{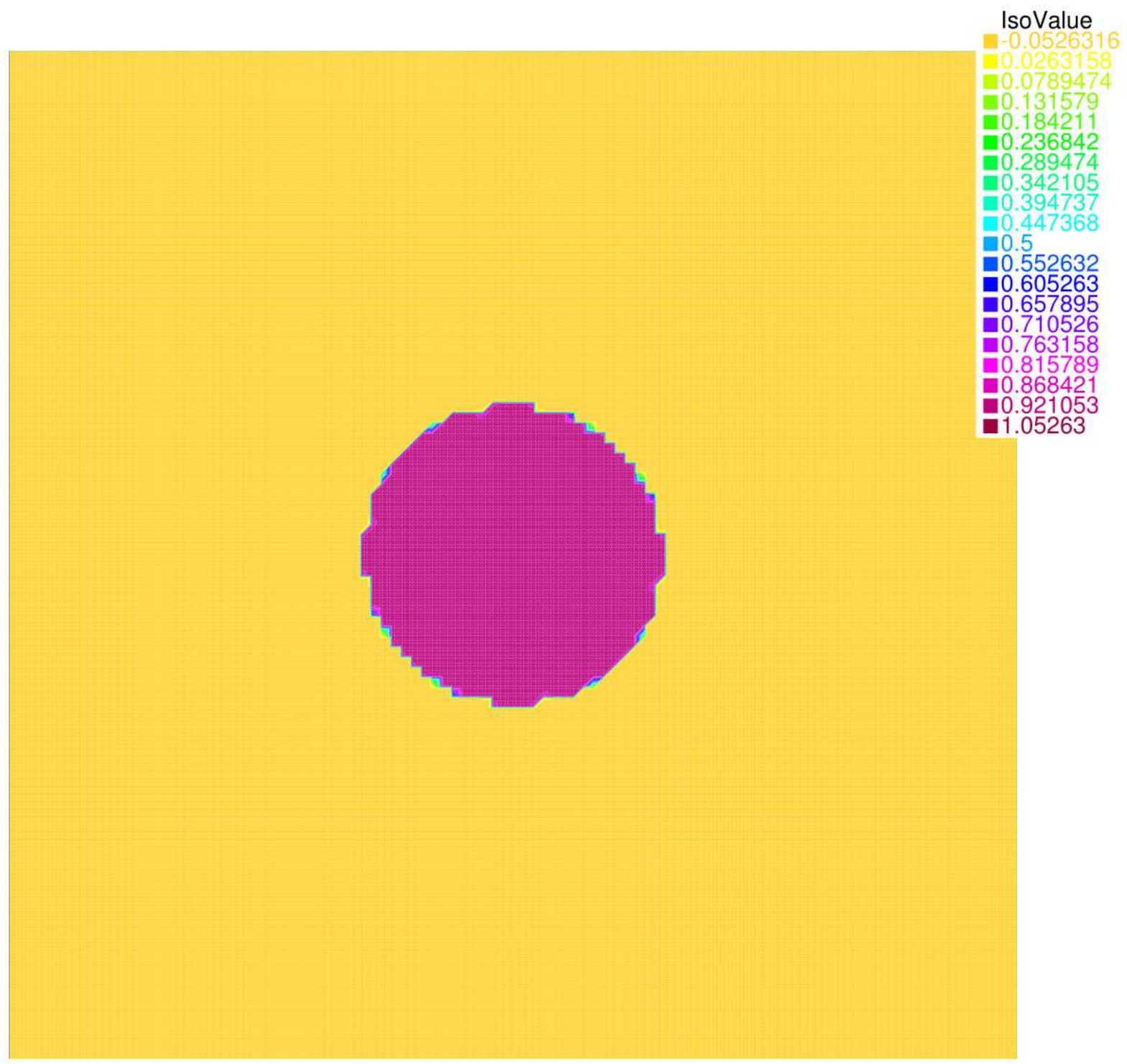}
\includegraphics[scale=0.12]{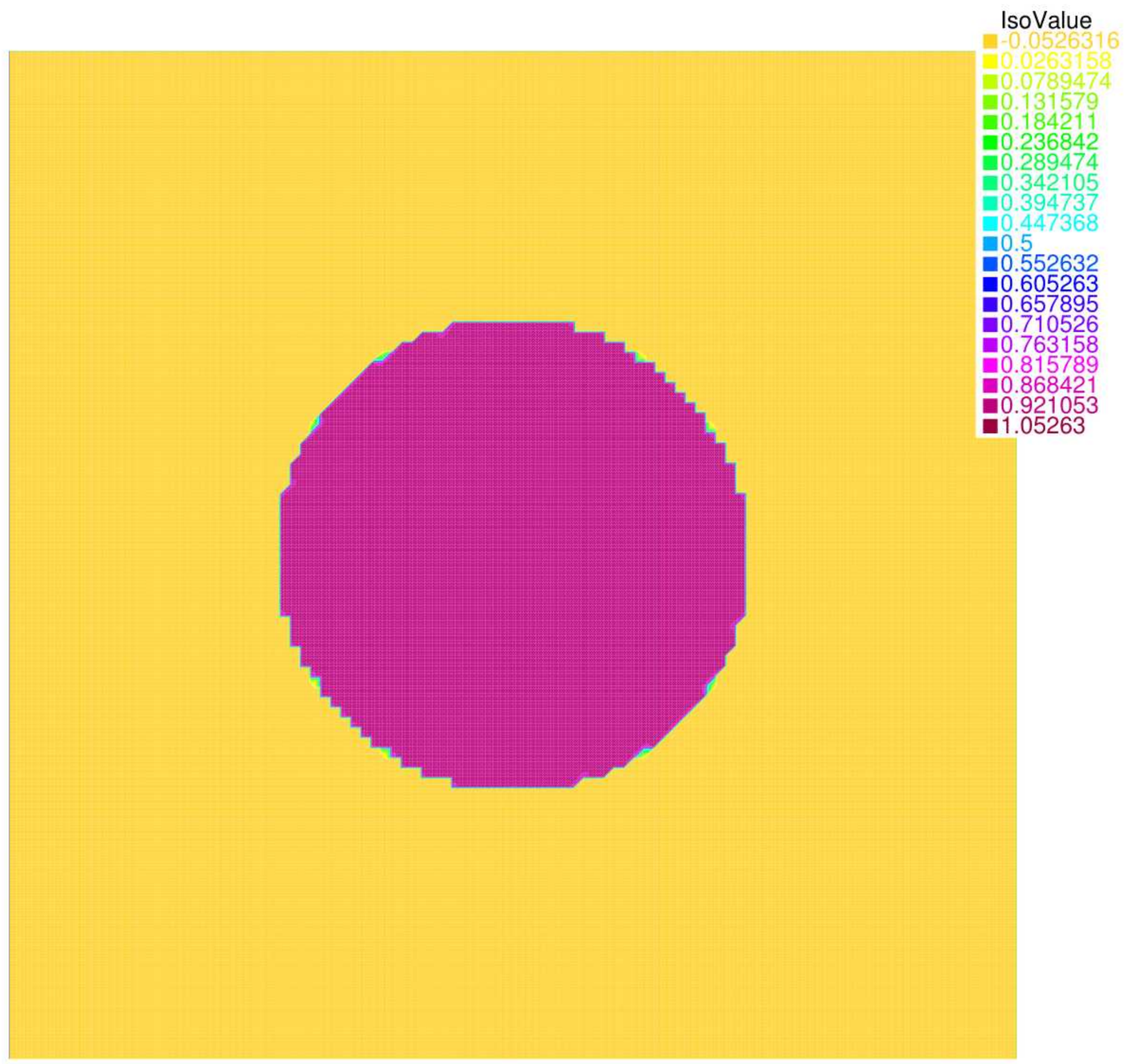}
\includegraphics[scale=0.12]{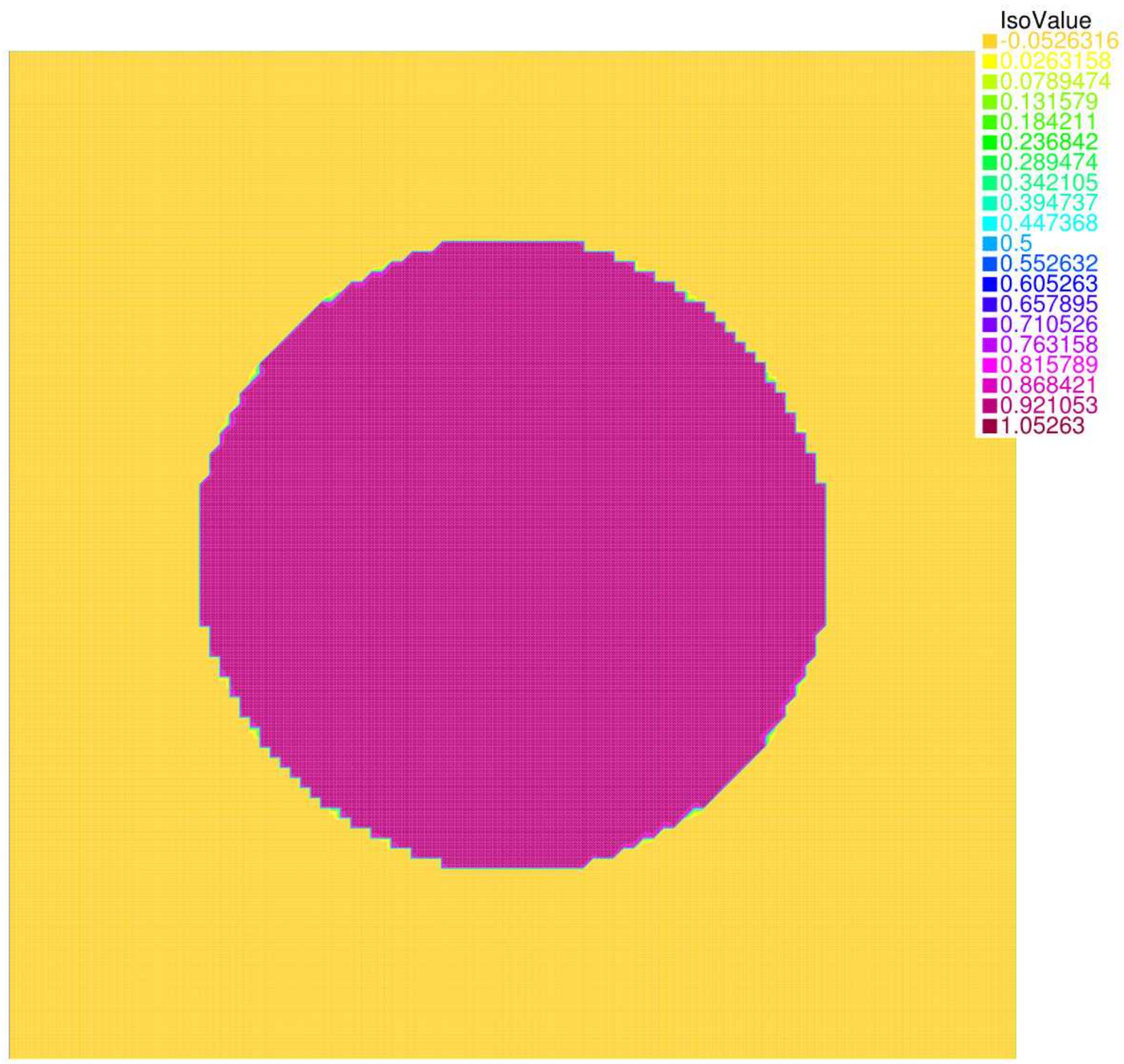}
\includegraphics[scale=0.12]{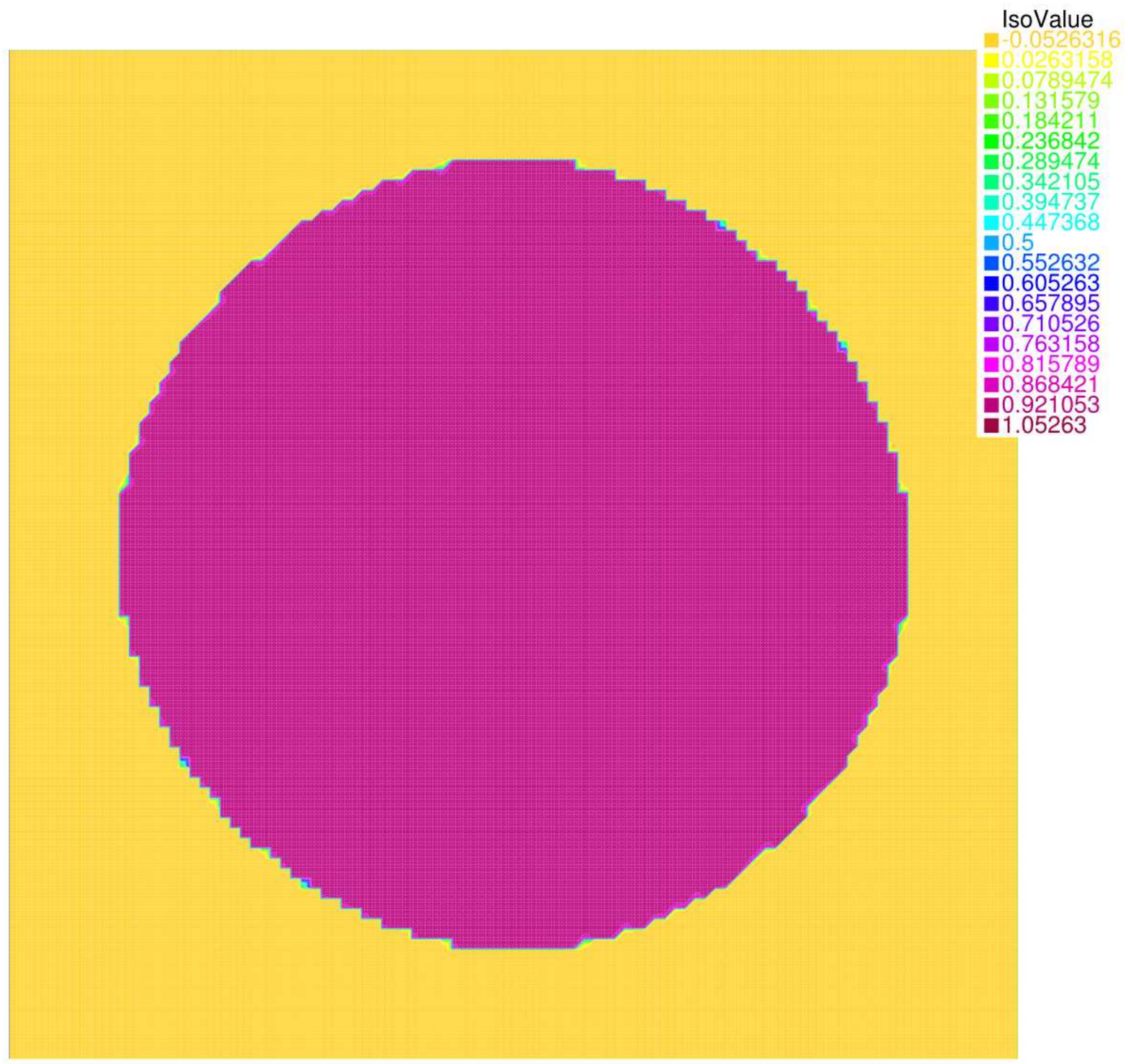}
\\
\includegraphics[scale=0.12]{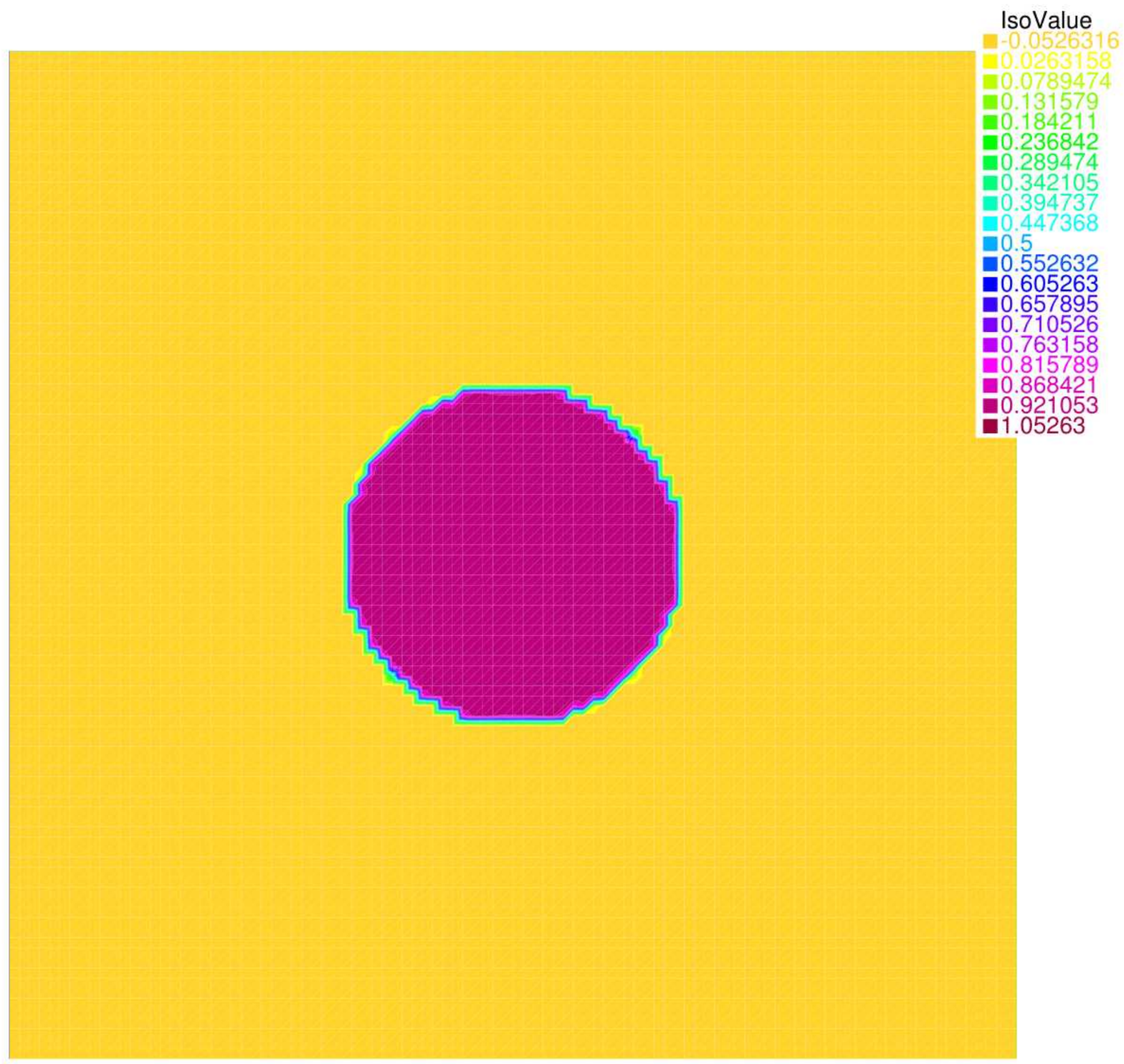}
\includegraphics[scale=0.12]{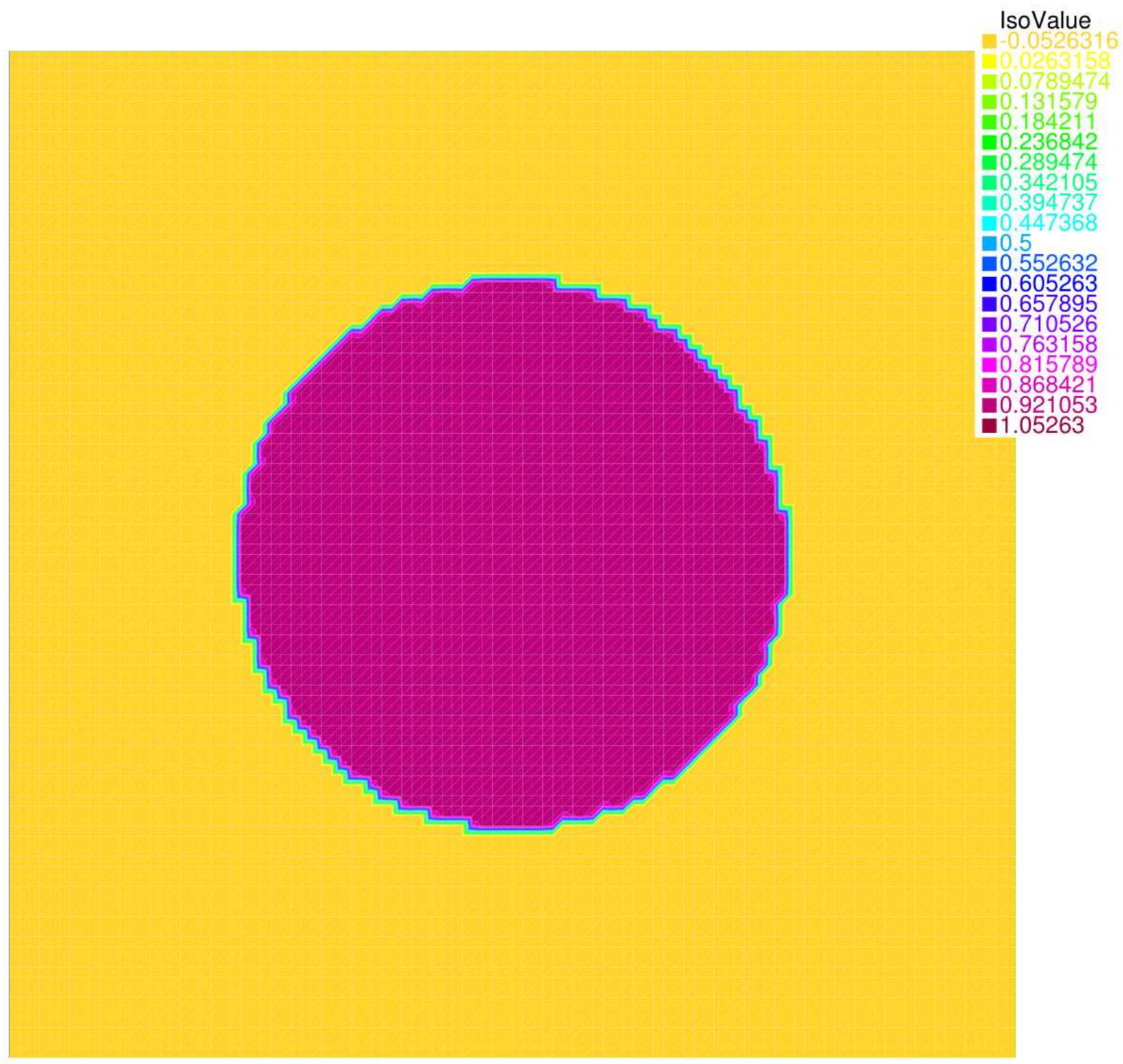}
\includegraphics[scale=0.12]{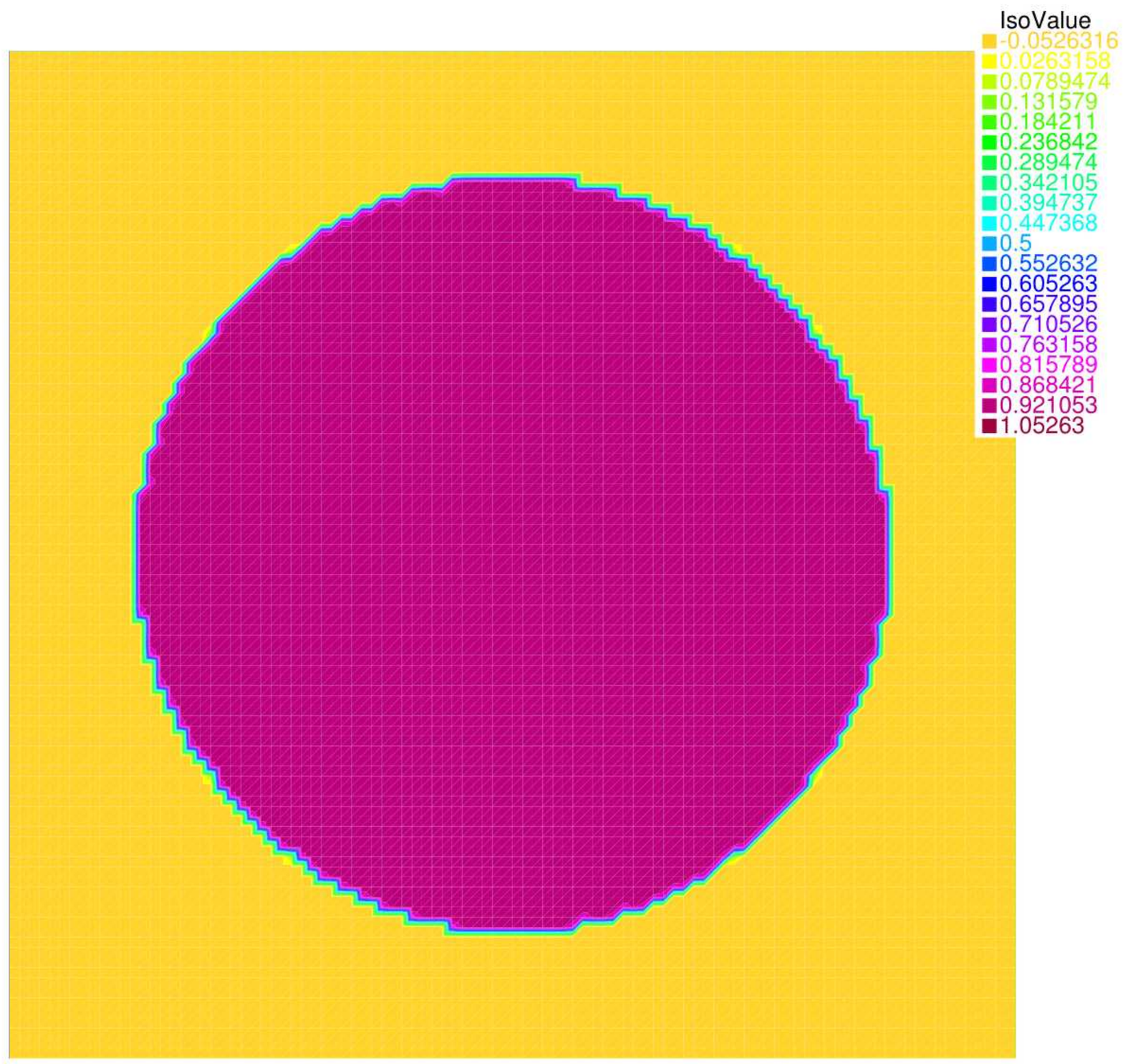}
\includegraphics[scale=0.12]{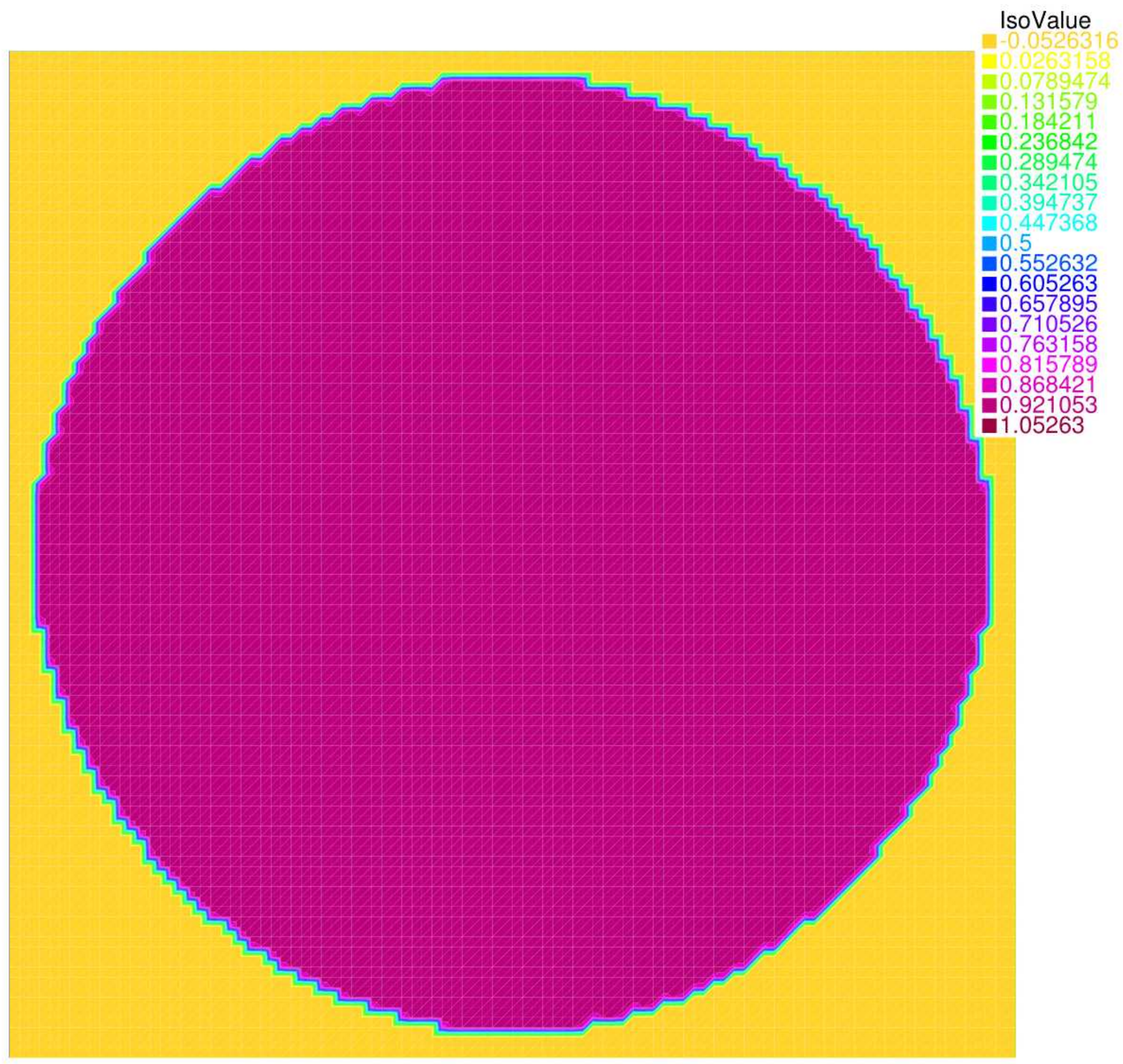}
\caption{Evolution of the pressure at times $t = 0.1, 0.2, 0.3, 0.4$ for $\alpha=0.5$ (top) and $\alpha=1$ (bottom). }
\label{fig2:alpha}
\end{figure}

\begin{figure}[h!]
\includegraphics[scale=0.12]{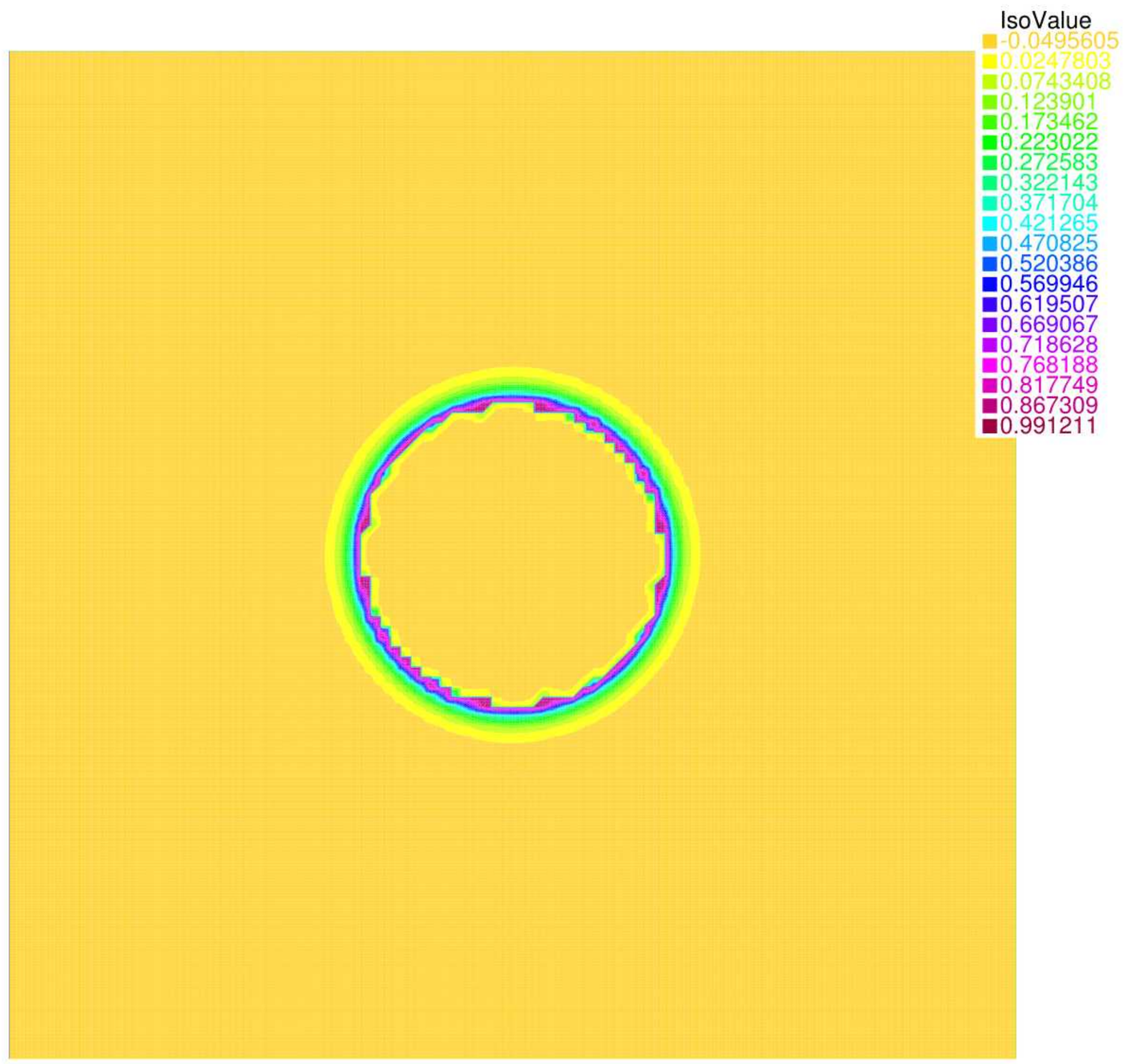}
\includegraphics[scale=0.12]{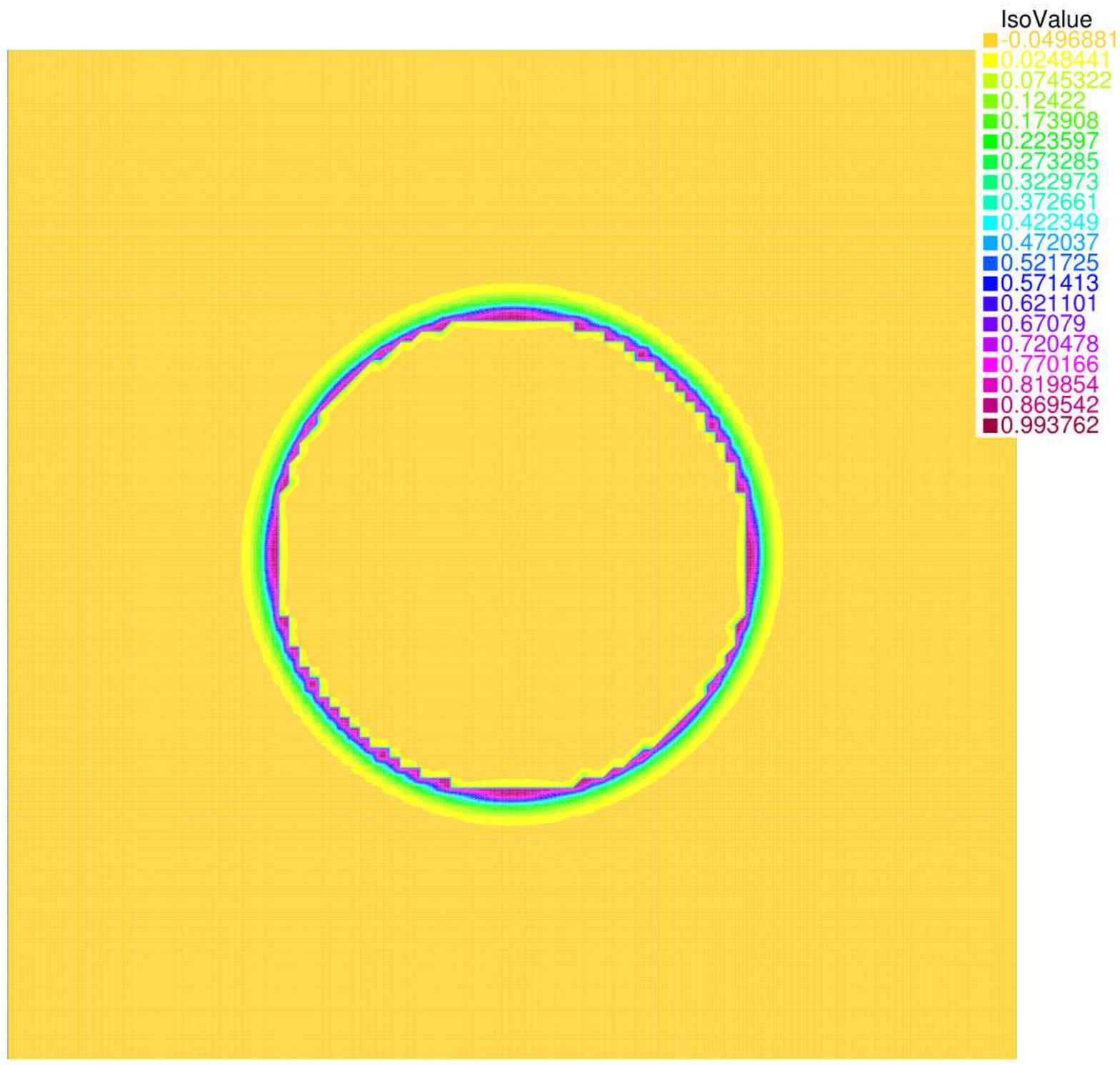}
\includegraphics[scale=0.12]{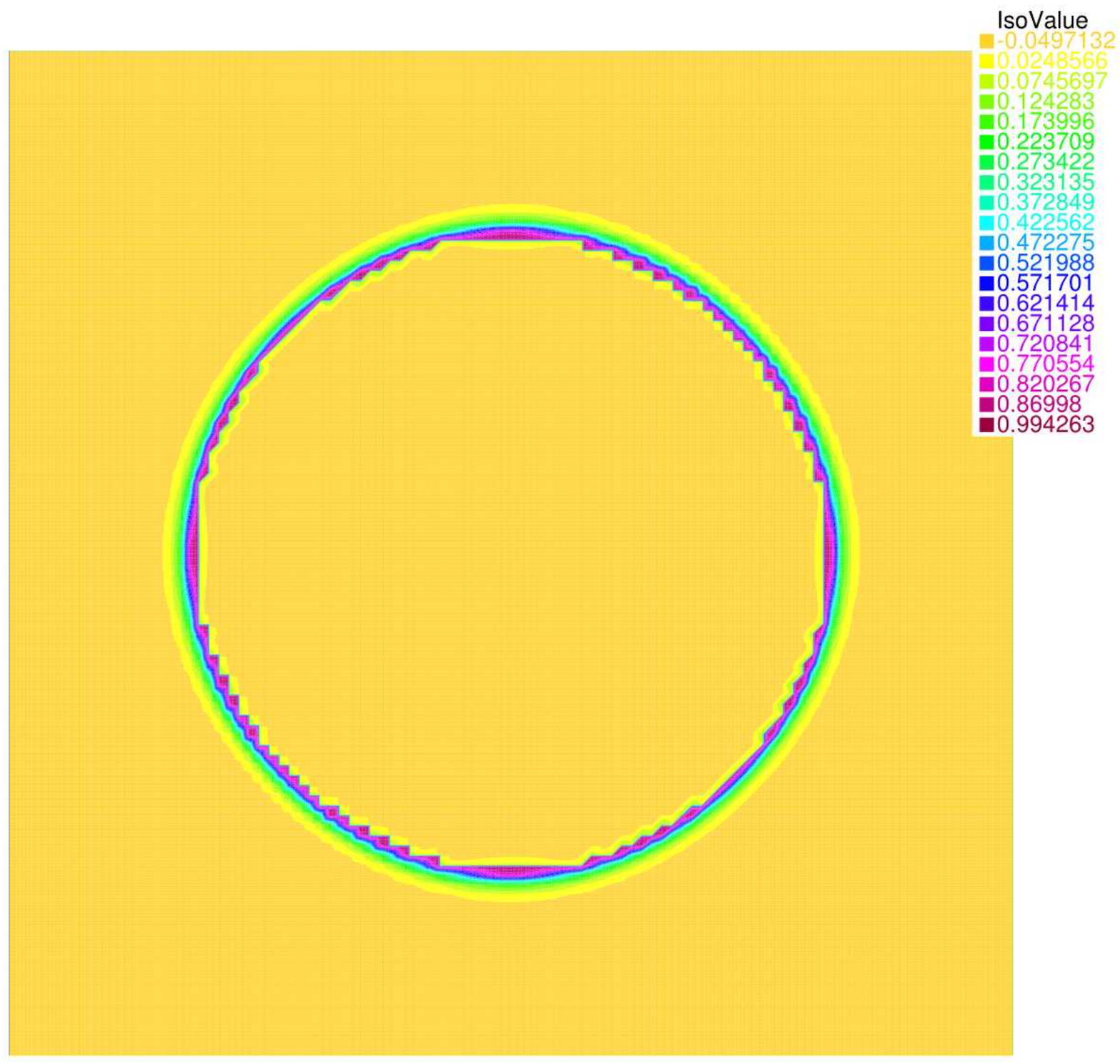}
\includegraphics[scale=0.12]{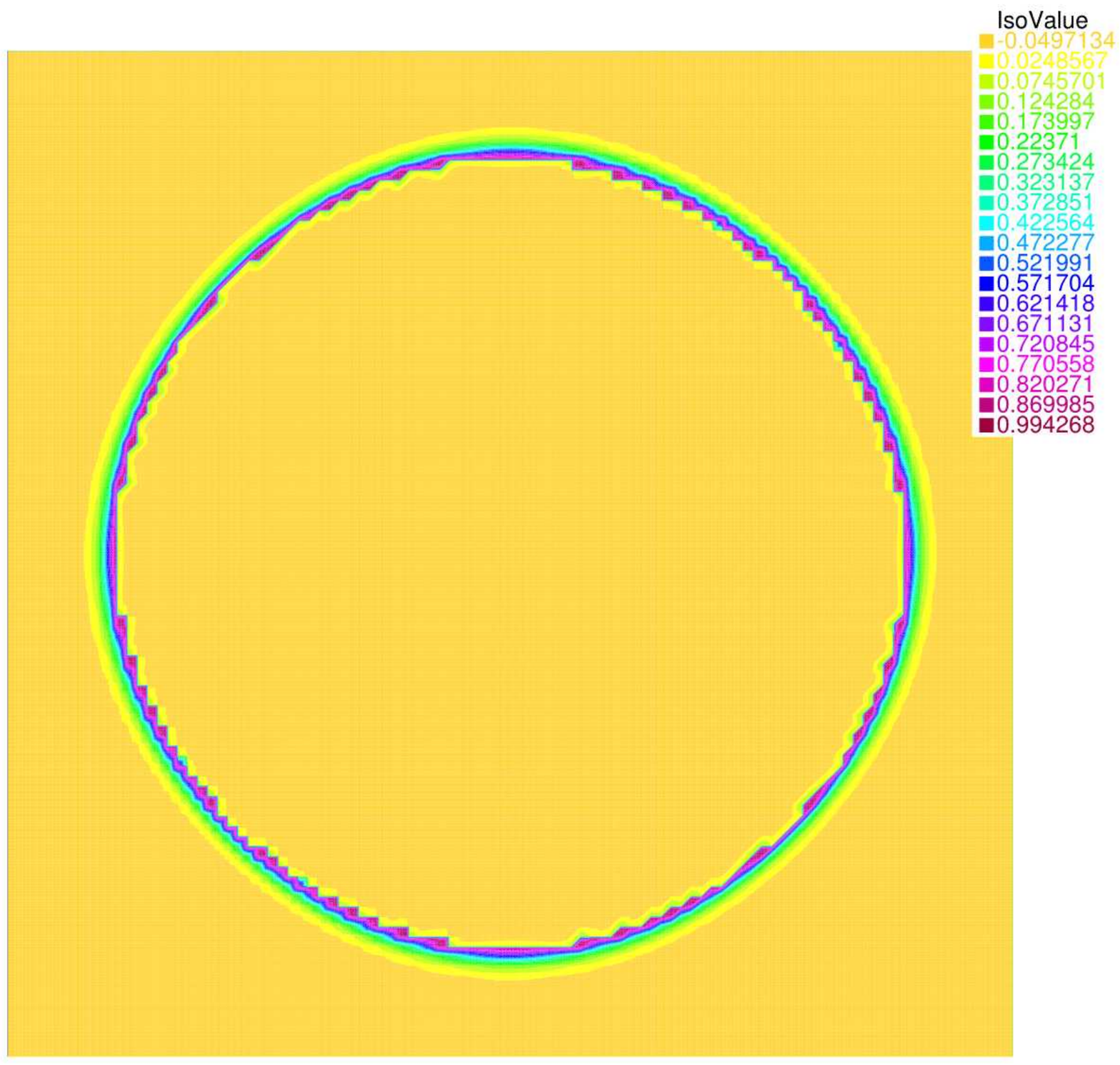}
\\
\includegraphics[scale=0.12]{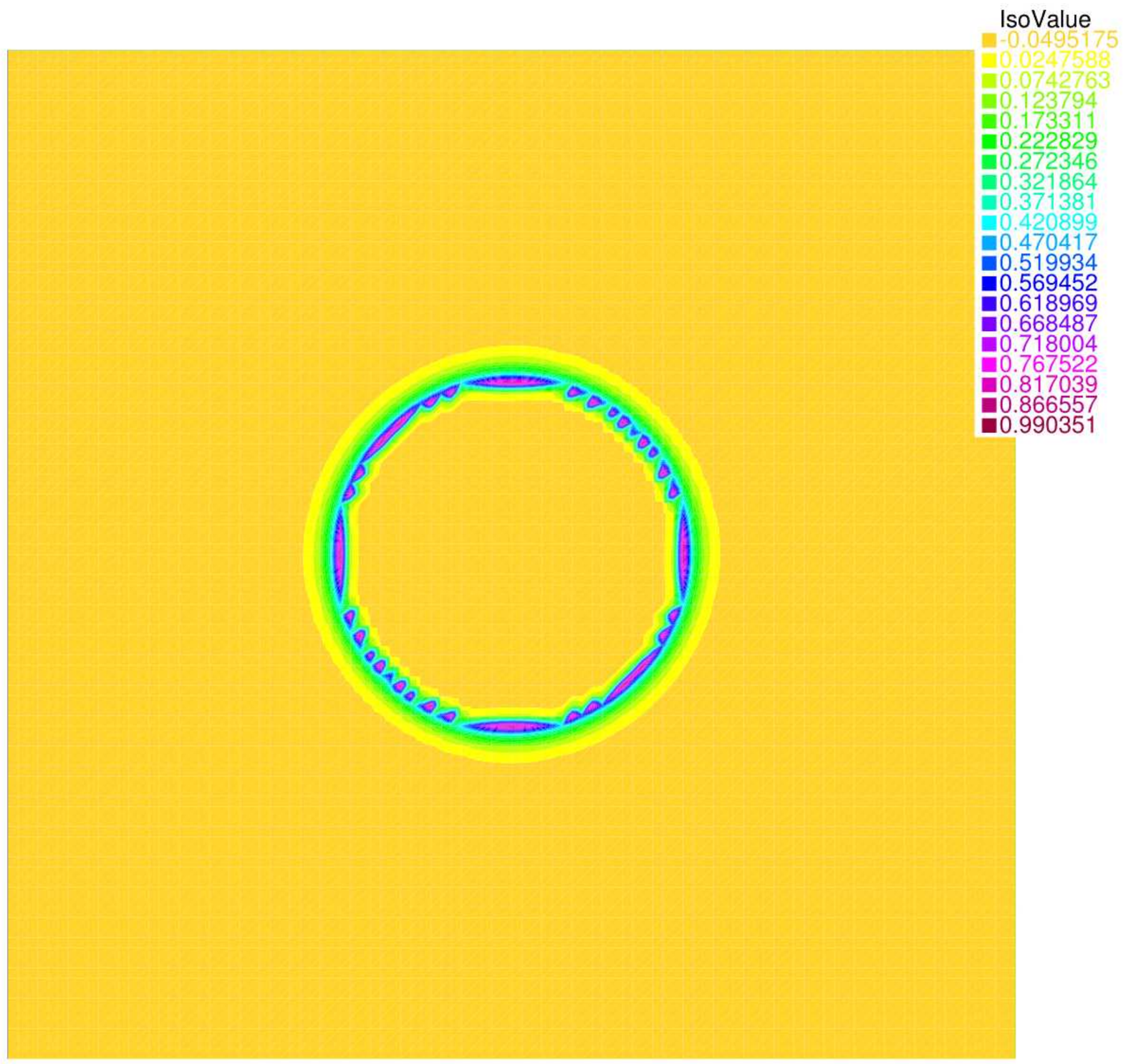}
\includegraphics[scale=0.12]{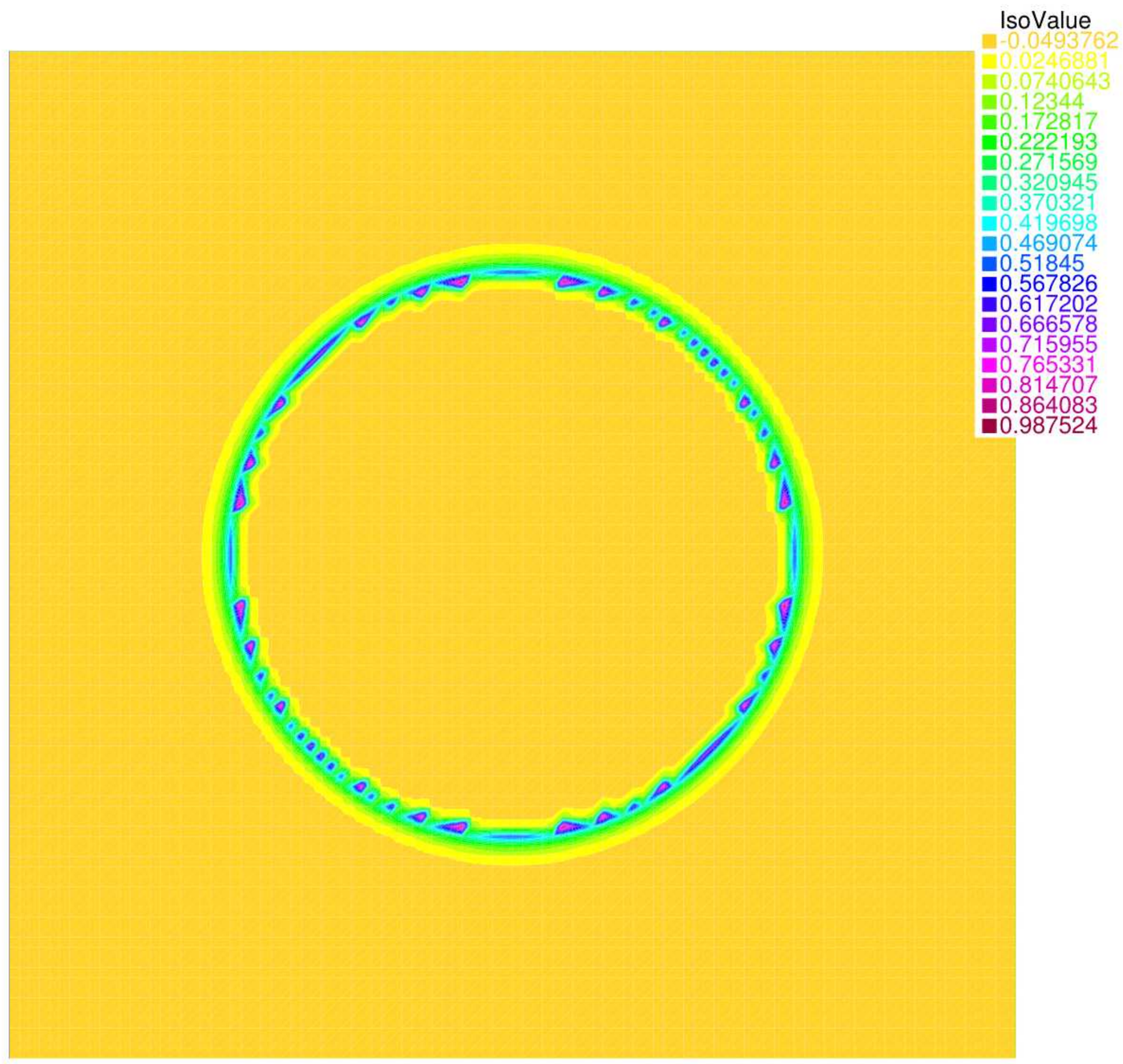}
\includegraphics[scale=0.12]{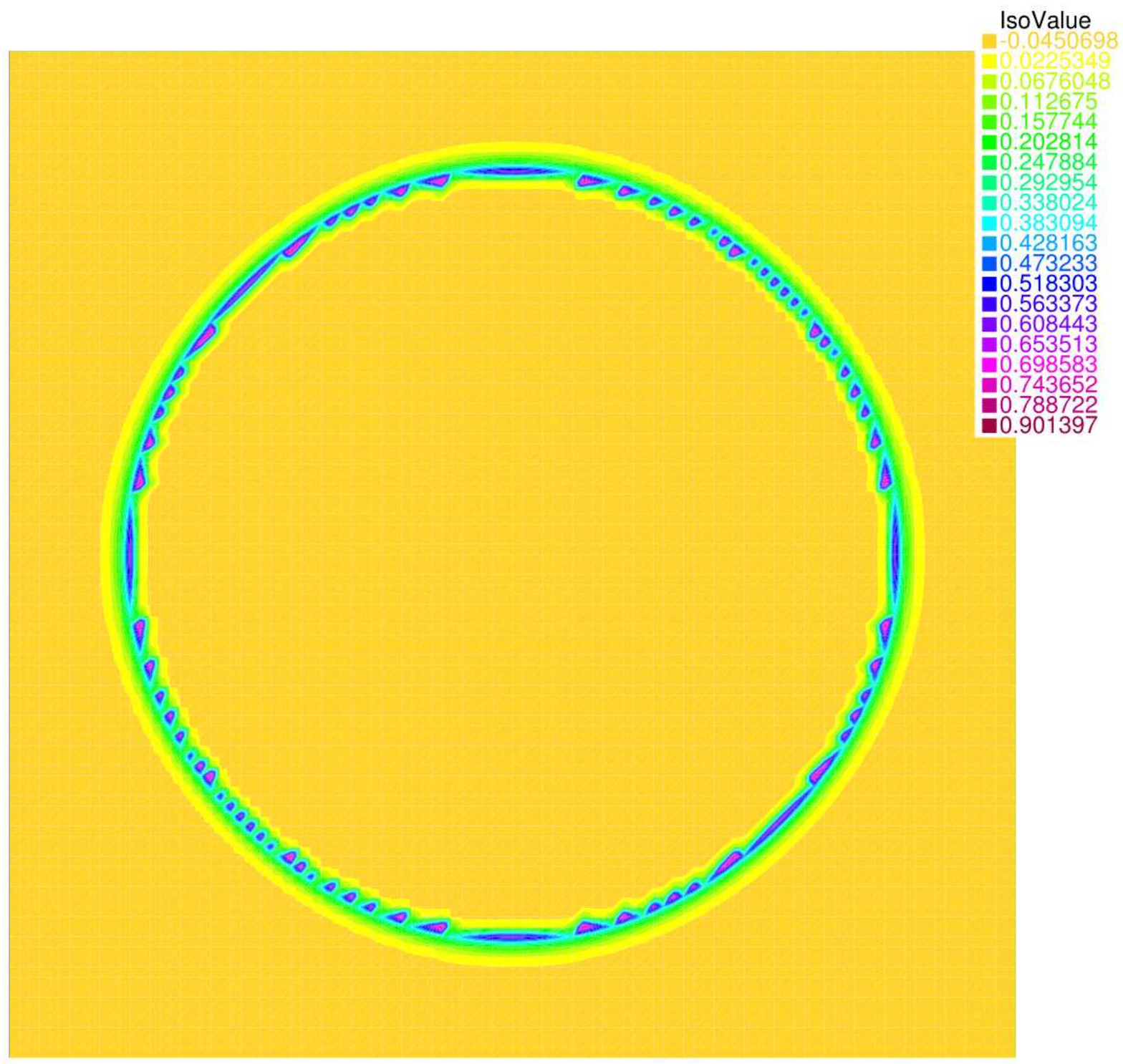}
\includegraphics[scale=0.12]{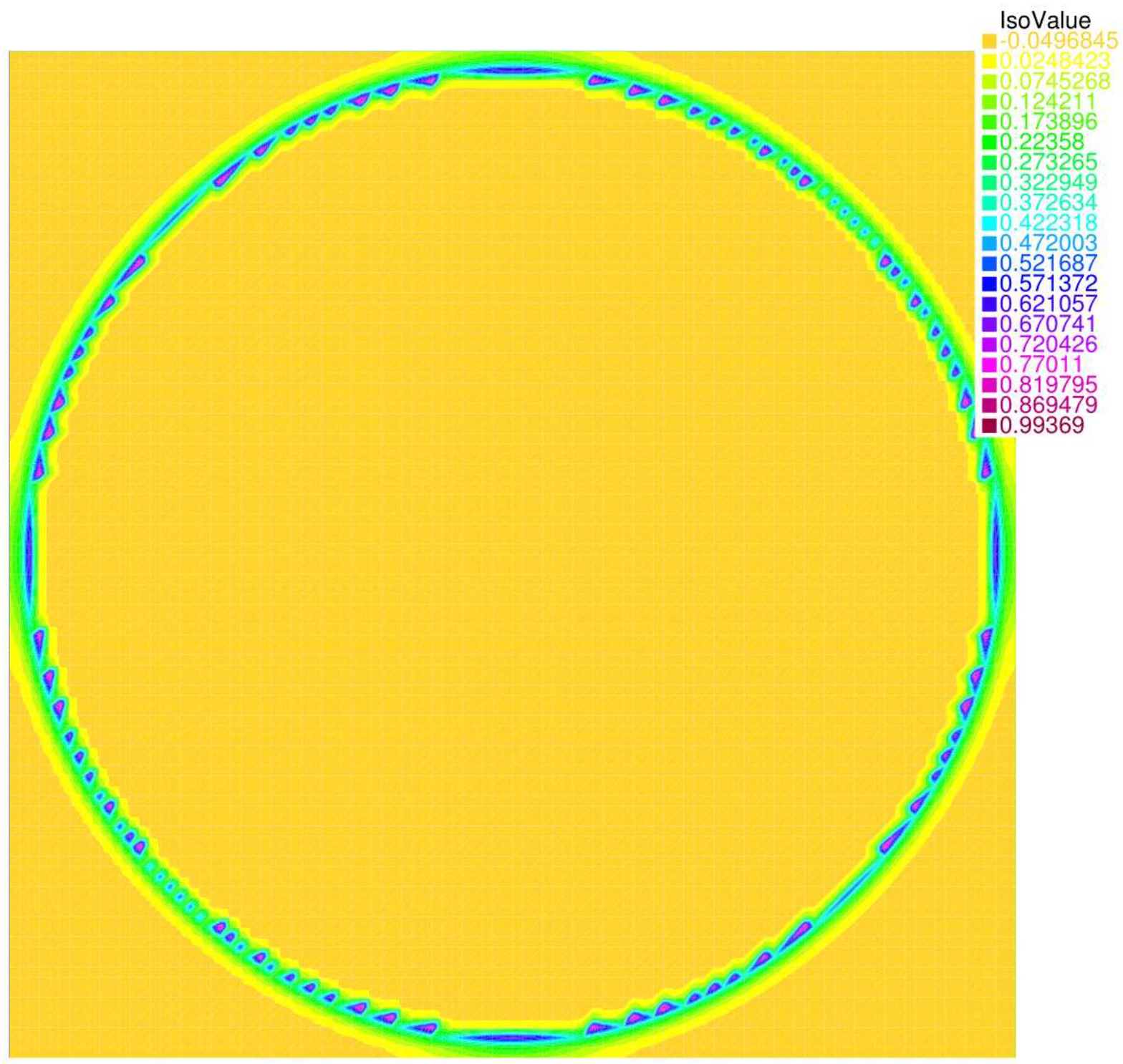}
\caption{Evolution of the difference between the density and pressure at times $t = 0.1, 0.2, 0.3, 0.4$ for $\alpha=0.5$ (top) and $\alpha=1$ (bottom). }
\label{fig3:alpha}
\end{figure}

\subsubsection{Analysis of the effect of $\nu$ (active motion coefficient)} Now we set $P_{\rm max}=1$ and take different values of $\nu=0$, $0.5$ and $1$. The evolution of the density $n_{h,k}$ is shown in Figure \ref{fig:nu} where we see that the velocity of propagation of the tumor cells increases with respect to $\nu$ as noted for times $t=0.1$, $0.2$, $0.3$ and $0.4$. Moreover, no particular differences have been observed in the width of the interface between the tumor and pre-tumor cells for the different values of $\nu$.    

\begin{figure}[h!]
\centering
\includegraphics[scale=0.12]{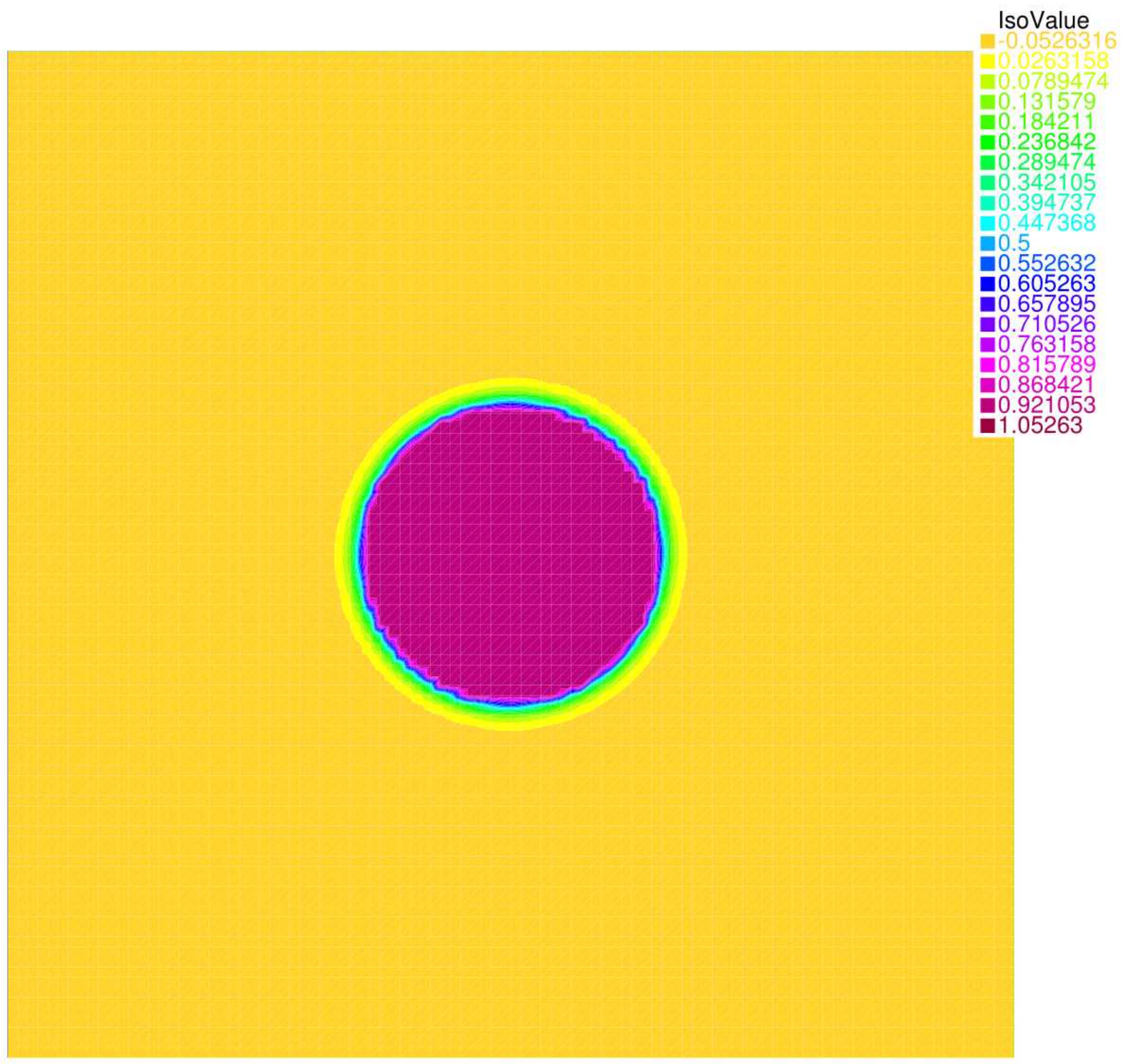}
\includegraphics[scale=0.12]{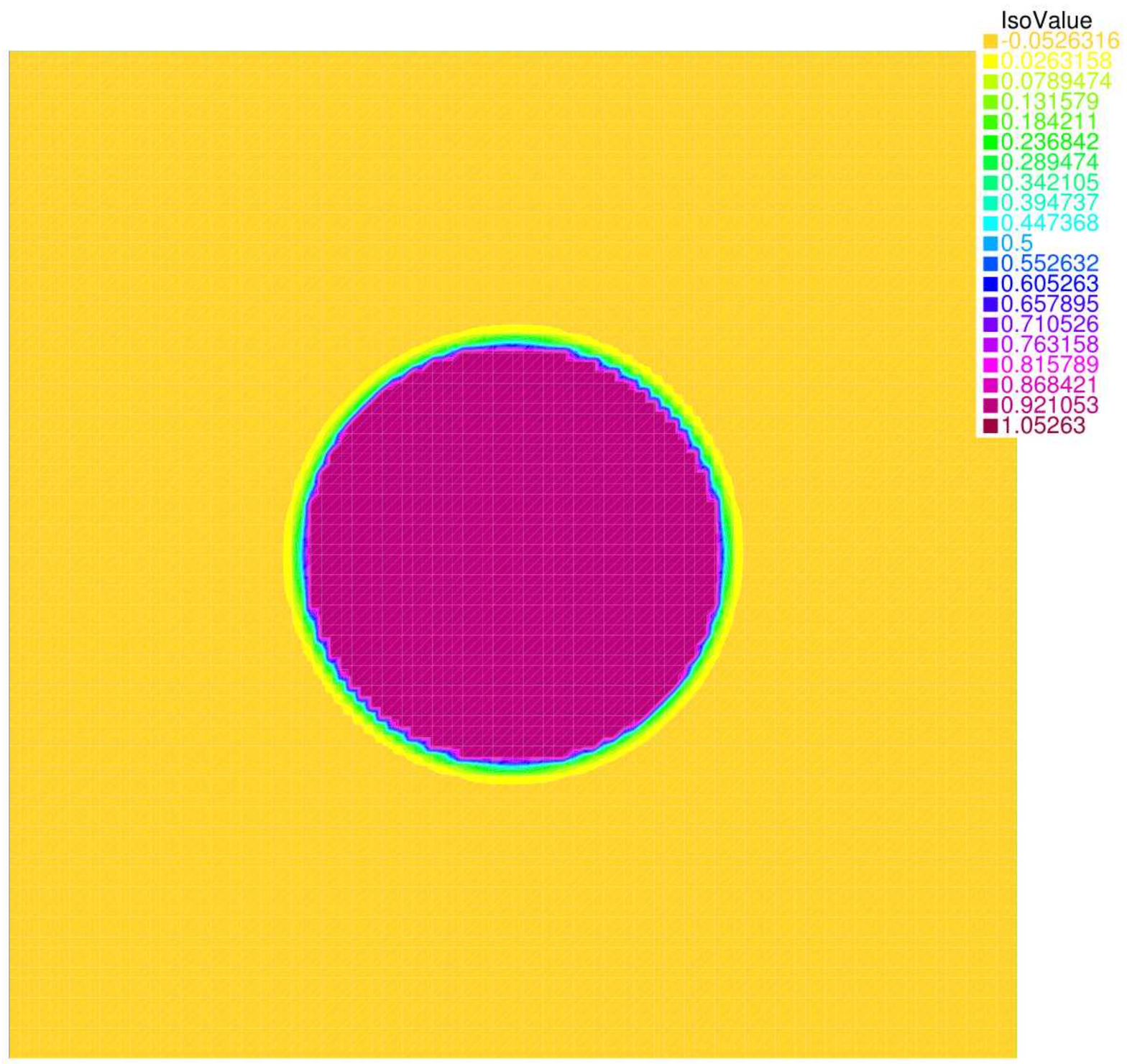}
\includegraphics[scale=0.12]{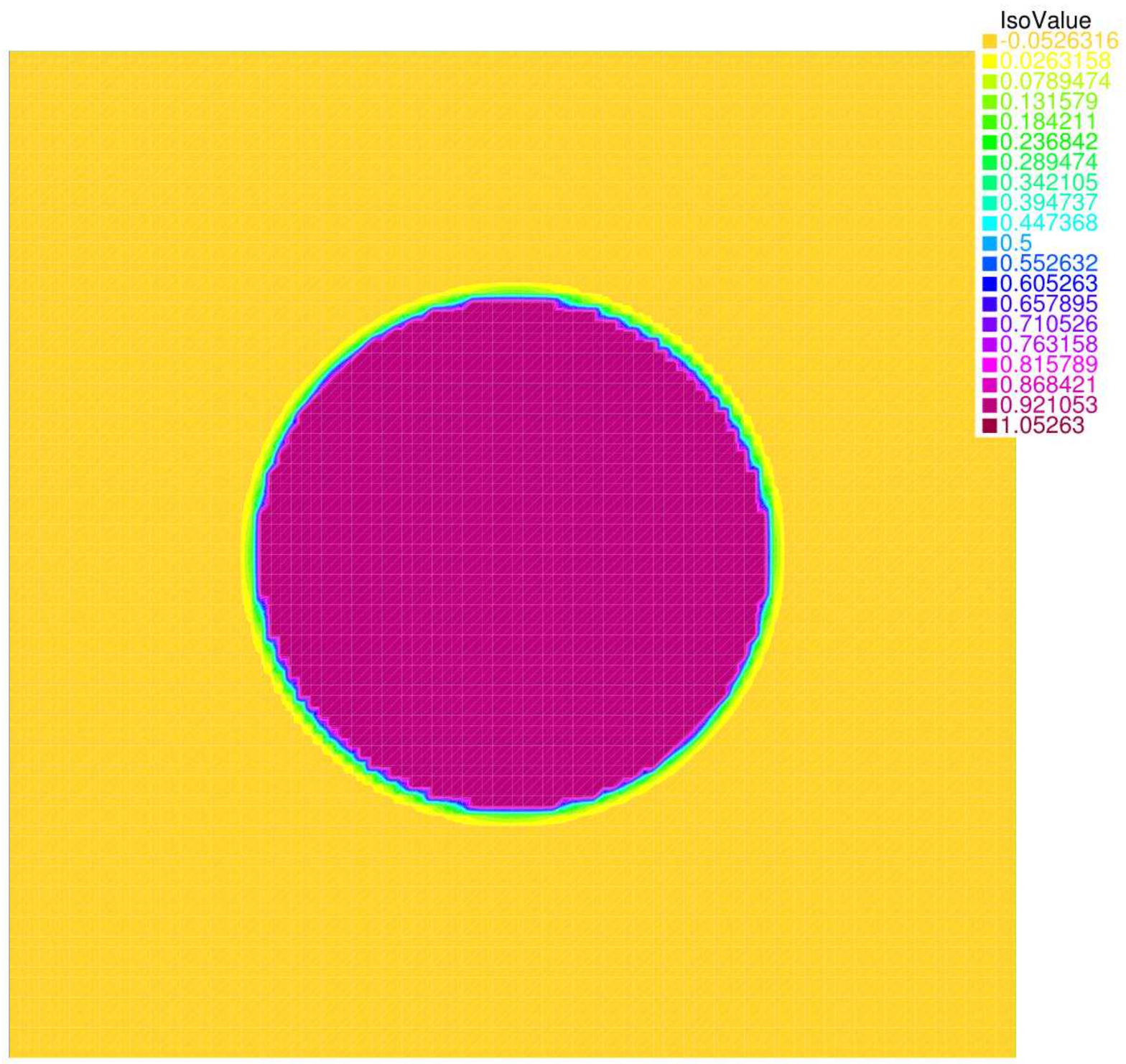}
\includegraphics[scale=0.12]{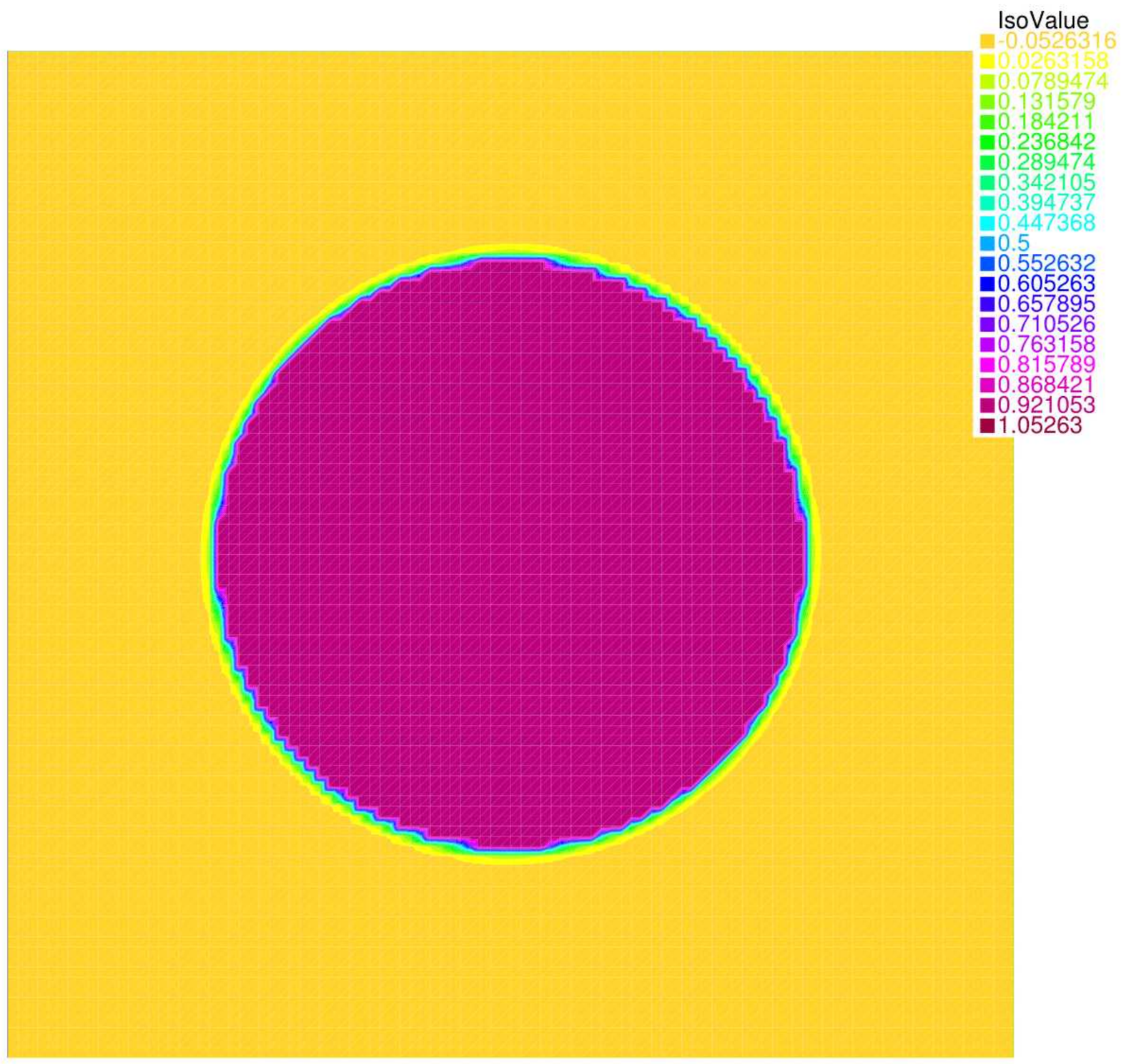}
\\
\includegraphics[scale=0.12]{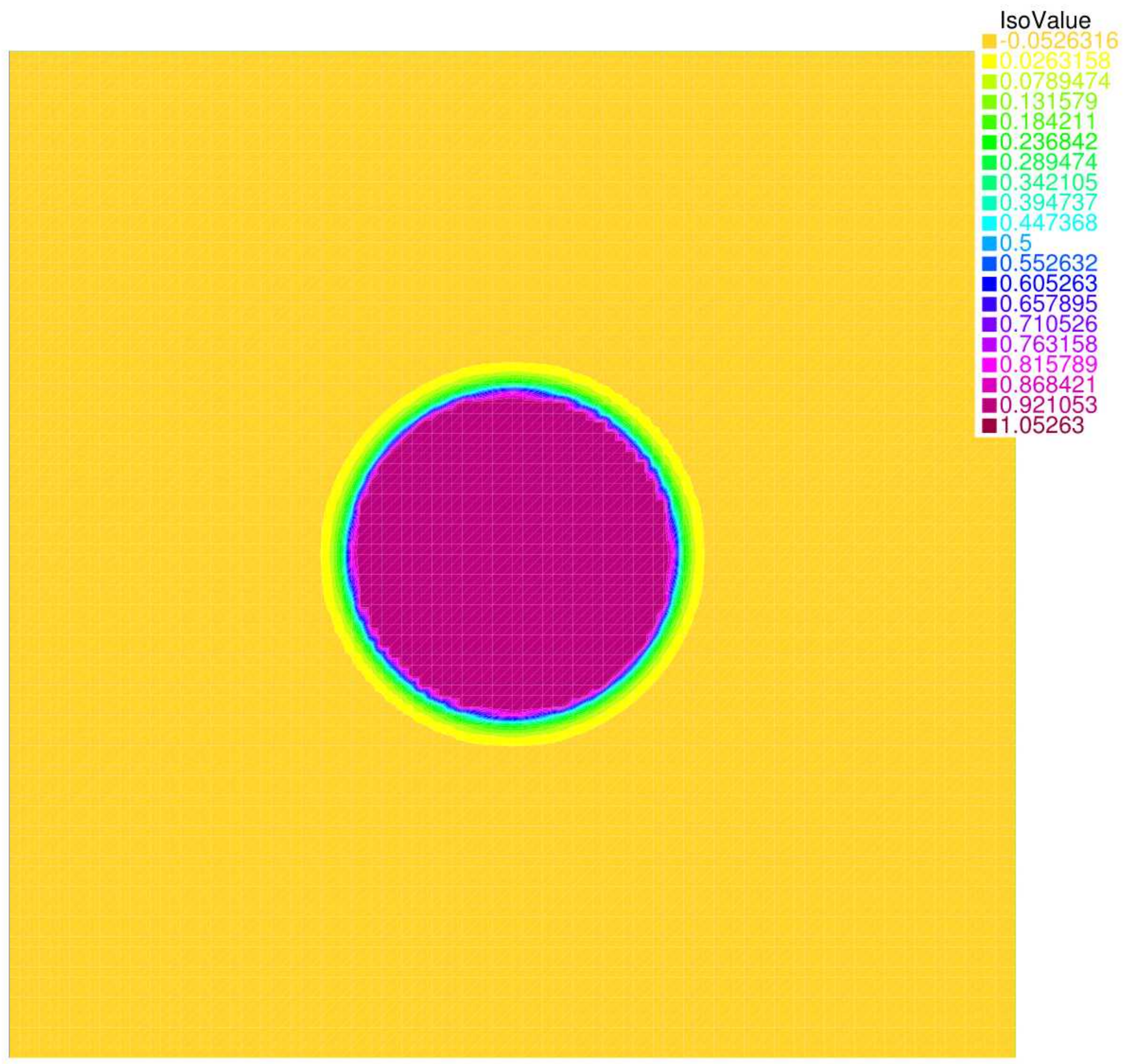}
\includegraphics[scale=0.12]{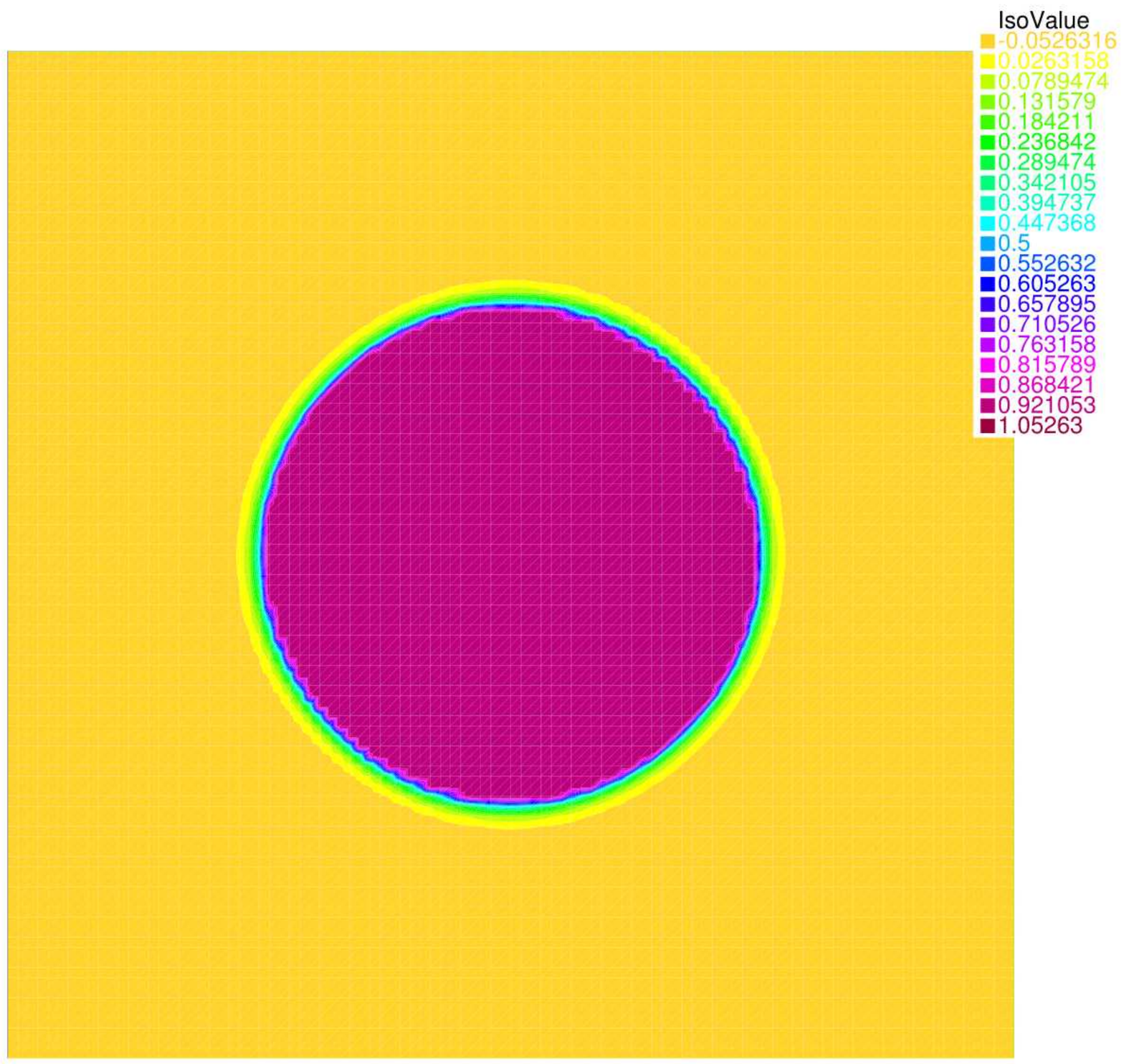}
\includegraphics[scale=0.12]{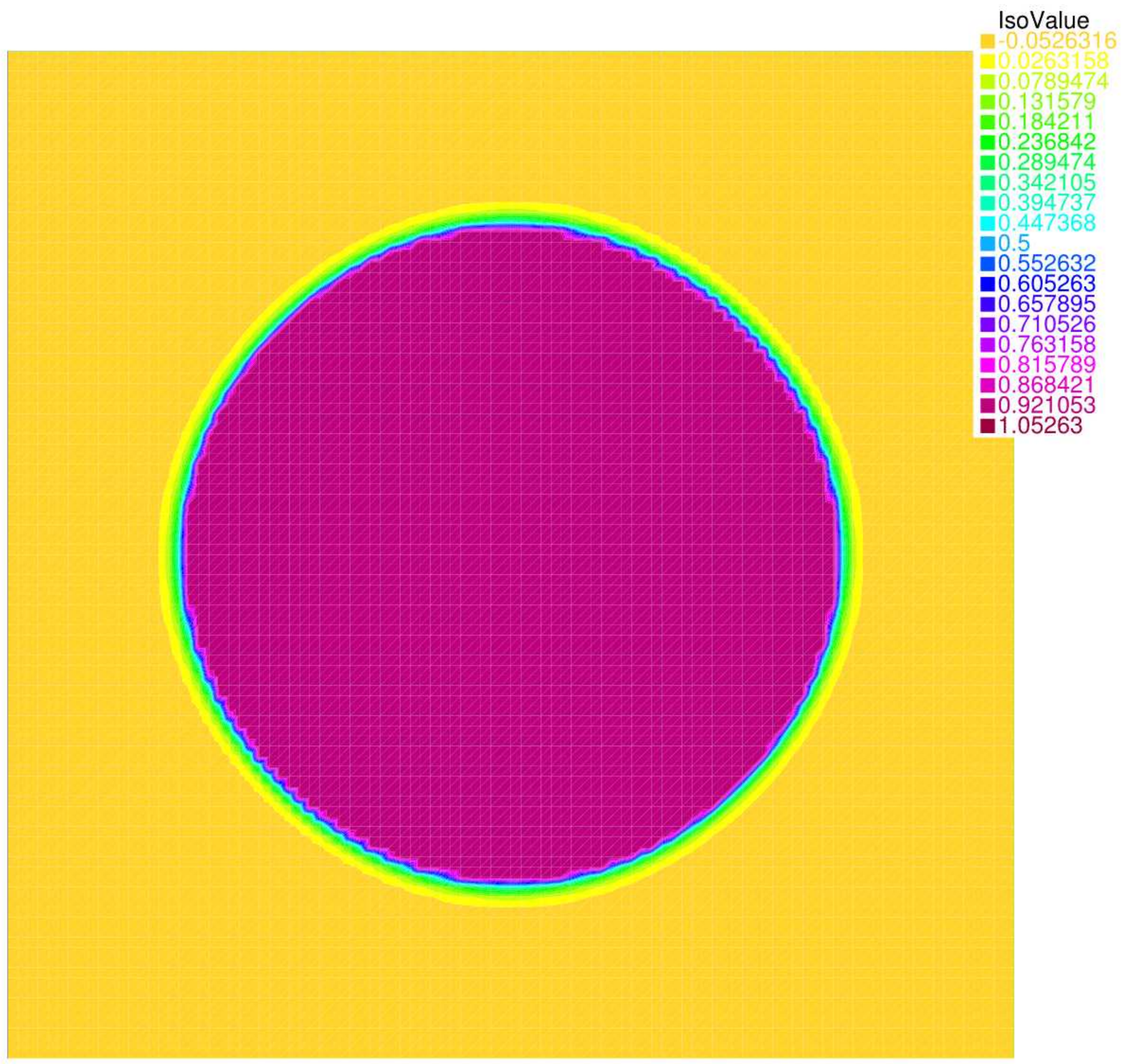}
\includegraphics[scale=0.12]{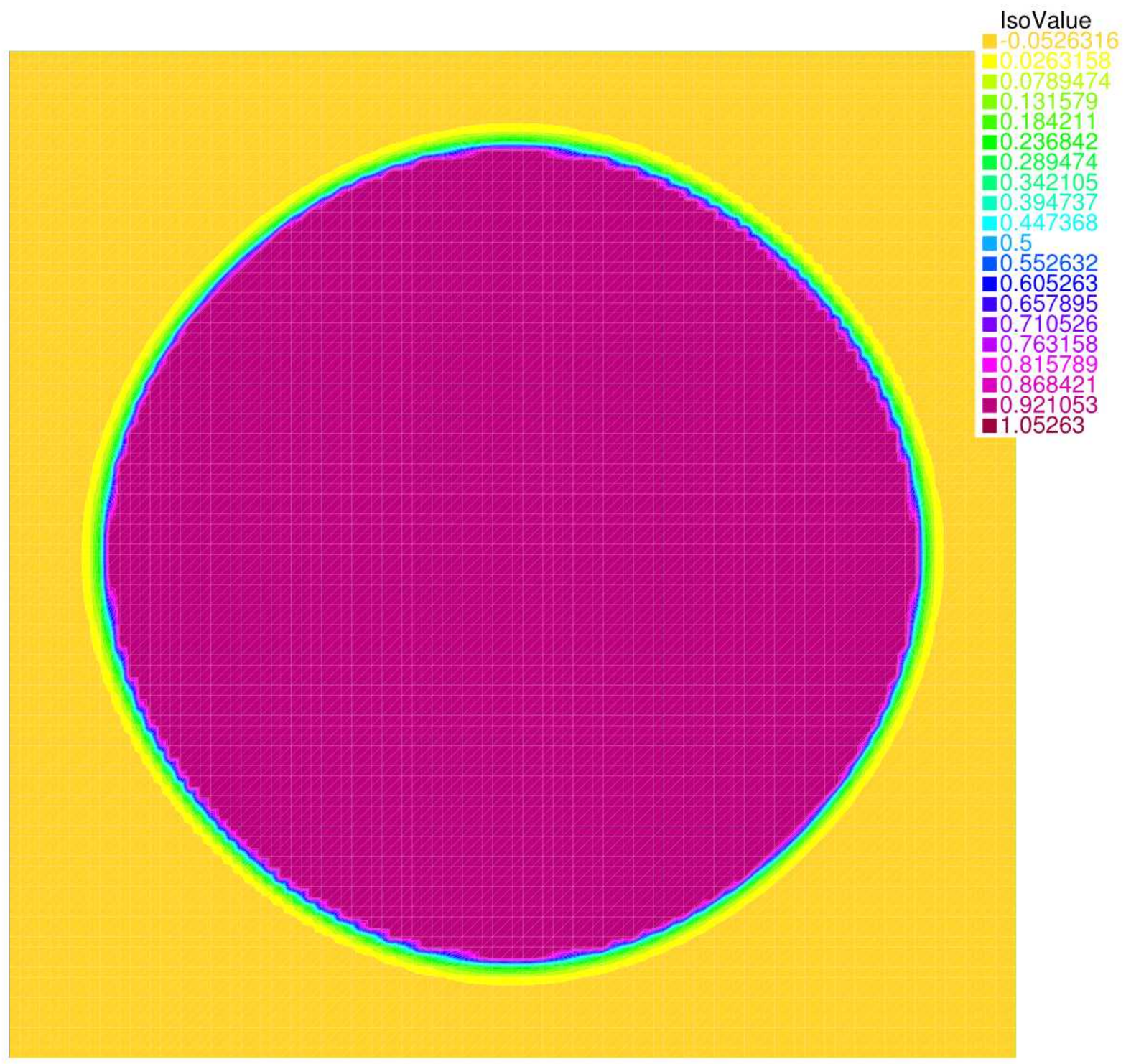}
\\
\includegraphics[scale=0.12]{Figures3/Density-01_nu=1.pdf}
\includegraphics[scale=0.12]{Figures3/Density-02_nu=1.pdf}
\includegraphics[scale=0.12]{Figures3/Density-03_nu=1.pdf}
\includegraphics[scale=0.12]{Figures3/Density-04_nu=1.pdf}
\caption{Comparison of the density at times $t = 0.1, 0.2, 0.3, 0.4$ for different $\nu=0$ (top), $0.5$ (middle), $1$ (bottom).}
\label{fig:nu}
\end{figure}

\subsubsection{Analysis of the effect of $k$} In this simulation we select $k=10$ and $1000$. The first thing we have noted is that there is a dependence between $k$ and $\tau$ which has been taken $0.5\cdot 10^{-5}$. As can be seen in Figure \ref{fig:k}, there are no particular differences for $k=10$ and $1000$ at times $t=0.1$, $0.2$, $0.3$ and $0.4$.  

\begin{figure}[h!]
\centering
\includegraphics[scale=0.12]{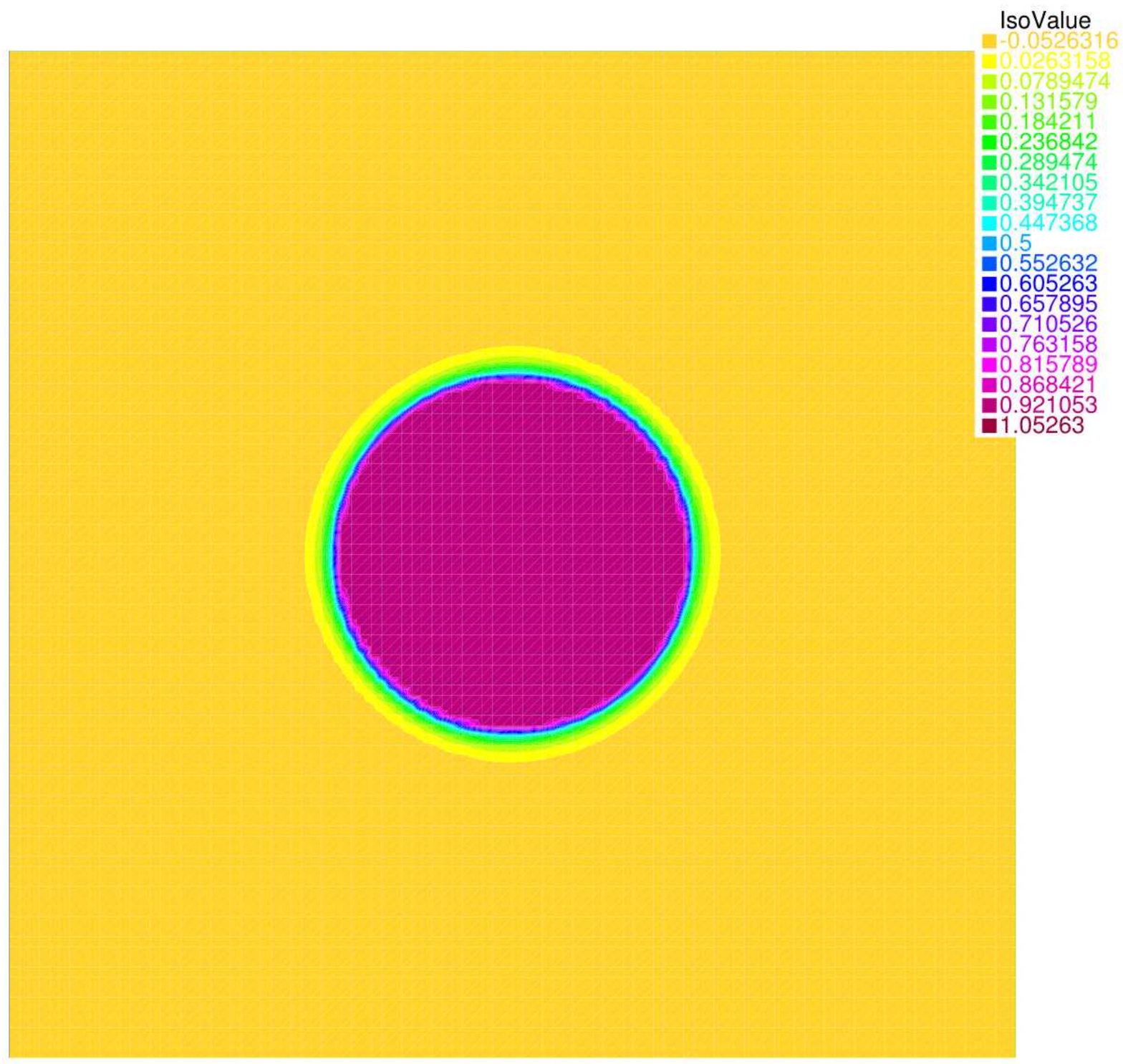}
\includegraphics[scale=0.12]{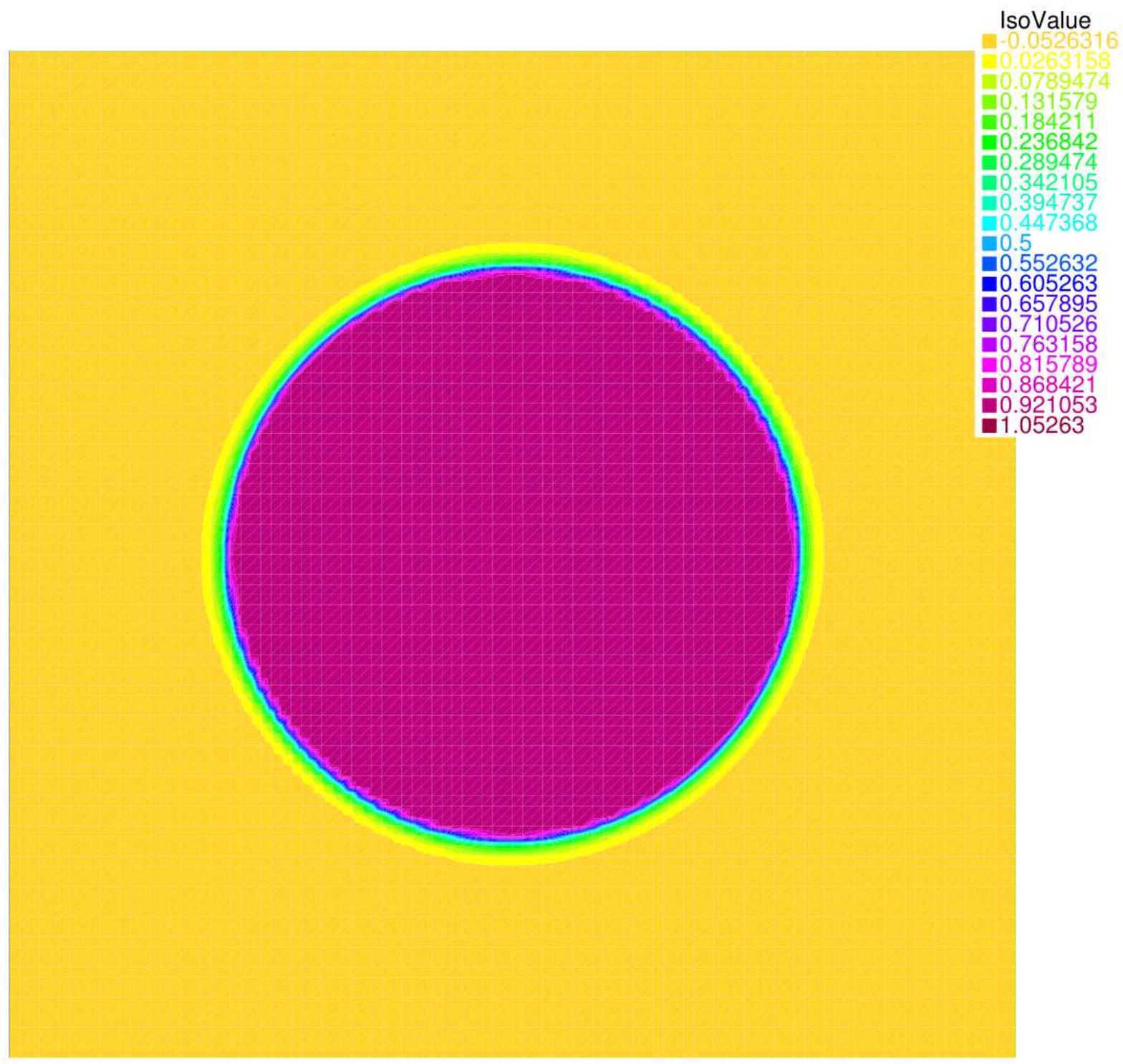}
\includegraphics[scale=0.12]{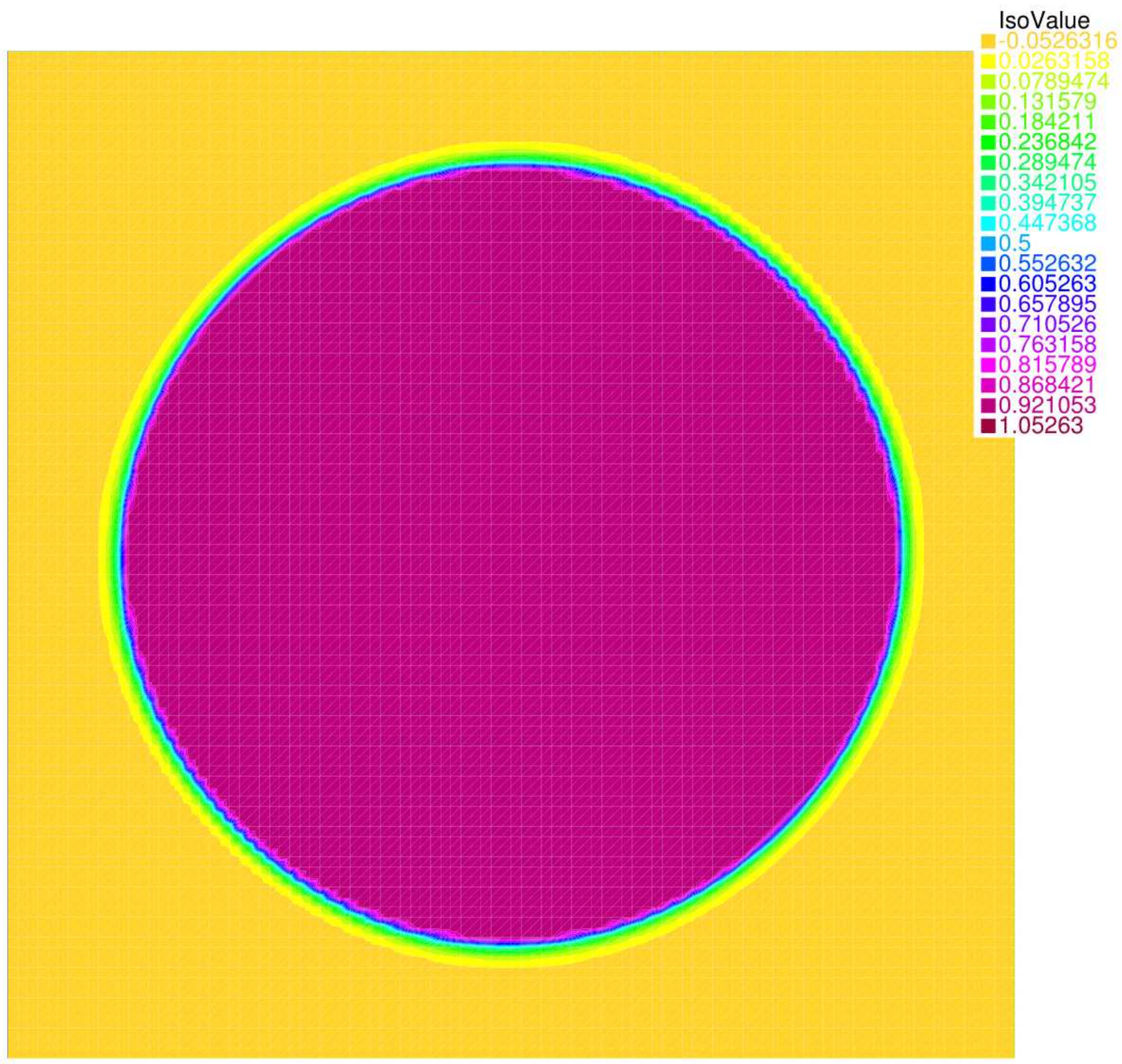}
\includegraphics[scale=0.12]{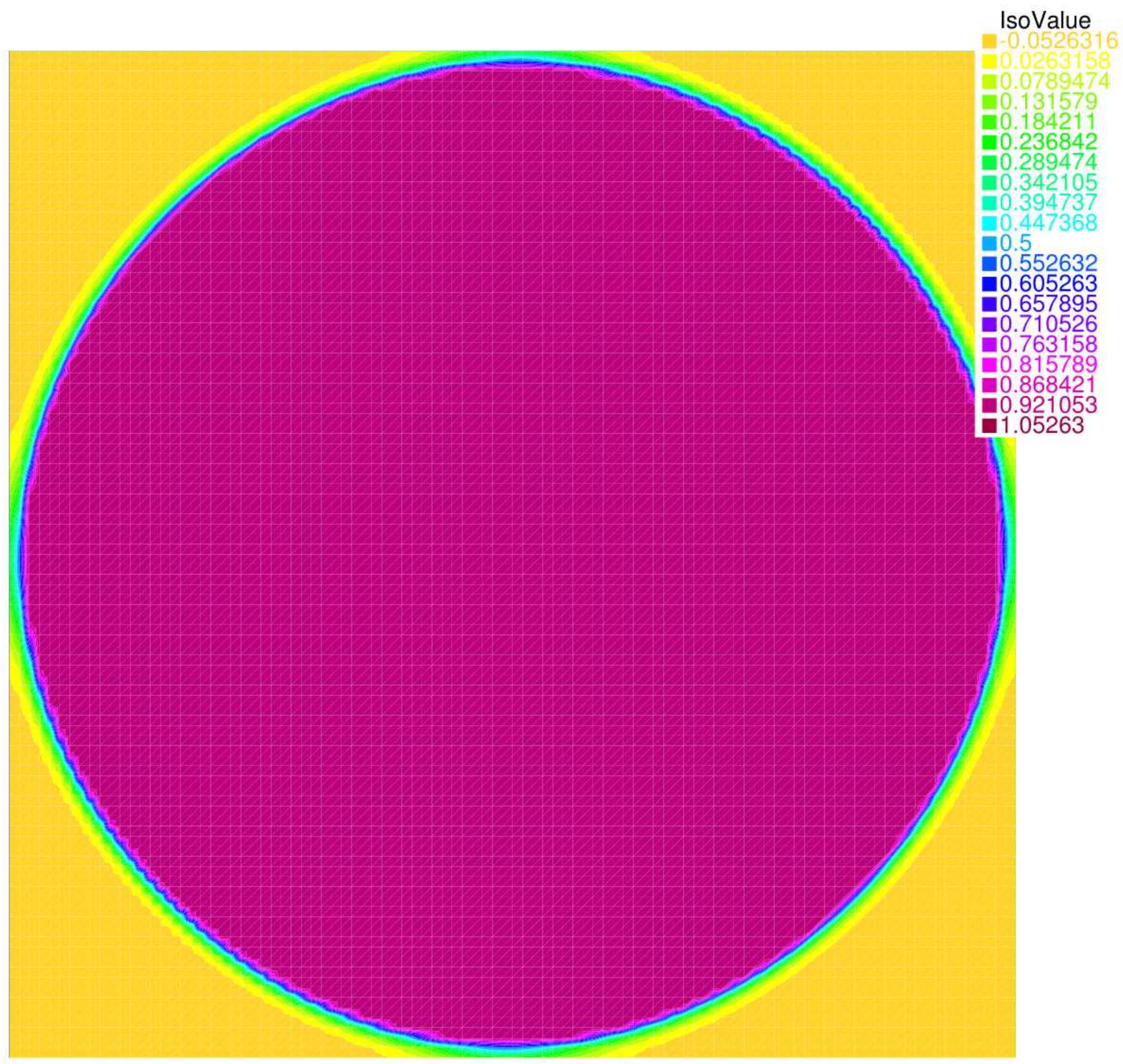}
\\
\includegraphics[scale=0.12]{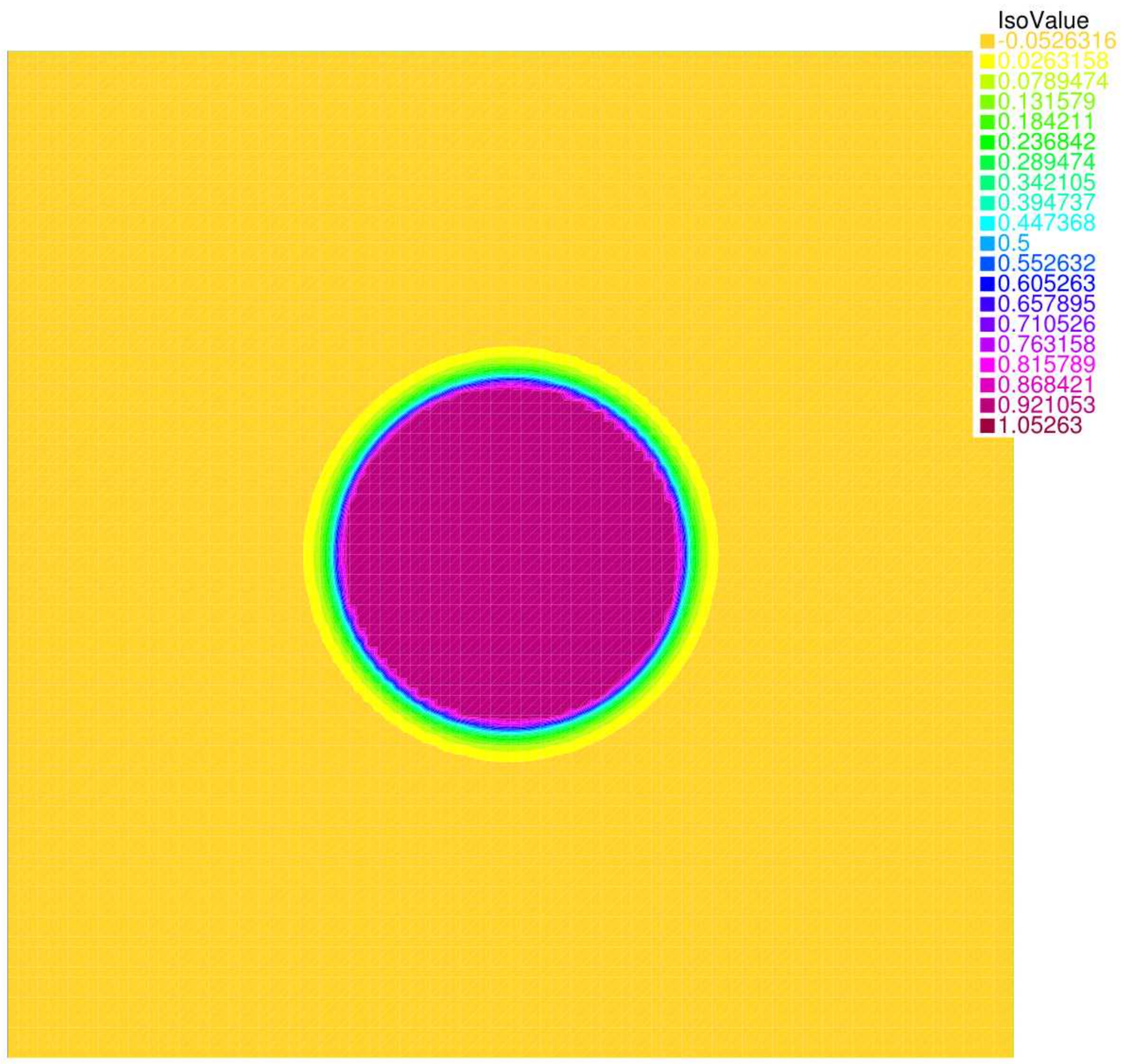}
\includegraphics[scale=0.12]{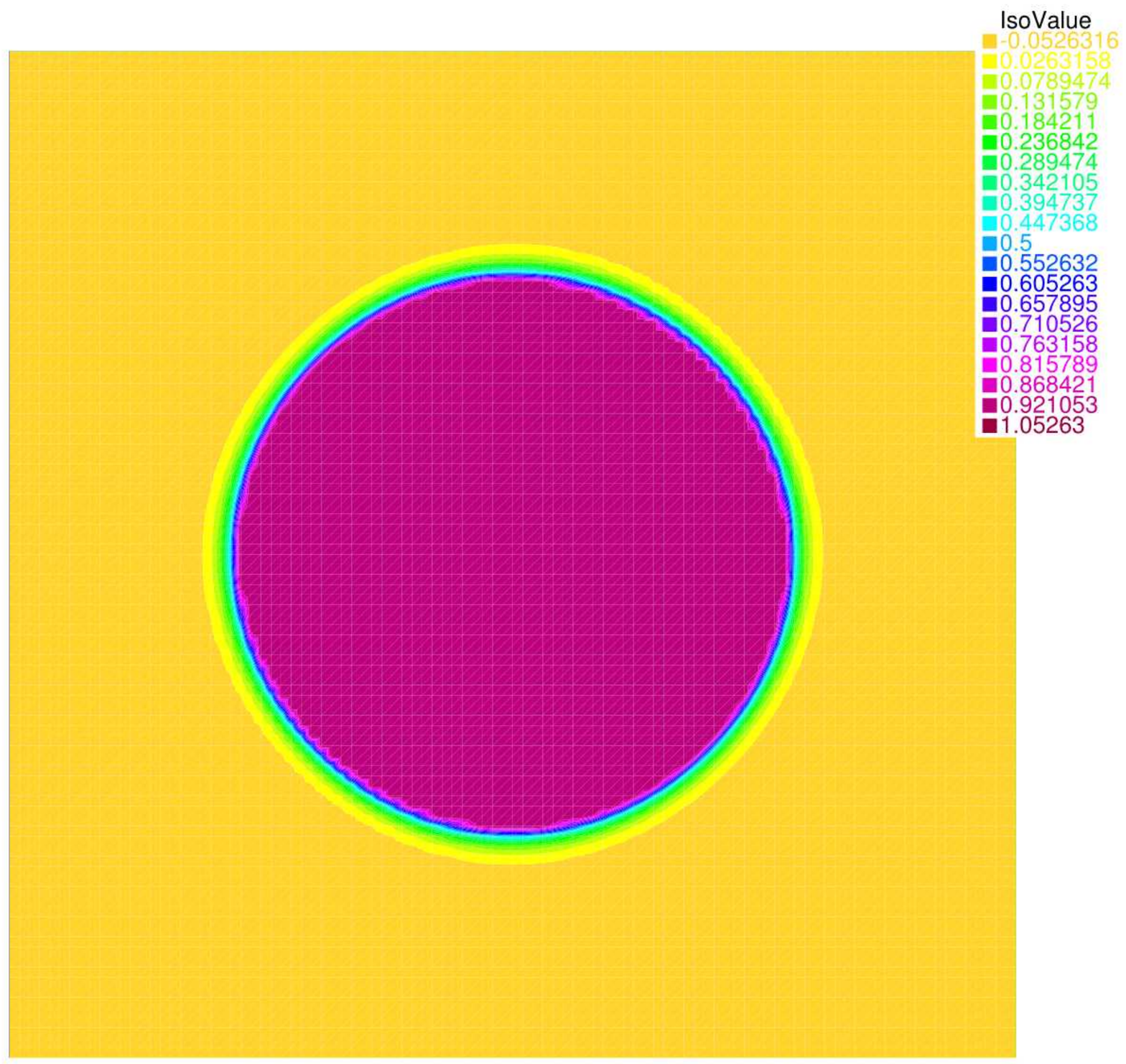}
\includegraphics[scale=0.12]{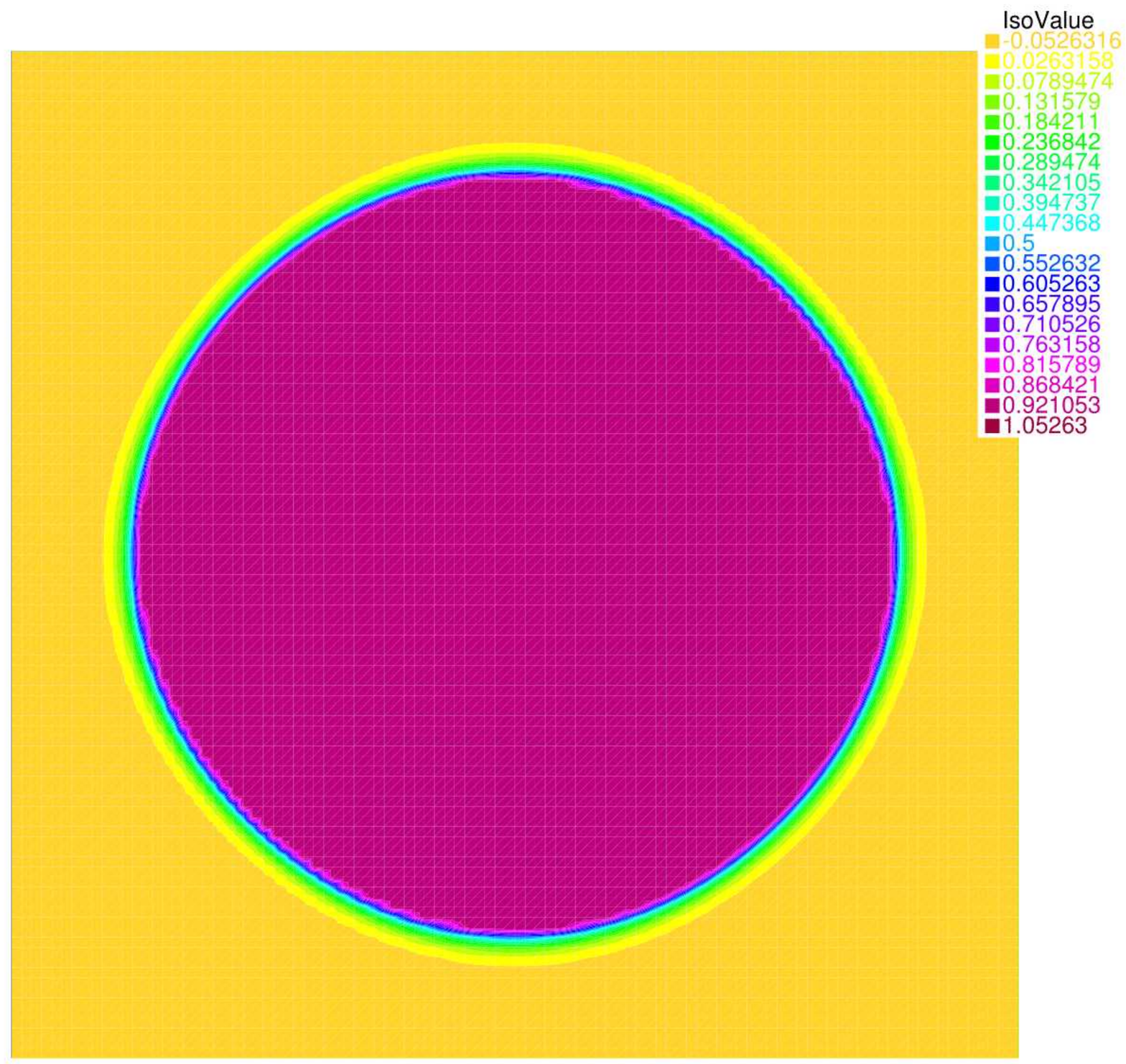}
\includegraphics[scale=0.12]{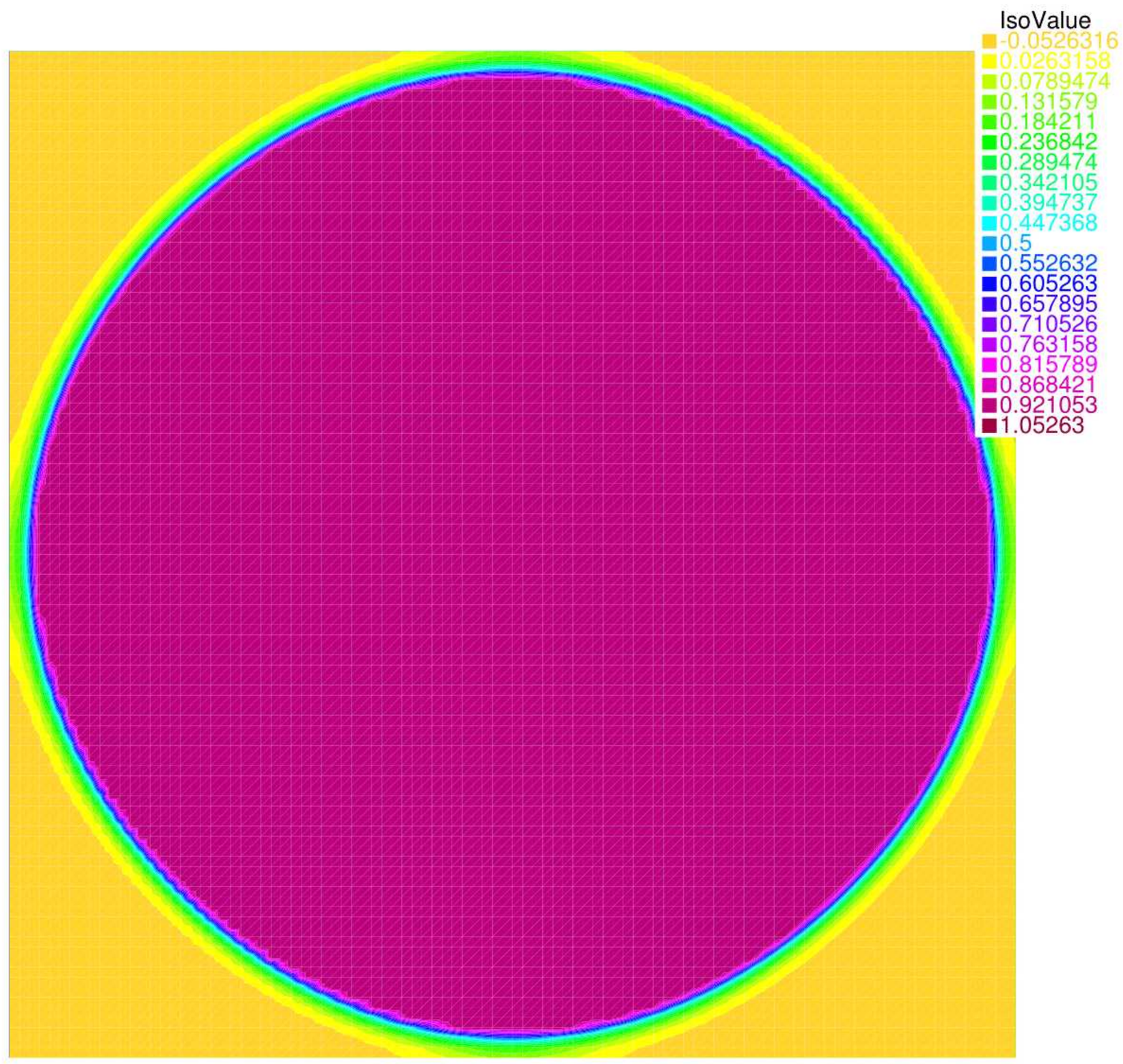}
\caption{Comparison of the density at times $t = 0.1, 0.2, 0.3, 0.4$ for different $k=10$ (top) and $1000$ (bottom).} 
\label{fig:k}
\end{figure}

\subsubsection{Analysis of the effect of $P_{\rm max}$} Let us take $P_{\rm max}=10$ and $30$. Figure \ref{fig:P} shows that the dynamics is sensitive to the different values for the homeostatic pressure. We highlight that, for $P_{\rm max}=30$, the evolution of the interphase is faster than the one for $P_{\rm max} =10$. Moreover, the shape of the interphase seems different as depicted in Figure \ref{fig:P} for times $t=0.1$, $0.2$, $0.3$ and $0.4$.

\begin{figure}[h!]
\centering
\includegraphics[scale=0.12]{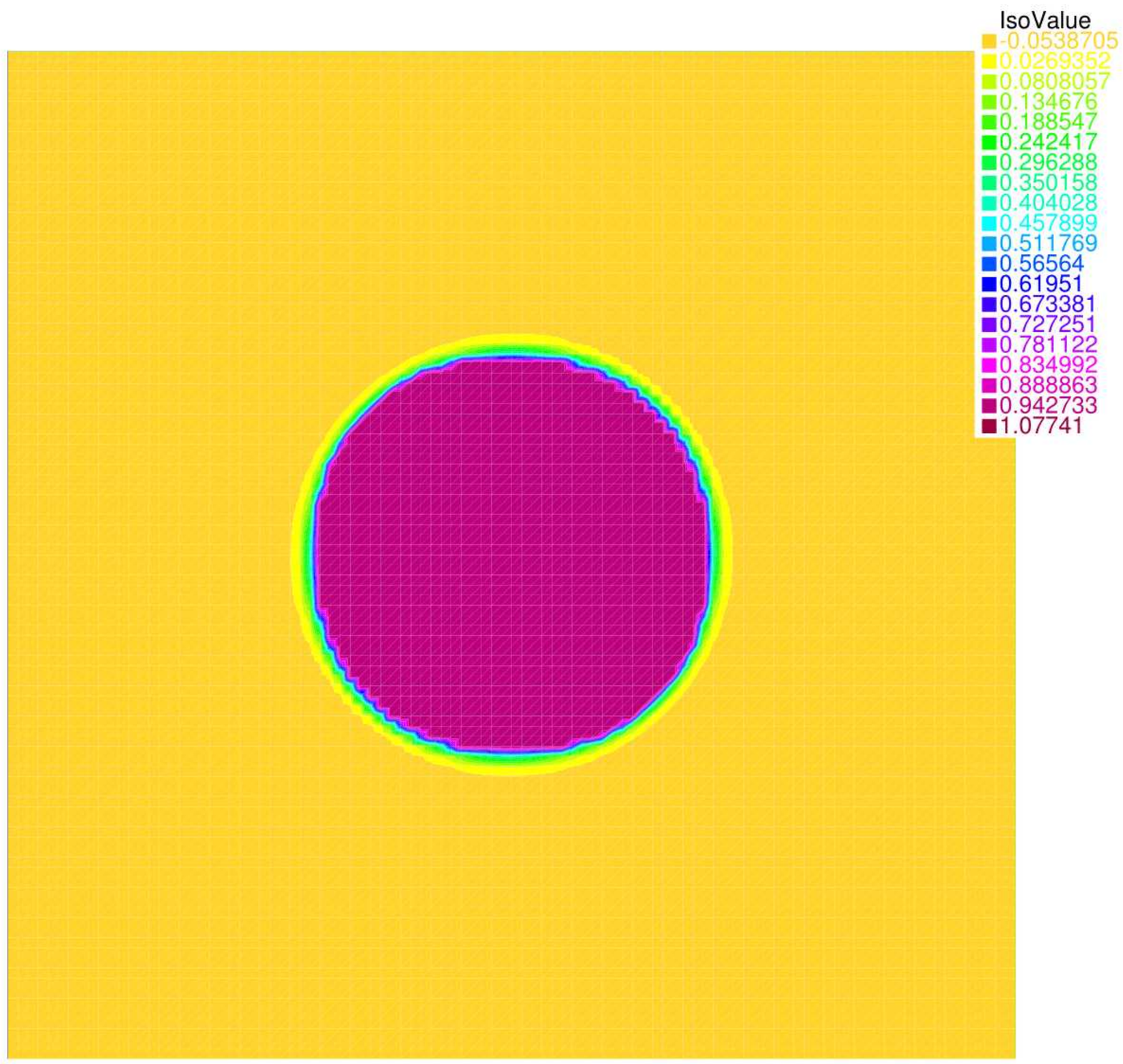}
\includegraphics[scale=0.12]{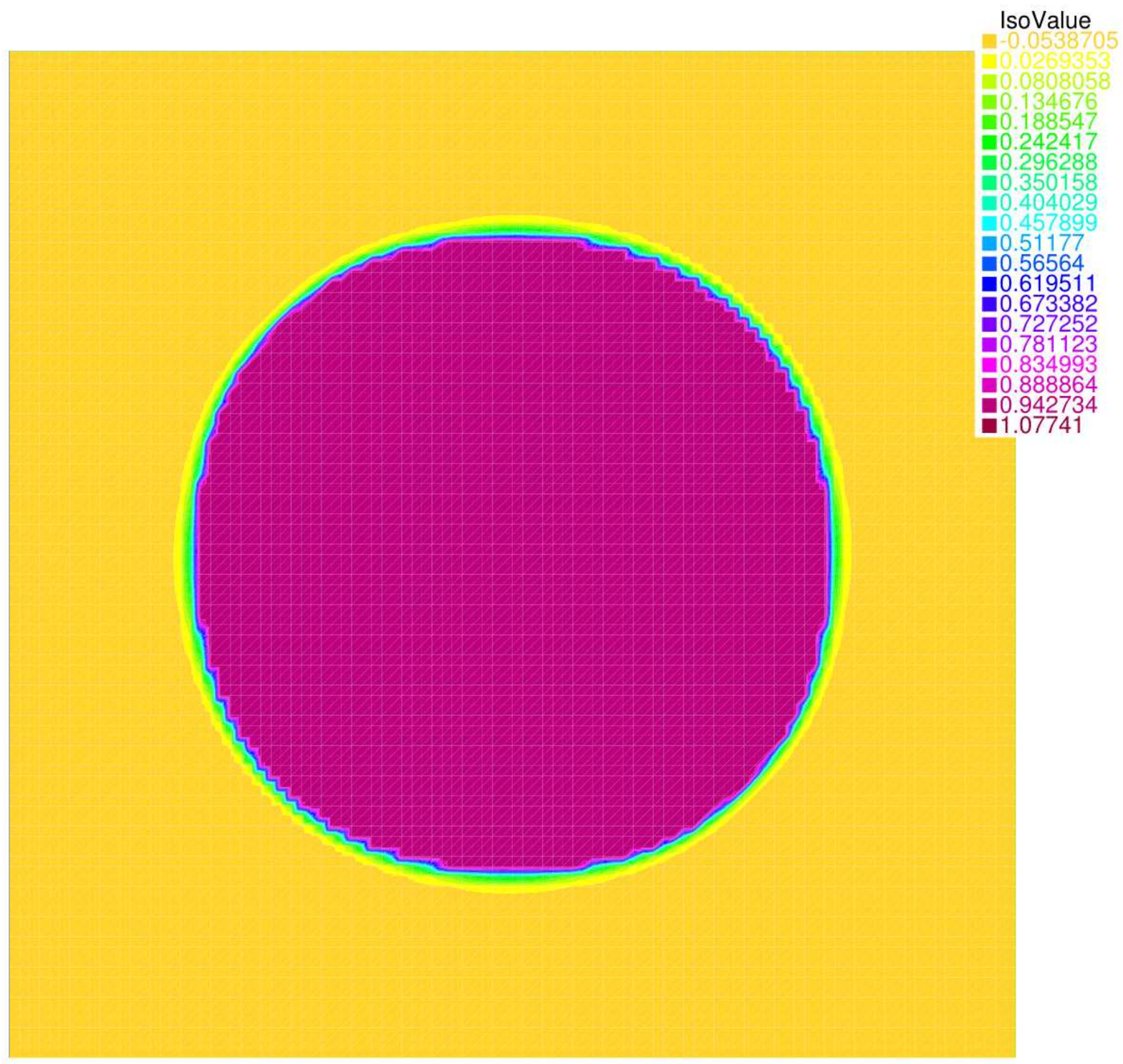}
\includegraphics[scale=0.12]{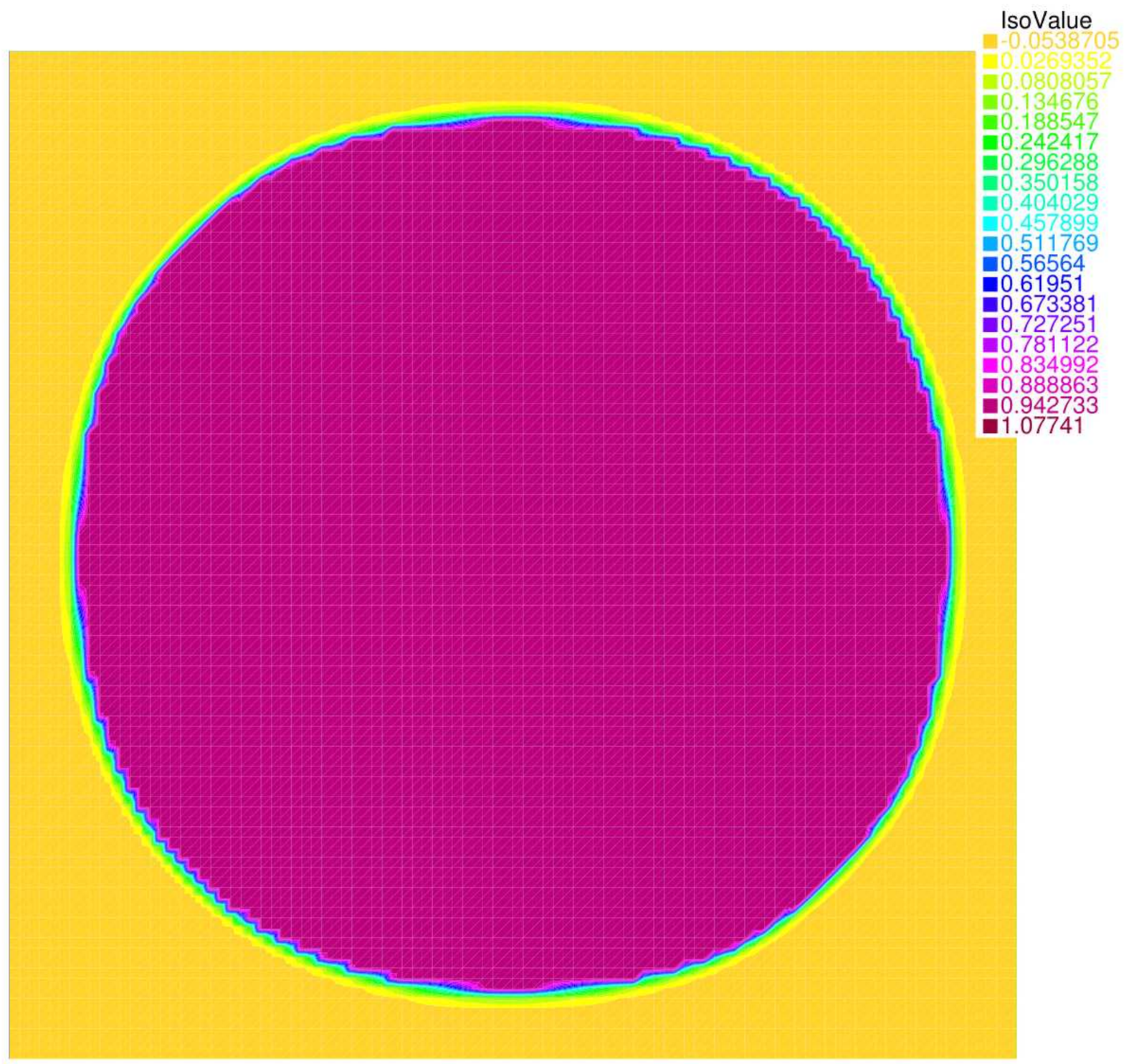}
\includegraphics[scale=0.12]{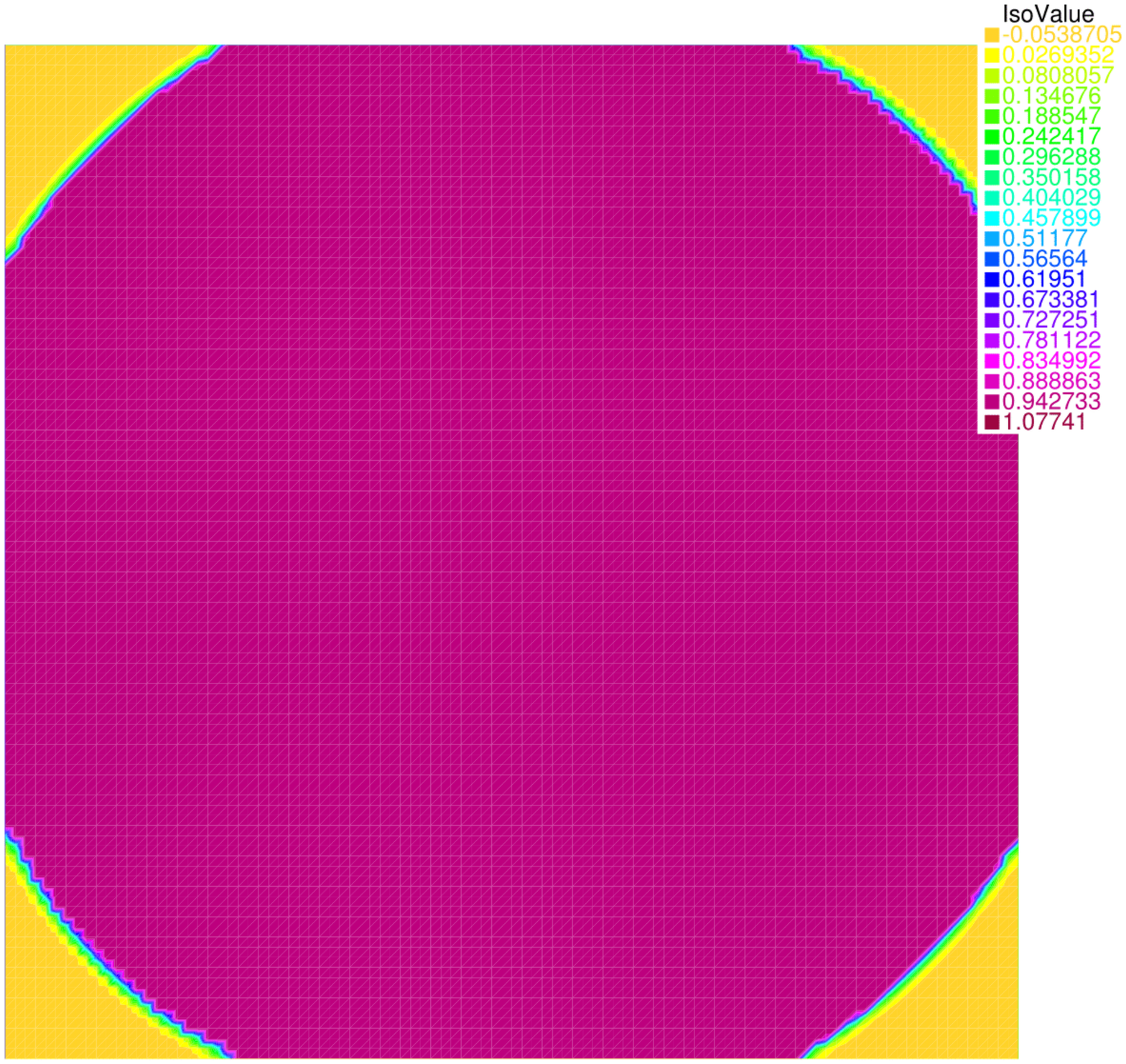}
\\
\includegraphics[scale=0.12]{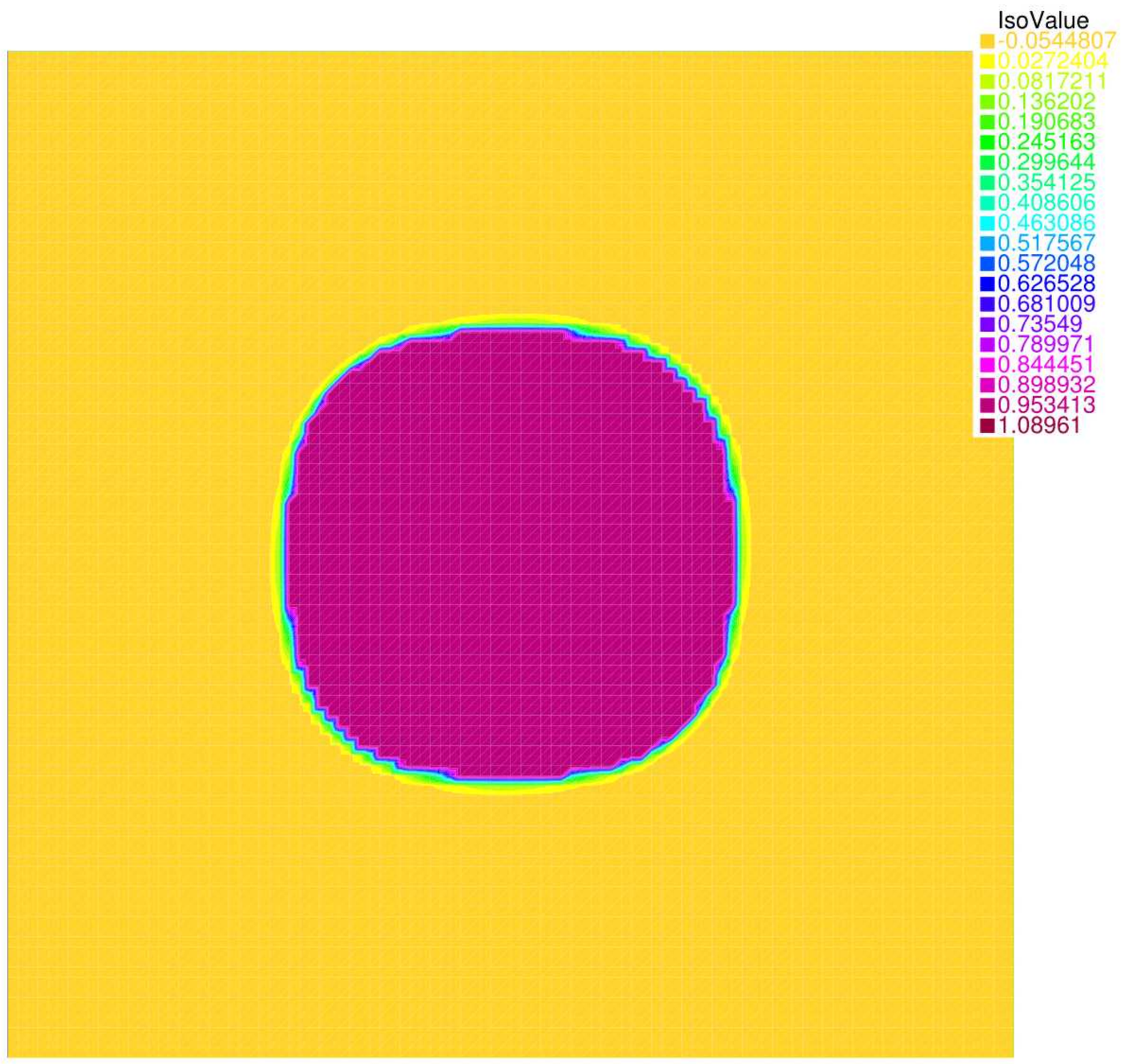}
\includegraphics[scale=0.12]{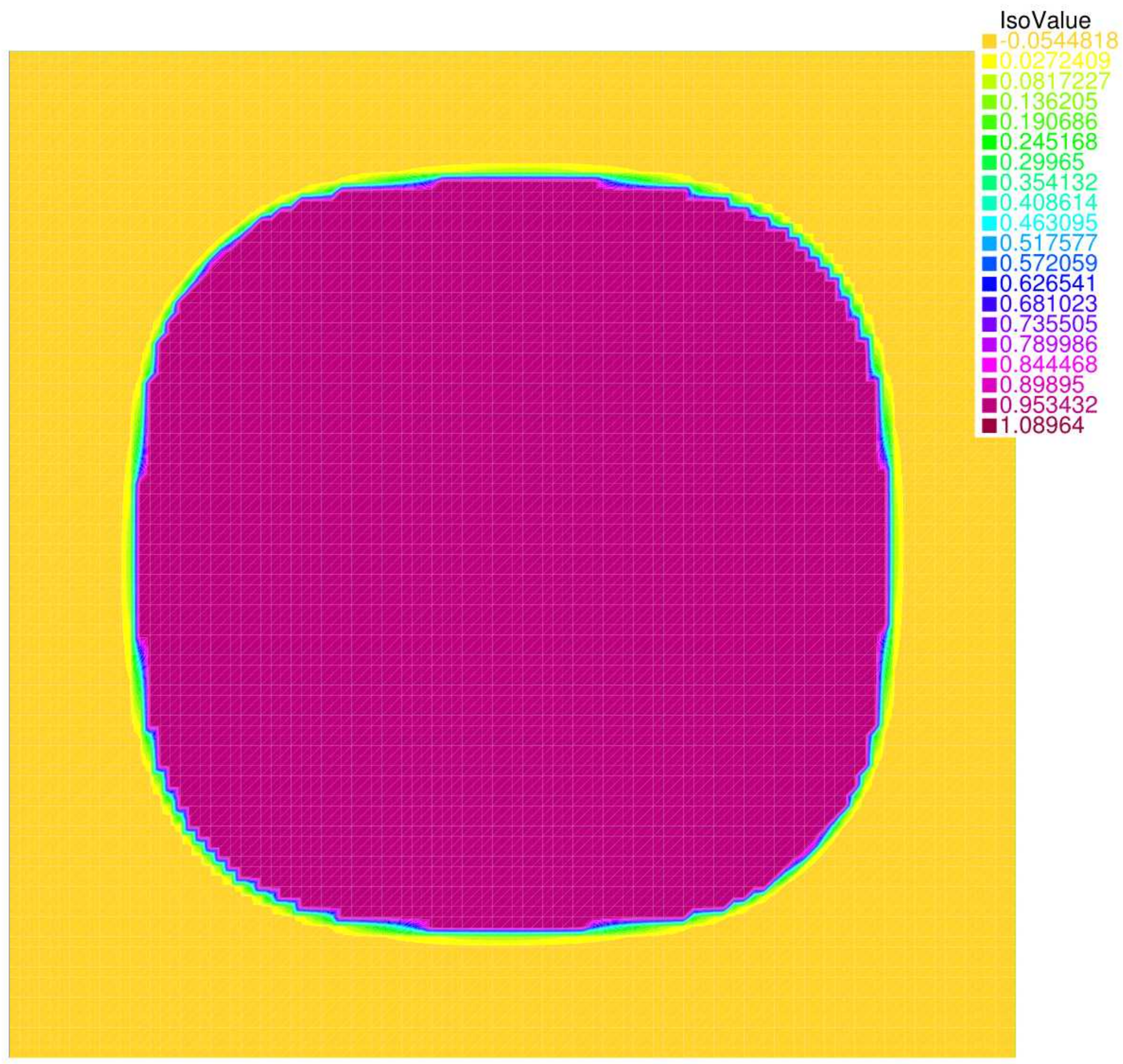}
\includegraphics[scale=0.12]{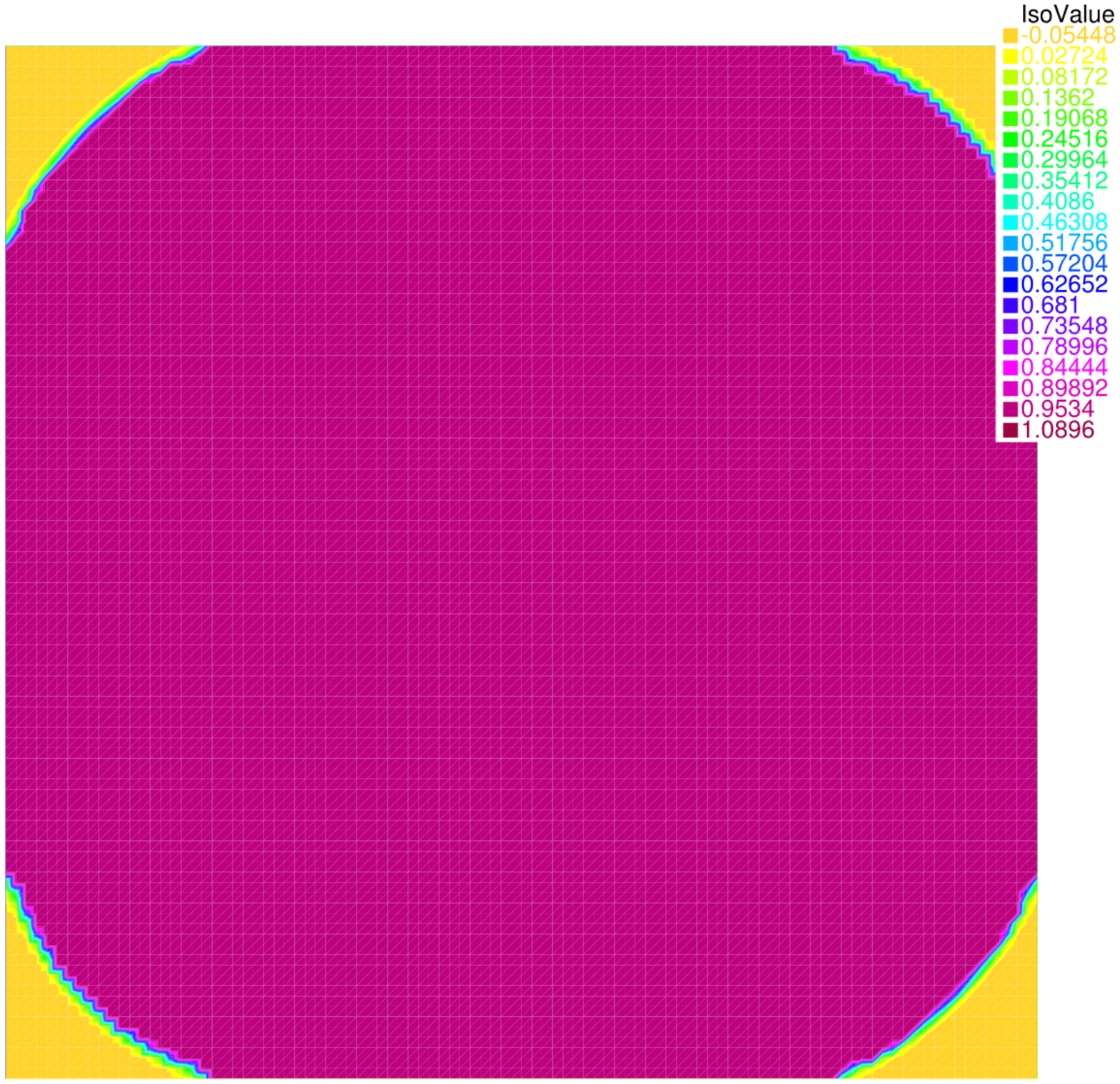}
\includegraphics[scale=0.12]{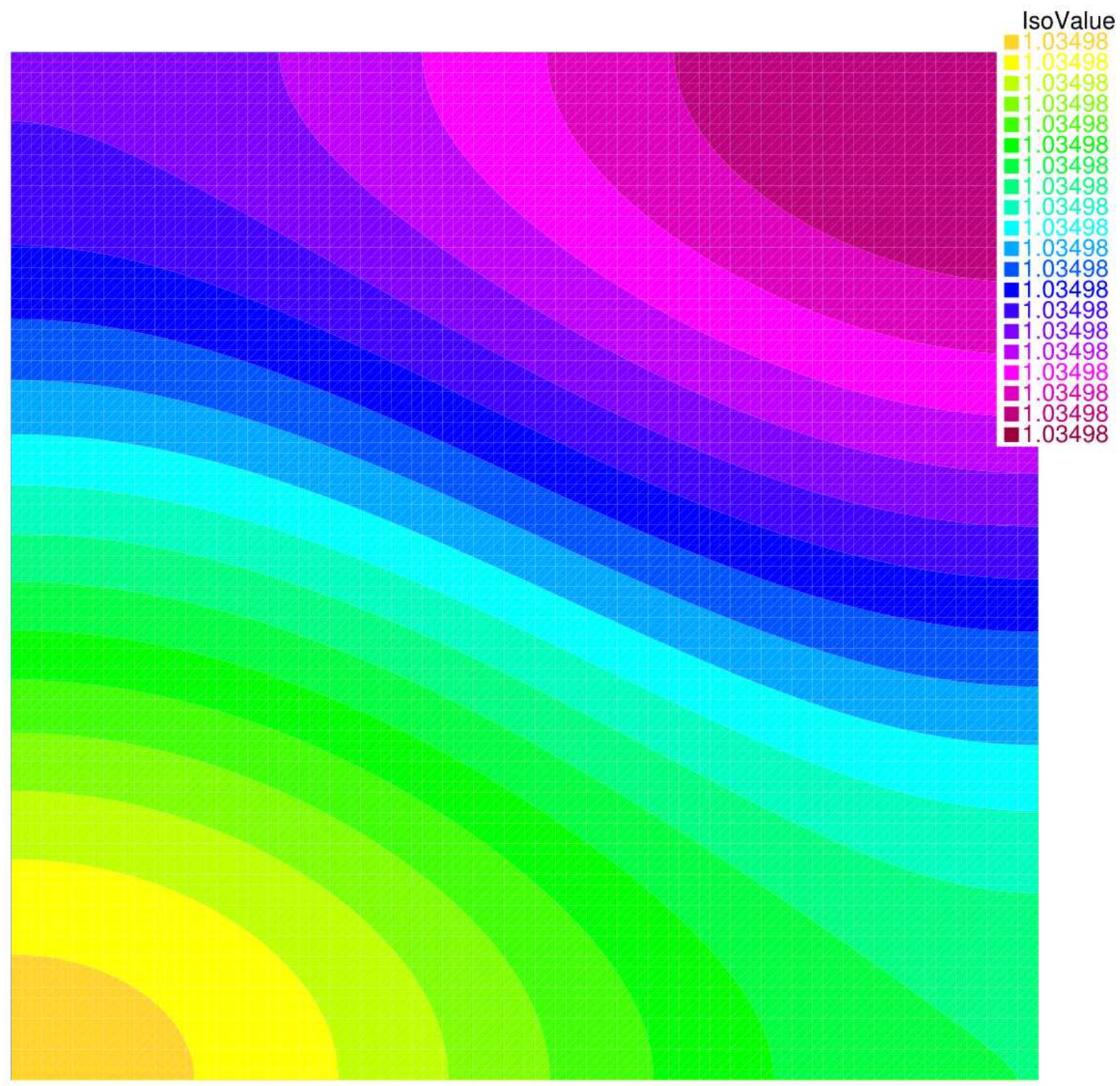}
\caption{Comparison of the density at times $t = 0.1, 0.2, 0.3, 0.4$ for different $P_{\rm max}=10$ (top) and $30$ (bottom).} 
\label{fig:P}
\end{figure}


\begin{thebibliography}{99}
%
\bibitem{Becker_Feng_Prohl_2008}{ \sc Becker, R.; Feng, X.; Prohl, A.} {\it Finite element approximations of the Ericksen-Leslie model for nematic liquid crystal flow.} SIAM J. Numer. Anal. 46 (2008), no. 4, 1704--1731.
%
\bibitem{Bertoluzza_1999}{\sc Bertoluzza, S}
{\it The discrete commutator property of approximation spaces.} 
C. R. Acad. Sci. Paris Sér. I Math. 329 (1999), no. 12, 1097--1102.
 %
\bibitem{Betteridge_Owen_Byre_Alarcon_Maini_2006} {\sc Betteridge, R.; Owen, M. R.; Byrne, H. M.; Alarcón, T.; Maini, P. K.}{\it The impact of cell crowding and active cell movement on vascular tumour growth.} Netw. Heterog. Media 1 (2006), no. 4, 515?535.
 %
\bibitem{Brenner_Scott_2008}{\sc Brenner, S. C.;  Scott, L. R.}, 
{\it The mathematical theory of finite element methods}, 
Third edition. Texts in Applied Mathematics, 15. Springer, New York, 2008.
%
\bibitem{Bru_Albertos_Subiza_Asenjo_Broe_2003}{\sc Brú, A.; Albertos, S.; Subiza, J. L.; Asenjo, J. A.; Broe, I.} {\it The universal dynamics of tumor growth.} Biophys. J. 85 (2003), no. 5, 2948--2961.
%
\bibitem{Byrne_Drasdo_2009}
{\sc Byrne, H. M.; Drasdo, D.}, 
{\it Individual-based and continuum models of growing cell populations: a comparison.}, 
J. Math. Biol. 58 (2009), no. 4-5, 657--687.
%
\bibitem{Ciarlet_Raviart_1973}
{\sc Ciarlet, P.G.; Raviart, P.-A.} 
{\it Maximum principle and uniform convergence for the finite element method}, Comput. Methods Appl. Mech. Engrg. 2 (1973) 17--31.
%
\bibitem{Drasdo_Hoehme_2012}{\sc Drasdo, D.; Hoehme, S.} {\it Modeling the impact of granular embedding media, and pulling versus pushing cells on growing cell clones.} New J. Phys. 14 (2012) 055025.
%
\bibitem{Ern_Guermond_2004}
{\sc Ern, A; Guermond, J.-L.}, 
{\it Theory and practice of finite elements},  
Applied Mathematical Sciences, 159. Springer-Verlag, New York, 2004.
%
\bibitem{Girault_Lions_2001}{\sc Girault, V.; Lions, J.-L.}, {\it Two-grid finite-element schemes for the transient Navier-Stokes problem}. M2AN Math. Model. Numer. Anal. 35 (2001), no. 5, 945--980.
%
\bibitem{Perthame_Quiros_Tang_Vauchelet_2014} {\sc Perthame, B.; Quirós, F.; Tang, M.; Vauchelet, N.}, 
{\it  Derivation of a Hele-Shaw type system from a cell model with active motion}, 
Interfaces Free Bound. 16 (2014), no. 4, 489--508.
%
\bibitem{Pertham_Quiros_Vazquez_2014}{\sc Perthame, B.; Quirós, F.; Vázquez, J.\ L.}, {\it The Hele-Shaw asymptotics for mechanical models of tumor growth}, Arch. Ration. Mech. Anal. 212 (2014), no. 1, 93--127.
%
\bibitem{Ranft_Basan_Egeti_Joanny_Prost_Julicher_2010}{\sc J. Ranft, M. Basan, J. Elgeti, J.-F. Joanny, J. Prost and F. Jülicher},{\it Fluidization of tissues
by cell division and apoptosis}, Proceedings of the National Academy of Sciences, 107 (2010), no. 49,
20863--20868.
%
\bibitem{Saut_Lagaert_Colin_2014}
{\sc Saut, O.; Lagaert, J.-B.; Colin, T.; Fathallah-Shaykh, H. M.} 
{\it A multilayer grow-or-go model for GBM: effects of invasive cells and anti-angiogenesis on growth.} Bull. Math. Biol. 76 (2014), no. 9, 2306--2333.
%
\bibitem{Scott_Zhang_1990}
{\sc Scott, L.R.; Zhang, S.}
{\it Finite element interpolation of non-smooth functions satisfying boundary conditions.} 
Math. Comp. 54 (1990) 483--493.
%
\bibitem{Simon_1987} {\sc Simon, J.}{\it Compact sets in the space $L^p(0,T;B)$}. Ann. Mat. Pura Appl. (4) 146 (1987), 65--96.
\end{thebibliography}
\end{document}